\documentclass[UTF8,10pt]{article}

\usepackage[left=3.18cm,right=3.18cm,top=2.54cm,bottom=2.54cm]{geometry}
\usepackage{amsfonts}
\usepackage{amsmath}
\usepackage{amssymb}
\usepackage{pgfplots}
\usepackage{tikz}
\usetikzlibrary{arrows}
\usepackage{titlesec}
\usepackage{bm}
\usepackage{appendix}
\usepackage{tikz-cd}
\usepackage{mathtools}
\usepackage{amsthm}
\usepackage{enumitem}
\usepackage[colorlinks,linkcolor=black,anchorcolor=black,citecolor=black]{hyperref}
\usepackage[numbers]{natbib}
\usepackage{cases}
\usepackage{fancyhdr}
\usepackage[scaled=0.92]{helvet}	
\usepackage{txfonts}

\theoremstyle{definition}
\newtheorem{thm}{Theorem}[section]

\newtheorem{lem}[thm]{Lemma}
\newtheorem{prop}[thm]{Proposition}

\newtheorem{cor}[thm]{Corollary}

\newtheorem*{rmk}{Remark}

\newcommand{\R}{\mathbb{R}}

\newcommand{\N}{\mathbb{N}}

\newcommand{\T}{\mathbb{T}}

\newcommand{\TP}{\overline{\partial}{}}
\newcommand{\TL}{\overline{\Delta}{}}
\newcommand{\tpl}{\overline{\partial}^2\overline{\Delta}}
\newcommand{\curl}{\text{curl }}
\newcommand{\dive}{\text{div }}

\newcommand{\q}{\quad}
\newcommand{\p}{\partial}
\newcommand{\dd}{\mathfrak{D}}
\newcommand{\DD}{\mathcal{D}}

\newcommand{\nab}{\nabla}
\newcommand{\ak}{\tilde{a}}
\newcommand{\Ak}{\tilde{A}}
\newcommand{\ek}{\tilde{\eta}}
\newcommand{\ar}{\mathring{a}}
\newcommand{\ark}{\mathring{\tilde{{a}}}}
\newcommand{\Ark}{\mathring{\tilde{{A}}}}
\newcommand{\Jrk}{\mathring{\tilde{{J}}}}
\newcommand{\er}{\mathring{\eta}}
\newcommand{\qr}{\mathring{q}}
\newcommand{\erk}{\mathring{\tilde{{\eta}}}}
\newcommand{\vr}{\mathring{v}}
\newcommand{\br}{\mathring{b}}
\newcommand{\Jr}{\mathring{J}}
\newcommand{\Jk}{\tilde{J}}
\newcommand{\psir}{\mathring{\psi}}
\newcommand{\park}{\nabla_{\mathring{\tilde{a}}}}

\newcommand{\divr}{\text{div}_{\mathring{\tilde{a}}}}

\newcommand{\curlr}{\text{curl}_{\mathring{\tilde{a}}}}

\newcommand{\pa}{\nabla_a}

\newcommand{\pak}{\nabla_{\tilde{a}}}

\newcommand{\diva}{\text{div}_{\tilde{a}}}

\newcommand{\curla}{\text{curl}_{\tilde{a}}}

\newcommand{\lkk}{\Lambda_{\kk}}

\newcommand{\lapak}{\Delta_{\tilde{a}}}
\newcommand{\lapark}{\Delta_{\mathring{\tilde{a}}}}

\newcommand{\di}{\text{div}\,}

\newcommand{\cp}{\overline{\partial}{}}
\newcommand{\dx}{\,dx}
\newcommand{\dy}{\,dy}
\newcommand{\dz}{\,dz}
\newcommand{\dt}{\,dt}
\newcommand{\dS}{\,dS}

\newcommand{\kk}{\kappa}

\newcommand{\EE}{\mathcal{E}}

\newcommand{\XX}{\mathbf{X}}
\newcommand{\VV}{\mathbf{V}}
\newcommand{\QQ}{\mathbf{Q}}
\newcommand{\GG}{\mathbf{G}}
\newcommand{\FF}{\mathbf{F}}
\newcommand{\VVr}{\mathring{\mathbf{V}}}
\newcommand{\QQr}{\mathring{\mathbf{Q}}}

\newcommand{\FFr}{\mathring{\mathbf{F}}}

\newcommand{\w}{\widetilde}
\newcommand{\PP}{\mathcal{P}}

\newcommand{\wt}{\rho'(p)}

\newcommand{\idt}{\int_{\mathcal{D}_t}}
\newcommand{\ipdt}{\int_{\partial\mathcal{D}_t}}
\newcommand{\io}{\int_{\Omega}}

\newcommand{\ig}{\int_{\Gamma}}
\numberwithin{equation}{section}
\usepackage{xcolor}

\begin{document}

\title{\textbf{Local Well-posedness of the Free-Boundary Problem in Compressible Resistive Magnetohydrodynamics}}
\author{{\sc Junyan Zhang}\\{\footnotesize Department of Mathematics, National University of Singapore,}\\{\footnotesize 10 Lower Kent Ridge Road, Singapore 119076.}\\{\footnotesize Email: \texttt{zhangjy@nus.edu.sg}}
 }
\date{\today}
\maketitle

\setcounter{tocdepth}{2}

\begin{abstract}
 We prove the local well-posedness in Sobolev spaces of the free-boundary problem for compressible inviscid resistive isentropic MHD system under the Rayleigh-Taylor physical sign condition, which describes the motion of a free-boundary compressible plasma in an electro-magnetic field with magnetic diffusion. We use Lagrangian coordinates and apply the tangential smoothing method introduced by Coutand-Shkoller \cite{coutand2007LWP,coutand2012LWP} to construct the approximation solutions. One of the key observations is that the Christodoulou-Lindblad type elliptic estimate \cite{christodoulou2000motion} together with magnetic diffusion not only gives the common control of magnetic field and fluid pressure directly, but also controls the Lorentz force which is a higher order term in the energy functional. 
\end{abstract}

\tableofcontents

\section{Introduction}
In this paper, we consider the 3D magnetohydrodynamics (MHD) equations with magnetic resistivity
\begin{equation}\label{CMHD}
\begin{cases}
\rho D_t u=B\cdot\nabla B-\nabla P,~~P:=p+\frac{1}{2}|B|^2~~~& \text{in}~\DD; \\
D_t\rho+\rho\dive u=0~~~&\text{in}~\DD; \\
D_t B-\lambda\Delta B=B\cdot\nabla u-B\dive u,~~~&\text{in}~\DD; \\
\dive B=0~~~&\text{in}~\DD,
\end{cases}
\end{equation}
describing the motion of a compressible conducting fluid in an electro-magnetic field with magnetic diffusion, $\lambda>0$ is the magnetic diffusivity constant. $\DD:={\bigcup}_{0\leq t\leq T}\{t\}\times \DD_t$ and $\DD_t\subset \R^3$ is the domain occupied by the conducting fluid whose boundary $\p\DD_t$ moves with the velocity of the fluid. $\nabla:=(\p_{x_1},\p_{x_2},\p_{x_3})$ is the standard spatial derivative and $\dive X:=\p_{x_\alpha} X^\alpha$ is the standard divergence for any vector field $X$. Throughout this paper, $X^\alpha=\delta^{\alpha\beta}X_\beta$ for any vector field $X$, i.e., we use Einstein summation convention. The fluid velocity $u=(u_1,u_2,u_3)$, the magnetic field $B=(B_1,B_2,B_3)$, the fluid density $\rho$, the pressure $p$ and the domain $\DD\subseteq[0,T]\times\R^3$ are to be determined. $D_t:=\p_t+u\cdot\nabla$ is the material derivative. Here we consider the isentropic case, and thus the fluid pressure $p=p(\rho)$ should be a given strictly increasing smooth function of the density $\rho$. 

\subsection{Initial and boundary conditions and constraints}

We consider the Cauchy problem of \eqref{CMHD}: Given a simply-connected bounded domain $\DD_0\subset \R^3$ and the initial data $u_0$, $\rho_0$ and $B_0$ satisfying the constraints $\dive B_0=0$ in $\DD_0$ and $B=\mathbf{0}$ on $\p\DD_0$, we want to find a set $\DD$, the vector field $u$, the magnetic field $B$, and the density $\rho$ solving \eqref{CMHD} satisfying the initial conditions:
\begin{equation}\label{MHDI} 
\DD_0=\{x: (0,x)\in \DD\},\q (u,B,\rho)=(u_0, B_0,\rho_0),\q \text{in}\,\,\{0\}\times \DD_0.
\end{equation}

Thus we need to introduce the initial and boundary conditions. First, we require the following boundary conditions on the free boundary $\p\DD={\cup}_{0\leq t\leq T}\{t\}\times \p\DD_t$:
\begin{equation}\label{MHDB}
\begin{cases}
D_t|_{\partial\DD}\in T(\partial\DD) \\
p=0&\text{on}~\partial\DD, \\
B=\mathbf{0}&\text{on}~\partial\DD.
\end{cases}
\end{equation}

\subsubsection*{Illustration on the boundary conditions}

The first condition in \eqref{MHDB} means that the boundary moves with the velocity of the fluid. It can be equivalently rewritten as: $V(\p\DD_t)=u\cdot n$ on $\p\DD$ or $(1,u)$ is tangent to $\p\DD$. The second condition in \eqref{MHDB} means that outside the fluid region $\DD_t$ is the vacuum. The third boundary condition requires that the magnetic field vanishes on the boundary. 

\begin{rmk} 
~
\begin{enumerate}
\item In \eqref{CMHD}, the divergence-free condition of $B$ is just a constraint for initial data and thus the system is not over-determined. Indeed, the original version of the third equation in \eqref{CMHD} is $\p_t B=-\nab \times E$ where $E:=-u\times B+\lambda j$ is the electric field and $j:=\nab\times B$ is the current density. Taking divergence in this equation and using the continuity equation will give us $D_t(\rho^{-1}\dive B)=0$ and thus $\dive B=0$ is preserved if it initially holds. Throughout this manuscript, we will always use the heat equation of $B$ to do estimates.
\item For ideal MHD ($\lambda=0$), the boundary condition $B=\mathbf{0}$ should NOT be an imposed conditions for the system, otherwise the system is over-determined. Instead, this is a direct result of propagation of the constraints on initial data. See also Hao-Luo \cite{hao2014priori} for details. However, for resistive MHD, the equation of $B$ is a parabolic equation which has to be given a boundary condition as opposed to the ideal case. Such a condition will not make the system be over-determined. See also Wang-Xin \cite{wangxinMHD2} in the case of incompressible inviscid resistive MHD surface waves.
\end{enumerate}
\end{rmk}

\subsubsection*{Related physical background}

 The free-boundary problem considered in this manuscript originates from the plasma-vacuum free-interface model, which is an important theoretic model both in laboratory and in astro-physical magnetohydrodynamics: The plasma is confined in a vacuum in which there is another magnetic field $\hat{B}$, and there is a free interface $\Gamma(t)$, moving with the motion of plasma, between the plasma region $\Omega_+(t)$ and the vacuum region $\Omega_{-}(t)$. This model requires that \eqref{CMHD} holds in the plasma region $\Omega_+(t)$ and the pre-Maxwell system holds in vacuum $\Omega_{-}(t)$:
\begin{equation}\label{outsideB}
\curl\hat{B}=\mathbf{0},~~~\dive\hat{B}=0.
\end{equation}  
On the interface $\Gamma(t)$, it is required that there is no jump for the pressure or the \textbf{normal components of magnetic fields}:
\begin{equation}\label{interface}
B\cdot n=\hat{B}\cdot n,~~~P:=p+\frac12|B|^2=\frac12|\hat{B}|^2
\end{equation}  where $n$ is the exterior unit normal to $\Gamma(t)$. Note that for ideal MHD ($\lambda=0$) \eqref{interface} should also be a constraint on initial data which propagates instead of an imposed boundary condition. For more details, we refer readers to Chapter 4, 6 in \cite{MHDphy}.

Hence, the case considered in this manuscript is that the vacuum magnetic field $\hat{B}$ vanishes plus the imposed condition $B=\mathbf{0}$ on the boundary which is reasonable for resistive MHD. It characterizes the motion of an isolated plasma in an electro-magnetic field with magnetic diffusion.

\subsubsection*{Energy conservation/dissapation}

The free-boundary compressible resistive MHD system together with boundary conditions \eqref{MHDB} satisfies the following energy conservation/dissapation:  Set $Q(\rho)=\int_1^{\rho} p(R)/R^2 dR$, then we use \eqref{CMHD} to get
\begin{equation}\label{conserve1}
\begin{aligned}
&~~~~\frac{d}{dt}\left(\frac{1}{2}\idt\rho |u|^2\dx+\frac{1}{2}\idt|B|^2 \dx +\idt \rho Q(\rho)\dx\right)  \\
&=\idt \rho u\cdot D_tu\dx+\idt B\cdot D_t B\dx+\idt \rho D_t Q(\rho)\dx+\frac{1}{2}\idt \rho D_t(1/\rho)|B|^2\dx\\
&=\idt u\cdot (B\cdot\nabla B)\dx-\idt u\cdot\nabla P \dx+\idt B\cdot(B\cdot\nabla u)\dx-\idt |B|^2\dive u\dx \\
&~~~~+\idt p(\rho) \frac{D_t \rho}{\rho} \dx -\frac{1}{2}\idt \frac{D_t \rho}{\rho}|B|^2\dx.
\end{aligned}
\end{equation}

Integrating by part in the first term in the last equality, this term will cancel with $\idt B\cdot(B\cdot\p u)\dx$ because the boundary term and the other interior term vanishes due to $B\cdot n=0$ and $\dive B=0$ respectively. Here $n$ is the exterior unit normal vector to $\p\DD_t$. Also we integrate by parts in the second term and then use the continuity equation to get
\begin{equation}\label{conserve2}
\begin{aligned}
-\idt u\cdot\nabla P\dx&=\idt P\dive u\dx-\underbrace{\ipdt (u\cdot N)P dS}_{=0}=-\idt p\frac{D_t \rho}{\rho} \dx +\frac{1}{2}\idt |B|^2\dive u\dx\\
&=-\idt p\frac{D_t \rho}{\rho} \dx +\idt |B|^2\dive u\dx -\frac{1}{2}\idt |B|^2\dive u\dx \\
&=-\idt p\frac{D_t \rho}{\rho} \dx +\idt |B|^2\dive u\dx+\frac{1}{2}\idt \frac{D_t \rho}{\rho}|B|^2\dx.
\end{aligned}
\end{equation}

Summing up \eqref{conserve1} and \eqref{conserve2}, one can get the energy conservation for the free-boundary ideal compressible MHD:
\begin{equation}
\frac{d}{dt}\left(\frac{1}{2}\idt\rho |u|^2\dx+\frac{1}{2}\idt|B|^2 \dx +\idt \rho Q(\rho)\dx\right)  =0.
\end{equation} Also one can see this energy conservation coincides with the analogue for the free-boundary compressible Euler's equations in Lindblad-Luo \cite{lindblad2018priori}.

For the resistive compressible MHD as stated in \eqref{CMHD}, there will be one extra dissipation term, and one can integrate by part to get the energy dissipation.
\begin{equation}
\begin{aligned}
&~~~~\frac{d}{dt}\left(\frac{1}{2}\idt\rho |u|^2\dx+\frac{1}{2}\idt|B|^2 \dx +\idt \rho Q(\rho)\dx\right) \\
&=0+\lambda\idt B\cdot\Delta B \dx
=-\lambda\idt |\nabla B|^2 \dx<0.
\end{aligned}
\end{equation}

\subsubsection*{Equation of state: Isentropic liquid}

Since $p=p(\rho)$ and $p|_{\p\DD}=0$, we know the fluid density also has to be a constant $\bar{\rho_0}\geq 0$ on the boundary. We assume $\bar{\rho_0}>0$, corresponding to the case of liquid as opposed to gas. Hence
\begin{equation}\label{liquid}
p(\bar{\rho_0})=0,~~p'(\rho)>0,~~\text{for }\rho\geq \bar{\rho_0}.
\end{equation}

\subsubsection*{Physical constraints}

Next we impose the following natural conditions on $\rho'(p)$ for some fixed constant $A_0>1$. See also Luo \cite{luo2018cww}.
\begin{equation}\label{weight}
A_0^{-1}\leq|\rho^{(m)}(p)|\leq A_0,\text{ and }~A_0^{-1}|\wt|^m\leq|\rho^{(m)}(p)|\leq A_0|\wt|^m,~~for~1\leq m\leq 5.
\end{equation}

We also need to impose the Rayleigh-Taylor sign condition
\begin{equation}\label{sign}
-\nabla_n P\geq c_0>0~~on~\p\DD_t,
\end{equation} where , $c_0>0$ is a constant and $P:=p+\frac{1}{2}|B|^2$ is the total pressure. When $B=0$, Ebin \cite{ebin1987equations} proved the ill-posedness of the free-boundary incompressible Euler equations without Rayleigh-Taylor sign condition. For the free-boundary MHD equations, \eqref{sign} is also necessary: Hao-Luo \cite{hao2018ill} proved that the free-boundary problem of 2D incompressible MHD equations is ill-posed when \eqref{sign} fails. We also note that \eqref{sign} is only required for initial data and it propagates in a short time interval because one can prove it is $C_{t,x}^{0,1/4}$ H\"older continuous by using Morrey's embedding. See Luo-Zhang \cite{luozhang2019MHD2.5} and Zhang \cite{ZhangCRMHD1}.

\subsubsection*{Compatibility conditions on initial data}

To make the initial-boundary value problem \eqref{CMHD}-\eqref{MHDB} be well-posed, the initial data has to satisfy certain compatibility conditions on the boundary. In fact, the continuity equation implies that $\dive v|_{\p\DD}=0$ and thus we have to require $p_0|_{\p\DD_0}$=0 and $\dive v_0|_{\p\DD_0}$=0. Also the constraints on the magnetic field $\dive B=0$ and $B|_{\p\DD}=\mathbf{0}$ requires that $\dive B_0=0$ and $B_0|_{\p \DD_0}=\mathbf{0}$. Furthermore, we define the $k$-th($k\geq 0$) order compatibility condition as follows:
\begin{equation}\label{cck}
D_t^j p|_{\p\DD_0}=0,~~D_t^j B|_{\p\DD_0}=\mathbf{0}~~\text{at time }t=0~~\forall 0\leq j\leq k.
\end{equation} Such initial data has been constructed in the author's previous paper \cite{ZhangCRMHD1}.

In this manuscript, we prove the local well-posedness of the free-boundary compressible resistive MHD system under Rayleigh-Taylor sign condition: System \eqref{CMHD} with initial-boundary condition \eqref{MHDI}-\eqref{MHDB} and physical conditions \eqref{liquid}-\eqref{sign}.

\subsection{History and background}

\subsubsection*{Free-boundary Euler equations}

In the past a few decades, there have been numerous studies of the free-boundary problems in inviscid hydrodynamics. For incompressible Euler equations, Wu's work \cite{wu1997LWPww, wu1999LWPww} on the local well-posedness(LWP) of full water wave system have been considered as the first breakthrough in the study of free-surface perfect fluid. Lannes \cite{lannes2005ww} proved the case when the water wave has a fixed bottom. Note that the incompressible irrotational water wave system can be equivalently written as a dispersive system on the boundary, and thus long time behaviors are expected to be considered. Wu \cite{wu2009GWPww, wu2011GWPww} first established the (almost) global well-posedness. There are also other huge works on the global solution of incompressible water wave with or without surface tension, see Germain-Masmoudi-Shatah \cite{GMSgwp12, GMSgwp14}, Alazard-Burq-Zuily \cite{alazard2014cauchy}, Alazard-Delort \cite{alazard2015global}, Deng-Ionescu-Pausader-Pusateri \cite{dengGWPwwst}, Ifrim-Tataru \cite{ifrim2}, etc. In the case of nonzero vorticity, Christodoulou-Lindblad \cite{christodoulou2000motion} first established the a priori estimates without loss of regularity. Lindblad \cite{lindblad2003, lindblad2005well} proved the local well-posedness by using Nash-Moser iteration. Coutand-Shkoller \cite{coutand2007LWP, coutand2010LWP} introduced the tangential smoothing method to proved the local well-posedness with or without surface tension and avoid the loss of regularity. See also Shatah-Zeng \cite{shatah2008geometry, shatah2008priori,shatah2011local} for nonzero surface tension, and Zhang-Zhang \cite{zhangzhang08Euler} for incompressible water wave with vorticity.

The study of compressible perfect fluids is not quite developed as opposed to incompressible case. Lindblad \cite{lindblad2004, lindblad2005cwell} first established the local well-posedness by Nash-Moser iteration. Trakhinin \cite{trakhiningas2009} proved the case of a non-relativistic and relativistic gas (and liquid) in an unbounded domain by means of hyperbolic system and Nash-Moser iteration. Lindblad-Luo \cite{lindblad2018priori} established the first a priori estimates without loss of regularity and incompressible limit in the case of a liquid by using the wave equation together with delicate elliptic estimates and Luo \cite{luo2018cww} generalized to compressible water wave with vorticity. Later, Ginsberg-Lindblad-Luo \cite{GLL2019LWP} proved the LWP for a self-gravitating liquid, Luo-Zhang \cite{luozhangCWWlwp} proved the LWP for a gravity water wave with vorticity. In the case of nonzero surface tension, we refer readers to Coutand-Hole-Shkoller \cite{coutand2013LWP} for LWP and vanishing surface tension limit, Disconzi-Kukavica \cite{disconzi2017prioriI} for low regularity a priori estimates, Disconzi-Luo \cite{luo2019limit} for the incompressible limit. Among other things, we also mention the related works on the study of a gas: Coutand-Lindblad-Shkoller \cite{coutand2010priori} for the a priori bound, Coutand-Shkoller \cite{coutand2012LWP}, Jang-Masmoudi \cite{jang2014gas} and Luo-Xin-Zeng \cite{luoxinzeng2014} for LWP.

\subsubsection*{Free-boundary MHD equations: Incompressible case}

Compared with perfect fluids, the study of free-boundary MHD is much more complicated due to the strong coupling between fluid and magnetic field and the failure of irrotational assumption and enhanced-regularity curl estimates (See Luo-Zhang \cite{luozhang2019MHD2.5} for detailed discussion). For the incompressible ideal free-boundary MHD, Hao-Luo \cite{hao2014priori} first established the Christodoulou-Lindblad type a priori estimates under Rayleigh-Taylor sign condition and Gu-Wang \cite{gu2016construction} first proved the LWP. Then Hao-Luo \cite{haoluo2019} proved the LWP for the linearized system when the fluid region is diffeomorphic to a ball. Luo-Zhang \cite{luozhang2019MHD2.5} established the low regularity a priori estimates in a small fluid domain. 

For the full plasma-vacuum model, Gu \cite{guaxi1,guaxi2} proved the LWP for the axi-symmetric case. In general case, all of the existing results require a non-collinearity condition $|B\times\hat{B}|\geq c_0>0$ on the free interface which has stronger stabilization effect than Taylor sign condtion \eqref{sign}. Morando-Trakhinin-Trebeschi \cite{iMHDlinear} proved LWP for linearized case. Then Sun-Wang-Zhang \cite{sun2017well} proved the LWP for the full plasma-vacuum model. We also note that the a priori estimates and LWP of the full plasma-vacuum model in incompressible ideal MHD under Rayleigh-Taylor sign condition is still open when the vacuum magnetic field $\hat{B}$ is non-trivial. 

For the viscous and resistive case, we refer to Lee \cite{leeMHD1,leeMHD2}, and Padula-Solonnikov \cite{Solonnikov}. In the case of nonzero surface tension, Luo-Zhang \cite{luozhang2019MHDST} first established the a priori estimates for ideal MHD. Chen-Ding \cite{chendingMHDlimit} proved the inviscid limit. Wang-Xin \cite{wangxinMHD2} proved the global well-posedness for the plasma-vacuum model for inviscid resistive MHD around a uniform transversal magnetic field. Guo-Zeng-Ni \cite{GuoMHDSTviscous} proved the decay rate of viscous-resistive incompressible MHD with surface tension.

\subsubsection*{Free-boundary MHD equations: Compressible case}

The study of free-boundary compressible MHD is even more delicate due to the extra coupling between pressure wave and magnetic field. For compressible ideal MHD ($\lambda=0$), there is one derivative loss in the curl estimates, and thus a loss of normal derivative. We note that analogous loss does not appear in tangential estimates. The reason is that taking curl eliminates the symmetry enjoyed by the equations. Briefly speaking, one can recall the derivation of energy conservation that the term $-\frac12\idt |B|^2\dive u$ is cancelled by part of $-\idt u\cdot\nabla P$. But taking the curl eliminates the counterpart of $-\idt u\cdot\nabla P$ \textbf{before} such cancellation is produced because of $\curl\nabla P=0$. On the other hand, if we taking tangential derivatives instead of curl, then analogous cancellation is still preserved. We recommend readers to read Section 1.5 in the author's previous work \cite{ZhangCRMHD1} for detailed discussion.  

Before further discussion on the free-boundary problem in compressible MHD, let us first review the results for the fixed-domain problems. The compressible ideal MHD system in a fixed domain is a quasi-linear symmetric hyperbolic system with characteristic boundary conditions \cite{secchi1996}. The loss of normal derivatives seems necessary, especially when the uniform Lopatinskii condition fails for the linearized counterpart. On the one hand, Ohno-Shirota \cite{MHDexample} constructed explicit counter-examples to the local existence in standard Sobolev space $H^l(l\geq 2)$ for the linearized problem. On the other hand, Chen Shu-Xing \cite{CSX} first introduced the anisotropic Sobolev spaces $H_*^m$ to compensate the normal derivative loss for characteristic boundary problems. With the help of such funtional spaces, Yanagisawa-Matsumura \cite{MHDfirst} established the first LWP result for compressible ideal MHD. Later, the LWP result was improved by Secchi \cite{secchi1995,secchi1996} such that there is no regularity loss from initial data to solution, both are in $H_*^m$ for $m\geq 16$. However, it is still difficult to generalized Secchi's results to free-boundary case due to the extra derivative loss brought by free boundary.

Although there is derivative loss even for the linearized problem due to the failure of uniform Lopatinskii condition, the tame estimates, which allows a fixed order loss of derivatives, can still be established in anisotropic Sobolev space. Recently, Trakhinin-Wang \cite{trakhininMHD2020} proved the LWP for free-boundary compressible ideal MHD under Rayleigh-Taylor sign condition by using such method and Nash-Moser iteration. However, a suitable energy estimates without loss of regularity cannot be established by Nash-Moser iteration. To resolve such problem, we observe that the magnetic resistivity exactly compensates the derivative loss mentioned above. With the help of this observation, the author \cite{ZhangCRMHD1} recently proved the Christodoulou-Lindblad type a priori estimates in standard Sobolev spaces and incompressible limit for the free-boundary compressible resistive MHD system. As a contiuation, the presenting manuscript deals with the local well-posedness of this problem. Note that it is still quite non-trivial to pass from a priori bound to local existence for a free-boundary problem: Direct iteration and fixed-point argument often fails due to the lack of regularity of free boundary. Once the local well-posedness is established, the next possible goal is to consider the vanishing resistivity limit, though it seems quite difficult now.

As for the full plasma-vacuum model in compressible ideal MHD, Secchi-Trakhinin \cite{secchi2013well} proved the LWP under the non-collinearity condition $|B\times\hat{B}|\geq c_0>0$. The study of plasma-vacuum model in compressible ideal MHD under Rayleigh-Taylor sign condition is still open. See Trakhinin \cite{trakhininMHD2016} for more detailed discussion.

\subsection{Reformulation in Lagrangian coordinates}

We use Lagrangian coordinates to reduce free-boundary problem to a fixed domain. We assume $\Omega:=\T^2\times(0,1)$ to be the reference doamin. The coordinates on $\Omega$ is $y=(y_1,y_2,y_3)$. We define $\eta:[0,T]\times\Omega\to \DD$ as the flow map of velocity field $u$, i.e., 
\begin{align}
\p_t\eta(t,y) = u(t, \eta(t,y)),\q \eta(0,y)=y. 
\end{align}
By chain rule, it is easy to see that $D_t$ becomes $\p_t$ in the $(t,y)$ coordinates and the free-boundary $\p\DD_t$ becomes fixed ($\Gamma=\T^2$). We introduce the Lagrangian varaibles as follows: $v(t,y):=u(t, \eta(t,y))$, $b(t,y):=B(t, \eta(t,y))$, $q(t,y):=Q(t, \eta(t,y))$, and $R(t,y):=\rho(t, \eta(t,y))$.

Let $\p=\p_y$ be the spatial derivative in the Lagrangian coordinates. We introduce the cofactor matrix $a=[\p \eta]^{-1}$, specifically $a^{\mu\alpha}=a^{\mu}_{\alpha}:=\frac{\p y^{\mu}}{\p \eta^{\alpha}}$. In terms of $\eta,v,b,q,R$, the system \eqref{CMHD}-\eqref{sign} becomes

\begin{equation}\label{CMHDL}
\begin{cases}
\p_t\eta=v &~~~ \text{in}\,[0,T]\times\Omega\\
R\p_t v=\left(b\cdot\pa\right)b-\pa Q &~~~ \text{in}\,[0,T]\times\Omega\\
R'(q)\p_tq+R\di_a v=0 &~~~ \text{in}\,[0,T]\times\Omega\\
\p_t b-\lambda\Delta_a b=\left(b\cdot\pa\right)v-b\di_a v  &~~~ \text{in}\,[0,T]\times\Omega\\
\di_a b=0  &~~~ \text{in}\,[0,T]\times\Omega\\
\p_t|_{[0,T]\times \Gamma}\in T([0,T]\times \Gamma)\\
p=0,~~B=\mathbf{0} &~~~\text{on}\,[0,T]\times\Gamma\\
-\p_3 Q_0|_{\Gamma}\geq c_0>0\\
(\eta,v,b,q,R)=(\text{Id},v_0,b_0,q_0,\rho_0).
\end{cases}
\end{equation}

Here, $\nab_a^\alpha = a^{\mu\alpha}\p_\mu$ and $\di_a X=\nab_a\cdot X=a^{\mu\alpha}\p_\mu X_\alpha$ denote the covariant derivative and divergence in Lagrangian coordinates (or say Eulerian derivative/divergence). In the manuscript, we adopt the convention that the Greek indices range over $1,2,3$, and the Latin indices range over $1$ and $2$. In addition, since $\eta(0,\cdot)=\text{Id}$, we have $a(0,\cdot)=I$, where $I$ is the identity matrix, and $u_0,B_0,p_0$ and $v_0,b_0,q_0$ agree respectively. Furthermore, let $J:=\det [\p\eta]$ and $A:=Ja$. Then we have the Piola's identity
\begin{align}
\p_\mu A^{\mu\alpha}=0
\end{align}
And $J$ satisfies
\begin{align}
\p_t J = Ja^{\mu\alpha} \p_\mu v_\alpha, 
\label{eq of J}
\end{align} which together with $R'(q)\p_tq+R\di_a v=0$ gives that $\rho_0=RJ$. Therefore, $R$ can be fully expressed in terms of $\eta$ and $\rho_0$, and thus the initial-boundary value problem can be expressed as follows:

\begin{equation}\label{CRMHDL}
\begin{cases}
\p_t\eta=v &~~~ \text{in}\,[0,T]\times\Omega\\
\rho_0J^{-1}\p_t v=\left(b\cdot\pa\right)b-\pa Q &~~~ \text{in}\,[0,T]\times\Omega\\
\frac{JR'(q)}{\rho_0}\p_tq+\di_a v=0 &~~~ \text{in}\,[0,T]\times\Omega\\
\p_t b-\lambda\Delta_a b=\left(b\cdot\pa\right)v-b\di_a v  &~~~ \text{in}\,[0,T]\times\Omega\\
\di_a b=0  &~~~ \text{in}\,[0,T]\times\Omega\\
\p_t|_{[0,T]\times \Gamma}\in T([0,T]\times \Gamma)\\
p=0,~~B=\mathbf{0} &~~~\text{on}\,[0,T]\times\Gamma\\
-\p_3 Q|_{\Gamma}\geq c_0>0  &~~~\text{on}\,\{t=0\}\\
(\eta,v,b,q)=(\text{Id},v_0,b_0,q_0).
\end{cases}
\end{equation}

Our result is 
\begin{thm}
Let the initial data $v_0\in H^4(\Omega),~b_0\in H^{5}(\Omega),~q_0\in H^{4}(\Omega)$ satisfy the compatibility conditions up to 5-th order. Then there exists some $T_1>0$, such that the system \eqref{CRMHDL} has a unique solution $(\eta,v,b,q)$ in $[0,T_1]$ satisfying the following estimates
\begin{align}\label{noloss}
\sup_{0\leq T\leq T_1}\EE(T)\leq 2\left(\left\|v_0\right\|_4^2+\left\|b_0\right\|_{5}^2+\left\|q_0\right\|_4^2\right),
\end{align}where
\begin{equation}\label{EE}
\EE(T):=E(T)+H(T)+W(T)+\left\|\p_t^{4-k}\left(\left(b\cdot\pa\right)b\right)\right\|_{k}^2,
\end{equation}
where
\begin{align}
\label{E} E(T):=\left\|\eta\right\|_4^2+\left|a^{3\alpha}\TP^4\eta_{\alpha}\right|_0^2&+\sum_{k=0}^4 \left(\left\|\p_t^{4-k}v\right\|_{k}^2+\left\|\p_t^{4-k}b\right\|_{k}^2+\left\|\p_t^{4-k} q\right\|_k^2\right),\\
\label{H} H(T)&:=\int_0^T\io\left|\p_t^{5} b\right|^2\dy\dt+\left\|\p_t^4 b\right\|_1^2,\\
\label{W} W(T)&:=\left\|\p_t^{5} q\right\|_0^2+\left\|\p_t^4 q\right\|_1^2.
\end{align}
\end{thm}
\begin{rmk}
~
\begin{enumerate}
\item For simplicity we set the magnetic diffusivity constant $\lambda=1$.
\item The initial data satisfying the compatibility conditions has been constructed in the author's previous paper \cite{ZhangCRMHD1}, so these steps are omitted in this manuscript.
\end{enumerate}
\end{rmk}

\bigskip

We list all the notations used in this manuscript:

\bigskip

\noindent\textbf{List of Notations: }
\begin{itemize}
\item $\Omega:=\T^2\times(0,1)$ and $\Gamma:=\T^2\times(\{0\}\cup\{1\})$.
\item $\|\cdot\|_{s}$:  We denote $\|f\|_{s}: = \|f(t,\cdot)\|_{H^s(\Omega)}$ for any function $f(t,y)\text{ on }[0,T]\times\Omega$.
\item $|\cdot|_{s}$:  We denote $|f|_{s}: = |f(t,\cdot)|_{H^s(\Gamma)}$ for any function $f(t,y)\text{ on }[0,T]\times\Gamma$.
\item $P(\cdots)$:  A generic polynomial in its arguments;
\item $\PP_0$:  $\PP_0=P(\|v_0\|_4,\|B_0\|_5,\|q_0\|_4)$;
\item $[T,f]g:=T(fg)-T(f) g$, and $[T,f,g]:=T(fg)-T(f)g-fT(g)$, where $T$ denotes a differential operator or the mollifier and $f,g$ are arbitrary functions.
\item $\TP,\TL$: $\TP=\p_1,\p_2$ denotes the tangential derivative and $\TL:=\p_1^2+\p_2^2$ denotes the tangential Laplacian.
\item $\nabla^{\alpha}_a f:=a^{\mu\alpha}\p_{\mu} f$, $\dive_a\mathbf{f}:=a^{\mu\alpha}\p_\mu \mathbf{f}_{\alpha}$ and $(\curl_a\mathbf{f})_\lambda:=\epsilon_{\lambda\mu\alpha}a^{\nu\mu}\p_\nu \mathbf{f}^{\alpha}$, where $\epsilon_{\lambda\mu\alpha}$ is the sign of the 3-permutation $(\lambda\mu\alpha)\in S_3.$
\end{itemize}
\begin{flushright}
$\square$
\end{flushright}

\subsection{Strategy of the proof}\label{strategy}

\subsubsection{Necessity of tangential smoothing}

Since there have been results \cite{ZhangCRMHD1} on the a priori bound of this system, it is natural to consider the fixed point argument and Picard iteration to construct the solution to \eqref{CRMHDL}. Specifically, one starts with trivial solution $(\eta^{(0)},v^{(0)},b^{(0)},q^{(0)})=(\text{Id},\mathbf{0,0},0)$, and inductively define $(\eta^{(n+1)},v^{(n+1)},b^{(n+1)},q^{(n+1)})$ by the following linearized system provided that $\{(\eta^{(k)},v^{(k)},b^{(k)},q^{(k)})\}_{0\leq k\leq n}$ are given. (Note that the div-free condition for magnetic field is  not needed in the nonlinear estimates.)
\begin{equation}\label{linearnn}
\begin{cases}
\p_t \eta^{(n+1)}=v^{(n+1)} &~~~\text{in } \Omega, \\
\frac{\rho_0}{J^{(n)}}\p_t v^{(n+1)}=(b^{(n)}\cdot\nabla_{a^{(n)}}) b^{(n+1)}-\nabla_{a^{(n)}} Q^{(n+1)},~~Q^{(n+1)}=q^{(n+1)}+\frac{1}{2}|b^{(n+1)}|^2 &~~~\text{in } \Omega, \\
\frac{J^{(n)} R'(q^{(n)})}{\rho_0}\p_tq^{(n+1)}+\dive_{a^{(n)}} v^{(n+1)}=0 &~~~\text{in } \Omega, \\
\p_t b^{(n+1)}-\Delta_{a^{(n)}} b^{(n+1)}=(b^{(n)}\cdot\nabla_{a^{(n)}}) v^{(n+1)}-b^{(n)}\dive_{a^{(n)}} v^{(n+1)} , &~~~\text{in } \Omega, \\
q^{(n+1)}=0,~b^{(n+1)}=\mathbf{0} &~~~\text{on } \Gamma, \\
(\eta^{(n+1)},v^{(n+1)}, b^{(n+1)},q^{(n+1)})|_{\{t=0\}}=(\text{Id},v_0, b_0,q_0).
\end{cases}
\end{equation}
Then we need to 
\begin{itemize}
\item Construct the solution to the linearized system \eqref{linearnn}.
\item Prove uniform-in-$n$ a priori estimates of \eqref{linearnn}.
\item Prove  $\{(\eta^{(n)},v^{(n)},b^{(n)},q^{(n)})\}_{n\in \N}$  converges strongly.
\end{itemize}
The first step is usually done by standard fixed point argument. While in the second step, we need to invoke the elliptic estimates \cite{christodoulou2000motion} Lemma \ref{GLL} to control $\pa b$ and $\pa q$ which is necessary for us to construct suitable functional space in step 1:
\[
\left\|\nabla_{a^{(n)}} b^{(n+1)}\right\|_{4}\lesssim P(\|\eta^{(n)}\|_4)\left(\|\Delta_a b^{(n+1)}\|_{3}+\|\TP\eta^{(n)}\|_4\|b\|_4\right).
\] We find that the $\eta$ loses one tangential derivative, which motivates us to regularise the flow map in tangential directions. This method was first introduced by Coutand-Shkoller \cite{coutand2007LWP,coutand2012LWP} and then widely used in numerous works on the study of free-boundary hydrodynamics. One can define $\ek:=\lkk^2\eta$ and replace $a$ by $\ak:=[\p\ek]^{-1}$ where $\lkk$ is the standard convolution mollifier in $\R^2$. Since the region is $\T^2\times(0,1)$, the boundary only appears in normal directions and thus the mollifier is always well-defined. After introducing the tangential smoothing, we need to investigate the following nonlinear $\kk$-approximation system:
\begin{equation}\label{nonlinearkk0}
\begin{cases}
\p_t \eta=v &~~~\text{in } \Omega, \\
\rho_0\Jk^{-1}\p_t v=(b\cdot\pak) b-\pak Q,~~Q=q+\frac{1}{2}|b|^2 &~~~\text{in } \Omega, \\
\frac{\Jk R'(q)}{\rho_0}\p_tq+\diva v=0 &~~~\text{in } \Omega, \\
\p_t b-\lapak b=(b\cdot\pak) v-b\diva v, &~~~\text{in } \Omega, \\
q=0,~b=\mathbf{0} &~~~\text{on } \Gamma, \\
-\p_3Q|_{t=0}\geq c_0>0 &~~~\text{on } \Gamma, \\
(\eta,v, b,q)|_{\{t=0\}}=(\text{Id},v_0, b_0,q_0).
\end{cases}
\end{equation} Specifically, we need to
\begin{enumerate}
\item Establish nonlinear a priori estimates uniform in $\kk>0$.
\item Construct the unique strong solution to the nonlinear $\kk$-approximation system for each fixed $\kk>0$.
\end{enumerate}

\subsubsection{A priori estimates of the smoothed approximation system}

\subsubsection*{Magnetic field and Lorentz force: Elliptic estimates}

The ideas of deriving the a priori estimates are quite different from ideal MHD case. For ideal MHD,  the magnetic field can be written as $b=J^{-1}(b_0\cdot\p)\eta$, which should be controlled together with the same derivatives of $v=\p_t\eta$, and higher order terms are expected to vanish due to subtle cancellation. However, the magnetic resistivity yields a direct control of the Lorentz force $(b\cdot\pak)b$ as a source term. First, we notice that $b$ vanishes on the boundary and so does the Lorentz force. Therefore, Lemma \ref{GLL} can be directly applied to $(b\cdot\pak)b$ in the following way:
\[
\left\|(b\cdot\pak) b\right\|_4\approx\left\|\pak((b\cdot\pak) b)\right\|_{3}\lesssim P(\|\ek\|_{3})\left(\left\|\lapak((b\cdot\pak) b)\right\|_{2}+\left\|\TP\ek\right\|_{3}\left\|(b\cdot\pak) b\right\|_{3}\right).
\] Then we invoke the heat equation of $b$ to replace the Laplacian term by first order derivative and thus control the Lorentz force by lower order terms. Similar processes can be applied to the time derivatives of Lorentz force and the space-time derivatives of magnetic field $b$. 

\begin{rmk}
We can control $\|(b\cdot\pak)b\|_4$ by elliptic estimates as above, but it is NOT possible to control the $H^5$-norm of $b$ even if one has the diffusion effect on the magnetic field. Indeed, if we use elliptic estimates to control $\|b\|_5$, then
\[
\|\pak b\|_4\lesssim P(\|\ek\|_4)(\|\lapak b\|_3+\|\TP\ek\|_4\|b\|_4),
\]where $\|\TP\ek\|_4$ is out of control. In other words, the regularity of the free surface is not enough to enhance the full spatial regularity of $b$.
\end{rmk}

\subsubsection*{Velocity and pressure: Div-curl and tangential estimates}

The second and third equations of the nonlinear $\kk$-system above present the following relations if we temporarily omit the coefficients $\rho_0 J$ and $R'(q)$:
\[
\p\p_t^{k} q\approx \pak \p_t^k q\approx-\p_t^{k+1}v+\p_t^k\left((b\cdot\pak)b-\frac{1}{2}|b|^2\right),
\]
and
\[
\p\p_t^{k}v\xrightarrow{\dive}\p_t^k\diva v+\text{curl}+\text{boundary}\approx\p_t^{k+1}q+\text{curl}+\text{boundary}.
\]
Again we note that, it is exactly the curl part that denies the possibility of an energy control in standard Soboelv spaces for compressible ideal MHD. However, for resistive MHD, the contribution of curl part can be directly controlled by taking $\curla$ in the second equation of \eqref{nonlinearkk0} because $\curla(\pak Q)=0$ and $\curla(b\cdot\pak)b$ has been directly controlled. As for the boundary term, we are able to further reduce to interior tangential estimates and divergence estimates by using normal trace Lemma \ref{normaltrace}. Since the terms containing magnetic field $b$ can be reduced to lower order with the help of magnetic diffusion, the procedure above allows us to trade one spatial derivative by one time derivative:
\begin{align*}
\p^4 q\to\p^3\p_t v\xrightarrow{\dive}\p^2\p_t^2q\to\p\p_t^3v &\xrightarrow{\dive} \p_t^4 q \\
\p^3\p_t q\to\p^2\p_t^2 v\xrightarrow{\dive}\p\p_t^3 q&\to\p_t^4 v,
\end{align*}which reduces to the estimates of full time derivatives.

In the tangential estimates, if the tangential derivatives $\dd^4$ contain at least one time derivative, then the energy estimates can be derived by direct computation because $[\dd^4,\ak^{\mu\alpha}]\p_{\mu}Q$ can be controlled. However, when $\dd^4=\TP^4$ contains no time derivative, that commutator is out of control due to $\TP^4\ak\approx\TP^r\p\ek\cdot\p\ek$. In the study of Euler equations, one can integrate $\TP^{1/2}$ by parts and use the enhanced regularity $\|\eta\|_{9/2}$ to control this commutator, but here the presence of magnetic field destroys such property. The reason is that MHD system does not preserve the irrotational assumption propagated from initial data, and thus it is not possible for the flow map to get enhanced regularity. Our method is to introduce the Alinhac good unknowns: For a function $f$, we define $\mathbf{f}:=\TP^4 f-\TP^4\eta\cdot\pak f$ to be the Alinhac good unknown of $f$ when commuting $\TP^4$ with $\pak$. Under this setting, we have that 
\[
\TP^4(\pak f)=\pak\mathbf{f}+\text{ controllable commutators}.
\] Such crucial fact was first observed by Alinhac \cite{alinhacgood89}, and then widely used in the study of first-order quasilinear hyperbolic system. In the study of fluid mechanics, there are also many applications in the study of incompressible inviscid fluids, e.g., \cite{MRgood2017,wangxingood,gu2016construction} and so on, but not widely used in the compressible case \cite{luozhangCWWlwp}. 

Now we apply $\tpl$ (instead of $\TP^4$, will be explained later) to the second equation of \eqref{nonlinearkk0} and define $\VV=\tpl v-\tpl\ek\cdot\pak v$ and $\QQ=\tpl Q-\tpl\ek\cdot\pak Q$ to be the Alinhac good unknowns for $v$ and $Q=q+\frac{1}{2}|B|^2$. Then we have
\[
\rho_0\Jk^{-1}\p_t\VV+\pak\QQ=\underbrace{\tpl\left((b\cdot\pak)b\right)+[\rho_0\Jk^{-1},\tpl]\p_t v_{\alpha}+\rho_0\Jk^{-1}\p_t(\tpl\ek\cdot\pak v)+C(Q)}_{\FF},
\]subject to 
\[
\QQ=-\tpl\ek_{\beta}\ak^{3\beta}\p_3 Q~~~on~\Gamma.
\] and 
\[
\pak\cdot\VV=\tpl(\diva v)-C^{\alpha}(v_{\alpha})~~~in~\Omega.
\]

Taking $L^2$ inner product with $\VV$, and integrating by parts, the interior term is $$\int_0^T\io\QQ\tpl(\diva v)\dy\dt.$$ Plugging $Q=q+\frac{1}{2}|b|^2$ and $\diva v\approx -\p_t q$ into this integral, we are able to get an energy term $$-\frac12\io\frac{\Jk R'(q)}{\rho_0}\left|\tpl q\right|^2\dy$$ plus controllable terms. The boundary term reads

\[
I_0=\int_0^T\ig \p_3 Q\ak^{3\alpha}\ak^{3\beta}\tpl\ek_{\beta}(\tpl\p_t\eta_{\alpha}-\tpl\ek\cdot\pak v_{\alpha})\dS\dt.
\]

The first term produces the Taylor sign term contributing to the boundary term in $E_\kk(t)$ after commuting a $\lkk$:

\[
\begin{aligned}
&\int_0^T\ig \p_3 Q\ak^{3\alpha}\ak^{3\beta}\tpl \ek_{\beta}\tpl\p_t\eta_{\alpha}\dS\dt \\
=&\frac{1}{2}\ig \p_3 Q\left|\ak^{3\alpha}\tpl\lkk\eta_{\alpha}\right|^2\dS\bigg|^T_0\\
&+\ig \p_3 Q \ak^{3\beta}\tpl\lkk\eta_{\beta}\ak^{3\gamma}\TP_i\lkk^2v_{\gamma}\ak^{i\alpha}\tpl\lkk\eta_{\alpha}\dS
-\ig \p_3 Q\ak^{3\alpha}\ak^{3\beta}\tpl\ek_{\beta}\tpl\ek_{\gamma}\ak^{i\gamma} \TP_iv_{\alpha}\dS
\end{aligned}
\]
When $\kk=0$, the two terms in the second line automatically cancel, while for the smoothed version such structure no longer holds. However, inspired by a remarkable observation in Gu-Wang \cite{gu2016construction}, we notice that these two terms can be cancelled if we introduce a correction term $$\ig \p_3 Q\ak^{3\alpha}\ak^{3\beta}\tpl\ek_{\beta}\tpl\psi\dS,$$ where $\psi$ is defined by 
\begin{equation}
\begin{cases}
\Delta \psi=0  &~~~\text{in }\Omega, \\
\psi=\TL^{-1}\mathbb{P}_{\neq 0}\left(\TL\eta_{\beta}\ak^{i\beta}\TP_i\lkk^2 v-\TL\lkk^2\eta_{\beta}\ak^{i\beta}\TP_i v\right) &~~~\text{on }\Gamma,
\end{cases}
\end{equation}where $\mathbb{P}_{\neq 0} $ denotes the standard Littlewood-Paley projection in $\T^2$  which removes the zero-frequency part. $\TL:=\p_1^2+\p_2^2$ denotes the tangential Laplacian operator. Therefore, it suffices to replace $\p_t\eta=v$ by $\p_t\eta=v+\psi$ to produce such an extra cancellation. The structure of $\psi$ also illustrates why we have to replace $\TP^r$ by $\tpl$.

\subsubsection*{Common control of higher order wave and heat equation}

In the previous estimates, we still need to control the $r$-th time derivatives of $\p_t b,(b\cdot\pak) b,\p_tq,\pak q$, which exactly come from the energy functionals of $r$-th time-differentiated heat eqution of $b$
\[
\p_t b-\lapak b=(b\cdot\pak) v-b\diva v
\] and wave equation of $q$
\[
\frac{\Jk R'(q)}{\rho_0}\p_t^2 q-\lapak q=b\cdot\lapak b+w_0,
\]where $w_0$ consists of first order derivatives terms. It is easy to observe that the Laplacian term in RHS of wave equation can be replaced by terms of first-order derivative with the help of magnetic resistivity. In other words, the magnetic resistivity compensate the derivative loss in the wave equation of $q$. Therefore, standard energy estimates of heat and wave equations give the control of $H_\kk(T)$ and $W_\kk(T)$. Note that $\p_t$ is tangential on the boundary. Integrating by parts does not produce any boundary term for the time differentiated equations. Therefore, we are able to close the a priori estimates.

\subsubsection{Linearisation and Picard iteration}

Now, it remains to prove the local existence of the solution to the nonlinear $\kk$-approximation system. The proof is standard linearisation and Picard iteration argument. We start with the trivial solution $(\eta^{(0)},v^{(0)},b^{(0)},q^{(0)})=(\eta^{(1)},v^{(1)},b^{(1)},q^{(1)})=(\text{Id},0,0,0)$. Inductively, given $\{(\eta^{(k)},v^{(k)},b^{(k)},q^{(k)})\}_{0\leq k\leq n}$ for some given $n\in\N^*$, we define $(\eta^{(n+1)},v^{(n+1)},b^{(n+1)},q^{(n+1)})$ by linearisation around $a^{(n)}:=[\p\eta^{(n)}]^{-1}$.

\begin{equation}\label{linearnnn}
\begin{cases}
\p_t \eta^{(n+1)}=v^{(n+1)}+\psi^{(n)} &~~~\text{in } \Omega, \\
\frac{\rho_0}{\Jk^{(n)}}\p_t v^{(n+1)}=(b^{(n)}\cdot\nabla_{\ak^{(n)}}) b^{(n+1)}-\nabla_{\ak^{(n)}} Q^{(n+1)},~~Q^{(n+1)}=q^{(n+1)}+\frac{1}{2}|b^{(n+1)}|^2 &~~~\text{in } \Omega, \\
\frac{\Jk^{(n)} R'(q^{(n)})}{\rho_0}\p_tq^{(n+1)}+\dive_{\ak^{(n)}} v^{(n+1)}=0 &~~~\text{in } \Omega, \\
\p_t b^{(n+1)}-\Delta_{\ak^{(n)}} b^{(n+1)}=(b^{(n)}\cdot\nabla_{\ak^{(n)}}) v^{(n+1)}-b^{(n)}\dive_{\ak^{(n)}} v^{(n+1)} , &~~~\text{in } \Omega, \\
q^{(n+1)}=0,~b^{(n+1)}=\mathbf{0} &~~~\text{on } \Gamma, \\
(\eta^{(n+1)},v^{(n+1)}, b^{(n+1)},q^{(n+1)})|_{\{t=0\}}=(\text{Id},v_0, b_0,q_0).
\end{cases}
\end{equation} Note that the directional derivative $(b\cdot\pak)$ should be given by the $n$-th solution instead of $(n+1)$-th, otherwise the linearisation is not complete and the construction of the solution to the linearized system is much more difficult in the fixed-point argument.

The construction of the solution to the linearized system needs the following function space in the fixed-point argument:

Define the norm $\|\cdot\|_{\XX^r}$ by
\[
\left\|f\right\|_{\XX^r}^2:=\sum_{m=0}^r\sum_{k+l=m}\left\|\p_t^k\p^lf\right\|_0^2
\] and a Banach space on $[0,T]\times\Omega$
\[
\XX(M,T):=\left\{\left(\xi,w,h,\pi\right)\bigg|\left(\xi,w,h,\pi\right)\bigg|_{t=0}=\left(\text{Id},v_0,b_0,q_0\right),\left\|\left(\xi,w,h,\pi\right)\right\|_{\XX}\leq M\right\}
\]where
\[
\left\|\left(\xi,w,h,\pi\right)\right\|_{\XX}:=\left\|\left(\xi,\p_t\xi,w,h,\park h,\pi,\p_t\pi,\park\pi\right)\right\|_{L_t^{\infty}\XX^4}+\|\p_t^5 h\|_{L_t^2L_x^2},
\] and $\ar$ denotes $a^{(n)}$. Here $(\xi,w,h,\pi)$ are the corresponding variables of $(\eta,v,b,q)$. 

We notice that, unlike the nonlinear estimates, the $H^4$-norm of $\park \pi$ is necessary because one has to use $v(T)=v_0+\int_0^T\p_t v$ to bound the $H^4$-norm of $v$. Such term should be estimated by means of elliptic estimates Lemma \ref{GLL}
\[
\|\park q\|_4\lesssim P(\|\erk\|_4)(\|\lapark q\|_3+\|\TP\erk\|_4\|q\|_4).
\] With the help of tangential smoothing, we are able to control $\|\TP\erk\|_4$ by sacrificing a $\kk^{-1}$: This is not allowed in the nonlinear estimates due to the $\kk$-independence requirement, however, in the linearisation part $\kk>0$ is fixed. Therefore, mimicing the previous proof of nonlinear a priori estimates, we are able to 
\begin{itemize}
\item Prove the local well-posedness of the solution to the linearised system \eqref{linearnnn}
\item Establish uniform-in-$n$ estimates of \eqref{linearnnn}
\item Finish the iteration scheme to produce the solution to nonlinear $\kk$-approximation system.
\end{itemize}This finalizes the whole proof.

\bigskip

\section{Preliminary Lemmas}\label{lemmata}

We need the following Lemmas in this manuscript. 

\subsection{Sobolev trace lemma}

\begin{lem}\label{harmonictrace}
Suppose that $s\geq 0.5$ and $u$ solves the boundary-valued problem
\[
\begin{cases}
\Delta u=0~~~&\text{ in }\Omega,\\
u=g~~~&\text{ on }\Gamma
\end{cases}
\] where $g\in H^{s}(\Gamma)$. Then it holds that 
\[
|g|_{s}\lesssim\|u\|_{s+0.5}\lesssim|g|_{s}
\]
\end{lem}
\begin{proof}
The result follows from the standard Sobolev trace lemma and Proposition 5.1.7 in M. Taylor \cite{taylorPDE1}.
\end{proof}

\begin{lem}[\textbf{Normal trace theorem}]\label{normaltrace}
It holds that for a vector field $X$
\begin{equation}\label{ntr}
\left|\TP X\cdot N\right|_{-0.5}\lesssim\|\TP X\|_0+\|\dive X\|_0
\end{equation}
\end{lem}
\begin{proof}
The proof directly follows from testing by a $H^{0.5}(\Gamma)$ function and divergence theorem. See Lemma 3.4 in \cite{gu2016construction}.
\end{proof}

\subsection{Properties of tangential smoothing operator}

As stated in the introduction, we are going to use the tangential smoothing to construct the approximate solutions. Here we list the definition and basic properties which are repeatedly used in this paper. Let $\zeta=\zeta(y_1,y_2)\in C_c^{\infty}(\R^2)$ be a standard cut-off function such that $\text{Spt }\zeta=\overline{B(0,1)}\subseteq\R^2,~~0\leq\zeta\leq 1$ and $\int_{\R^2}\zeta=1$. The corresponding dilation is $$\zeta_{\kk}(y_1,y_2)=\frac{1}{\kk^2}\zeta\left(\frac{y_1}{\kk},\frac{y_2}{\kk}\right),~~\kk>0.$$ Now we define
\begin{equation}\label{lkk0}
\lkk f(y_1,y_2,y_3):=\int_{\R^2}\zeta_{\kk}(y_1-z_1,y_2-z_2)f(z_1,z_2,z_3)\dz_1\dz_2.
\end{equation}

The following lemma records the basic properties of tangential smoothing.
\begin{lem}\label{tgsmooth}
 \textbf{(Regularity and Commutator estimates)} For $\kk>0$, we have

(1) The following regularity estimates:
\begin{align}
\label{lkk11} \|\lkk f\|_s&\lesssim \|f\|_s,~~\forall s\geq 0;\\
\label{lkk1} |\lkk f|_s&\lesssim |f|_s,~~\forall s\geq -0.5;\\ 
\label{lkk2} |\TP\lkk f|_0&\lesssim \kk^{-s}|f|_{1-s}, ~~\forall s\in [0,1];\\  
\label{lkk3} |f-\lkk f|_{L^{\infty}}&\lesssim \sqrt{\kk}|\TP f|_{1/2}.  
\end{align}

(2) Commutator estimates: Define the commutator $[\lkk,f]g:=\lkk(fg)-f\lkk(g)$. Then it satisfies
\begin{align}
\label{lkk4} |[\lkk,f]g|_0 &\lesssim|f|_{L^{\infty}}|g|_0,\\ 
\label{lkk5} |[\lkk,f]\TP g|_0 &\lesssim |f|_{W^{1,\infty}}|g|_0, \\ 
\label{lkk6} |[\lkk,f]\TP g|_{0.5}&\lesssim |f|_{W^{1,\infty}}|g|_{0.5}.
\end{align}
\end{lem}
\begin{proof}
(1): The estimates \eqref{lkk1} and \eqref{lkk2} follows directly from the definition \eqref{lkk0} and the basic properties of convolution. \eqref{lkk3} is derived by using Sobolev embedding and H\"older's inequality:
\begin{align*}
|f-\lkk f|&=\left|\int_{\R^2\cap B(0,\kk)} \zeta_{\kk}(z) (f(y-z)-f(y))\dz\right| \\
&\lesssim |\zeta_{\kk}|_{L^{4/3}}|\kk\TP f|_{L^4} \\
&\lesssim \sqrt{\kk}|\zeta|_{L^{4/3}} |f|_{1.5}.
\end{align*}

(2): The first three estimates can be found in Lemma 5.1 in Coutand-Shkoller \cite{coutand2010LWP}. To prove the fourth one, we note that
\[
\TP([\lkk,f]g)=\lkk(\TP f\TP g)+\lkk(f\TP^2 g)-\TP f\lkk \TP g-f \lkk \TP^2g=[\lkk,\TP f]\TP g+[\lkk, f]\TP^2 g.
\]
From \eqref{lkk4} and \eqref{lkk5} we know 
\begin{equation}\label{lkk7}
|\TP[\lkk,f]g|_0\lesssim|\TP f|_{L^{\infty}}|\TP g|_{0}+|f|_{W^{1,\infty}}|\TP g|_0\lesssim |f|_{W^{1,\infty}}|g|_1.
\end{equation} Therefore \eqref{lkk6} follows from the interpolation of \eqref{lkk5} and \eqref{lkk7}.
\end{proof}

\subsection{Elliptic estimates}
\begin{lem} \label{hodge} \textbf{(Hodge-type decomposition)} Let $X$ be a smooth vector field and $s\geq 1$, then it holds that
\begin{equation}
\|X\|_s\lesssim\|X\|_0+\|\curl X\|_{s-1}+\|\dive X\|_{s-1}+|X\cdot N|_{s-0.5}.
\end{equation}
\end{lem}
\begin{proof}
This follows from the well-known identity $-\Delta X=\curl\curl X-\nabla\dive X$.
\end{proof}

\begin{lem}\label{GLL}\textbf{(Christodoulou-Lindblad elliptic estimate)}
The following elliptic estimate holds for $r\geq 2$,
\begin{align}
\|\pak f\|_{H^r} \leq  C(\|\ek\|_{r})\Big(\|\lapak  f\|_{r-1}+\|\cp \ek\|_{r}\|f\|_{r}\Big).
\end{align}
\end{lem}
\begin{proof}
See Ginsberg-Lindblad-Luo \cite{GLL2019LWP} Proposition 5.3. When $r=1$, then $\|\cdot\|_r$ norm should be replaced by $\|\cdot\|_{L^{\infty}}$.
\end{proof}

\section{A priori estimates of the nonlinear approximation system}\label{nonlinear}

For $\kk>0$, we introduce the nonlinear $\kk$-approximation system. 

\begin{equation}\label{nonlinearkk}
\begin{cases}
\p_t \eta=v+\psi &~~~\text{in } \Omega, \\
\rho_0\Jk^{-1}\p_t v=(b\cdot\pak) b-\pak Q,~~Q=q+\frac{1}{2}|b|^2 &~~~\text{in } \Omega, \\
\frac{\Jk R'(q)}{\rho_0}\p_tq+\diva v=0 &~~~\text{in } \Omega, \\
\p_t b-\lapak b=(b\cdot\pak) v-b\diva v, &~~~\text{in } \Omega, \\
q=0,~b=\mathbf{0} &~~~\text{on } \Gamma, \\
(\eta,v, b,q)|_{\{t=0\}}=(\text{Id},v_0, b_0,q_0).
\end{cases}
\end{equation} Here, $\Omega:=\T^2\times(0,1)$ is the reference domain, $\Gamma:=\p\Omega=\T^2\times(\{0\}\cup\{1\})$ is the boundary. The notation  $\ek$ is the smoothed version of the flow map $\eta$ defined by $\ek:=\lkk^2\eta$, $\ak:=(\p\ek)^{-1}$ is the cofactor matrix of $[\p\ek]$, and $\Jk:=\det[\p\ek]$ is the Jacobian determinant. The term $\psi=\psi(\eta,v)$ is a correction term which solves the Laplacian equation
\begin{equation}\label{psi}
\begin{cases}
\Delta \psi=0  &~~~\text{in }\Omega, \\
\psi=\TL^{-1}\mathbb{P}_{\neq 0}\left(\TL\eta_{\beta}\ak^{i\beta}\TP_i\lkk^2 v-\TL\lkk^2\eta_{\beta}\ak^{i\beta}\TP_i v\right) &~~~\text{on }\Gamma,
\end{cases}
\end{equation}where $\mathbb{P}_{\neq 0} $ denotes the standard Littlewood-Paley projection in $\T^2$  which removes the zero-frequency part. $\TL:=\p_1^2+\p_2^2$ denotes the tangential Laplacian operator.

\bigskip

\begin{rmk}
~

\begin{enumerate}
\item The correction term $\psi\to 0$ as $\kk\to 0$. We introduce such a term to eliminate the higher order boundary terms which appears in the tangential estimates of $v$. These higher order boundary terms are zero when $\kk=0$ but are out of control when $\kk>0$.
\item The Littlewood-Paley projection is necessary here because we will repeatedly use $$|\TL^{-1}\mathbb{P}_{\neq 0}f|_{s}\approx |\mathbb{P}_{\neq 0}f|_{{H}^{s-2}}\approx |f|_{\dot{H}^{s-2}},$$ which can be proved by using Bernstein inequality.
\item The initial data is the same of origin system because the compatibility conditions stay unchanged after mollification by $\ak(0)=a(0)=\text{Id}$. Such initial data has been constructed in the author's previous paper \cite{ZhangCRMHD1}, so these steps are omitted in this manuscript.
\item The precise form of the commutators in the remaining context can be found in Section 4.4 in the author's previous work \cite{ZhangCRMHD1}. Details are omitted in the presenting manuscript.
\end{enumerate}
\end{rmk}

Now, we define the energy functional of \eqref{nonlinearkk}
\begin{equation}\label{EEkk}
\EE_\kk(T):=E_\kk(T)+H_\kk(T)+W_\kk(T)+\left\|\p_t^{4-k}\left(\left(b\cdot\pak\right)b\right)\right\|_{k}^2,
\end{equation}
where
\begin{align}
\label{Ekk} E_{\kk}(T):=\left\|\eta\right\|_4^2+\left|\ak^{3\alpha}\TP^4\lkk\eta_{\alpha}\right|_0^2&+\sum_{k=0}^4 \left(\left\|\p_t^{4-k}v\right\|_{k}^2+\left\|\p_t^{4-k}b\right\|_{k}^2+\left\|\p_t^{4-k} q\right\|_k^2\right),\\
\label{Hkk} H_{\kk}(T)&:=\int_0^T\io\left|\p_t^5 b\right|^2\dy\dt+\left\|\p_t^4 b\right\|_1^2,\\
\label{Wkk} W_{\kk}(T)&:=\left\|\p_t^5 q\right\|_0^2+\left\|\p_t^4 q\right\|_1^2
\end{align}denote the energy functional of fluid, higher order heat equation of $b$, and wave equation of $q$, respectively.

The context of this section is the uniform-in-$\kk$ a priori estimates of \eqref{nonlinearkk}.  

\begin{prop}\label{nonlinear1}
There exists some $T>0$ independent of $\kk$, such that the energy functional $E_\kk$ satisfies
\begin{equation}\label{EEkk1}
\sup_{0\leq t\leq T} \EE_\kk(t)\leq P(\|v_0\|_{4},\|b_0\|_5,\|q_0\|_4,\|\rho_0\|_4),
\end{equation} provided the following assumptions hold for all $t\in[0,T]$
\begin{align}
\label{taylor1} -\p_3 Q(t)\geq c_0/2~~~&\text{on }\Gamma,\\
\label{small1} \|\Jk(t)-1\|_3+\|\text{Id}-\ak(t)\|_3\leq \epsilon~~~&\text{in }\Omega.
\end{align}
\end{prop}

\begin{rmk}
The a priori assumptions can be easily justified once the energy bounds are established by using $\ak(T)-\text{Id}=\int_0^T \p_t\ak=\int_0^T \ak:\p_t\p\ek:\ak\dt$ and the smallness of $T$. See Lemma \ref{coeff}.
\end{rmk}

In Section \ref{linear}, we will prove the local well-posedness of \eqref{nonlinearkk} in an $\kk$-dependent time interval $[0,T_{\kk}]$.  Therefore, the uniform-in-$\kk$ a priori estimate guarantees that the solution $(\eta(\kk),v(\kk),b(\kk),q(\kk))$ to \eqref{nonlinearkk} converges to the solution to the original system as $\kk\to 0$, i.e., local existence of the solution to free-boundary compressible resistive MHD system is established.  For simplicity, we omit the $\kk$ and only write $(\eta,v,b,q)$ in this manuscript.

\subsection{Estimates of the correction term}\label{noneta}

First we bound the flow map and the correction term together with their smoothed version by the quantities in $\EE_{\kk}$. The following estimates will be repeatedly use in this section.

\begin{lem}\label{etapsi}
We have the following estimates for $(v,\psi,\eta)$ in \eqref{nonlinearkk}. 
\begin{align}
\label{eta4} \|\ek\|_4&\lesssim \|\eta\|_4, \\
\label{psi4} \|\psi\|_4&\lesssim P(\|\eta\|_4,\|v\|_3), \\
\label{psit4}\|\p_t \psi\|_4&\lesssim P(\|\eta\|_4,\|v\|_4,\|\p_t v\|_3),\\
\label{psitt3}\|\p_t^2 \psi\|_3&\lesssim P(\|\eta\|_4,\|v\|_4, \|\p_tv\|_3, \|\p_t^2 v\|_2), \\
\label{psittt2}\|\p_t^3\psi\|_2&\lesssim P(\|\eta\|_4,\|v\|_4, \|\p_tv\|_3, \|\p_t^2 v\|_2, \|\p_t^3v\|_{1}),\\
\label{psitttt1}\|\p_t^4\psi\|_1&\lesssim P(\|\eta\|_4,\|v\|_4, \|\p_tv\|_3, \|\p_t^2 v\|_2, \|\p_t^3v\|_{1},\|\p_t^4 v\|_0).
\end{align}
and
\begin{align}
\label{etat4} \|\p_t\ek\|_4\lesssim\|\p_t\eta\|_4&\lesssim P(\|\eta\|_4,\|v\|_4) , \\
\label{etatt3}\|\p_t^2\ek\|_3\lesssim\|\p_t^2\eta\|_3&\lesssim P(\|\eta\|_4,\|v\|_4,\|\p_t v\|_3) , \\
\label{etattt2}\|\p_t^3\ek\|_2\lesssim\|\p_t^3\eta\|_2&\lesssim P(\|\eta\|_4,\|v\|_4,\|\p_t v\|_3,\|\p_t^2 v\|_2), \\
\label{etatttt1}\|\p_t^4\ek\|_1\lesssim\|\p_t^4\eta\|_1&\lesssim  P(\|\eta\|_4,\|v\|_4, \|\p_tv\|_3, \|\p_t^2 v\|_2, \|\p_t^3v\|_{1})\\
\label{etattttt0}\|\p_t^5\ek\|_0\lesssim\|\p_t^5\eta\|_0&\lesssim  P(\|\eta\|_4,\|v\|_4, \|\p_tv\|_3, \|\p_t^2 v\|_2, \|\p_t^3v\|_{1},\|\p_t^4 v\|_0).
\end{align}
\end{lem}
\begin{proof}
First, \eqref{eta4} follow from \eqref{lkk11}, i.e., $\|\ek\|_{4}=\|\lkk^2\p \eta\|_{4}\lesssim\| \eta\|_{4}$. To bound $\p_t^k\ek$, it suffices to bound the same norm of $\p_t^k\eta$ and then apply \eqref{lkk11} again. From the first equation of \eqref{nonlinearkk}, one has $\p_t^{k+1}\eta=\p_t^k v+\p_t^k \psi$, so the estimates \eqref{etat4}-\eqref{etatttt1} directly follow from \eqref{psi4}-\eqref{psittt2}.

Commuting time derivatives through \eqref{psi}, we get the equations for $\p_t^k\psi~(k=0,1,2,3,4)$:
\begin{equation}\label{psitk}
\begin{cases}
\Delta \p_t^k\psi=0  &~~~\text{in }\Omega, \\
\p_t^k\psi=\TL^{-1}\mathbb{P}_{\neq 0}\p_t^k\left(\TL\eta_{\beta}\ak^{i\beta}\TP_i\lkk^2 v-\TL\lkk^2\eta_{\beta}\ak^{i\beta}\TP_i v\right) &~~~\text{on }\Gamma.
\end{cases}
\end{equation}
By the standard elliptic estimates and Sobolev trace lemma, we can get 
\begin{equation}\label{psi040}
\begin{aligned}
\|\psi\|_{4}\lesssim&\left|\TL^{-1}\mathbb{P}_{\neq 0}\left(\TL\eta_{\beta}\ak^{i\beta}\TP_i\lkk^2 v-\TL\lkk^2\eta_{\beta}\ak^{i\beta}\TP_i v\right)\right|_{3.5} \\
\lesssim& \left\|\TL\eta_{\beta}\ak^{i\beta}\TP_i\lkk^2 v-\TL\lkk^2\eta_{\beta}\ak^{i\beta}\TP_i v\right\|_{2} \\
\lesssim& P(\|\eta\|_4,\|v\|_3).
\end{aligned}
\end{equation}
Similarly one has
\begin{align}
\label{psit40} \|\p_t\psi\|_{4}\lesssim& P(\|\eta\|_4,\|v\|_4, \|\p_t v\|_3),\\
\label{psitt30} \|\p_t^2\psi\|_{3}\lesssim&P(\|\eta\|_4,\|v\|_4,\|\p_t v\|_3,\|\p_t^2v\|_2).
\end{align}

When $k=3$, it just needs more rigorous discussion.
\begin{equation}\label{psittt20} 
\begin{aligned}
\|\p_t^3\psi\|_{2}&\lesssim\left|\TL^{-1}\p_t^3\mathbb{P}\left(\TL\eta_{\beta}\ak^{i\beta}\TP_i\lkk^2 v-\TL\lkk^2\eta_{\beta}\ak^{i\beta}\TP_i v\right)\right|_{1.5} \\
&\lesssim \left|~\mathbb{P}\p_t^3\left(\TL\eta_{\beta}\ak^{i\beta}\TP_i\lkk^2 v-\TL\lkk^2\eta_{\beta}\ak^{i\beta}\TP_i v\right)\right|_{-0.5}\\
\end{aligned}
\end{equation}
where in the last step we apply the Bernstein's inequality.

Combining with $\p_t^{k+1}\eta=\p_t^k v+\p_t^k \psi$, \eqref{etat4}, \eqref{etatt3} and \eqref{etattt2} directly follows from \eqref{psit40} and \eqref{psitt30}, respectively. When $k=3$, one has to be cautious because the leading order term in \eqref{psittt20} is of the form $(\p_t^3\TL\eta)\ak\TP v$ and $\TL\eta\ak(\p_t^3\TP v)$ which can only be bounded in $L^2(\Omega)$ by the quantites in $\EE_{\kk}$ and thus loses control on the boundary. To control these terms on the boundary, we have to use the fact that $\dot{H}^{0.5}(\T^2)=(\dot{H}^{0.5}(\T^2))^*$.

First we separate them from other lower order terms which has $L^2(\Gamma)$ control.
\begin{equation}\label{XY}
\begin{aligned}
&~~~~\mathbb{P}\p_t^3\left(\TL\eta_{\beta}\ak^{i\beta}\TP_i\lkk^2 v-\TL\lkk^2\eta_{\beta}\ak^{i\beta}\TP_i v\right) \\
&=\mathbb{P}\underbrace{\left(\p_t^3 \TL\eta_{\beta}\ak^{i\beta}\TP_i\lkk^2 v-\p_t^3\TL\lkk^2\eta_{\beta}\ak^{i\beta}\TP_i v+\TL\eta_{\beta}\ak^{i\beta}\p_t^3\TP_i\lkk^2 v-\TL\lkk^2\eta_{\beta}\ak^{i\beta}\p_t^3\TP_i v\right)}_{\text{leading order terms=:X}}+\mathbb{P}Y.
\end{aligned}
\end{equation}
The control of $Y$ is straightforward by using Sobolev trace lemma and \eqref{etat4}, \eqref{etatt3},
\begin{equation}\label{Y}
\begin{aligned}
|\mathbb{P}Y|_{-0.5}&\leq|\mathbb{P}Y|_{0}\lesssim\|Y\|_{0.5} \\
&\lesssim P(\|\p_t^2\eta\|_{2.5}, \|\p_t \eta\|_{3.5}, \|\p_t^2 \ak\|_{1.5},\|\p_t^2 v\|_{1.5}, \|\p_t v\|_{2.5}) \\
&\lesssim P(\|\eta\|_4, \|v\|_4,\|\p_t v\|_3,\|\p_t^2 v\|_2).
\end{aligned}
\end{equation}
As for the $|\mathbb{P} X|_{-0.5}$ term, we first use the Bernstein inequality to get $|\mathbb{P} X|_{-0.5}\approx |X|_{\dot{H}^{-0.5}}$. Then the duality between $\dot{H}^{-0.5}$ and $\dot{H}^{0.5}$ yields that for any test function $\phi\in \dot{H}^{0.5}(\T^2)$ with $|\phi|_{\dot{H}^{0.5}}\leq 1$, one has
\begin{equation}\label{X1}
\begin{aligned}
\langle \TL\eta_{\beta}\ak^{i\beta}\p_t^3\TP_i\lkk^2 v, \phi\rangle &=\langle \p_t^3\TP_i\lkk^2 v, \TL\eta_{\beta}\ak^{i\beta}\phi\rangle=\langle \TP_i^{0.5}\p_t^3\lkk^2 v,\TP_i^{0.5}(\TL\eta_{\beta}\ak^{i\beta}\phi)\rangle \\
&\lesssim|\p_t^3\lkk^2 v|_{\dot{H}^{0.5}}|\TL\eta\ak\phi|_{\dot{H}^{0.5}} \\
&\lesssim\|\p_t^3 v\|_{1}(|\phi|_{\dot{H}^{0.5}}|\TL\eta\ak|_{L^{\infty}}+|\TL\eta\ak|_{\dot{W}^{0.5,4}}|\phi|_{L^4})\\
&\lesssim\|\p_t^3 v\|_1(\|\eta\|_4\|a\|_{L^{\infty}})|\phi|_{\dot{H}^{0.5}}.
\end{aligned}
\end{equation} Here we integrate 1/2-order tangential derivative on $\Gamma$ by part in the second step, and then apply trace lemma to control $|\p_t^3 \lkk^2 v|_{\dot{H}^{0.5}}$ and Kato-Ponce product estimate to bound $|\TL\eta\ak\phi|_{\dot{H}^{0.5}}$. Taking supremum over all $\phi\in H^{0.5}(\R^2)$ with $|\phi|_{\dot{H}^{0.5}}\leq 1$, we have by the definition of $\dot{H}^{0.5}$-norm that 
\begin{equation}\label{X0}
|\TL\eta_{\beta}\ak^{i\beta}\p_t^3\TP_i\lkk^2 v|_{\dot{H}^{-0.5}}\lesssim P(\|\eta\|_4,\|\p_t^3 v\|_1).
\end{equation}
Similarly as above, we have 
\begin{align}
\label{X2} |\TL\lkk^2\eta_{\beta}\ak^{i\beta}\p_t^3\TP_i v|_{\dot{H}^{-0.5}}&\lesssim P(\|\eta\|_4,\|\p_t^3 v\|_1), \\
\label{X3} |\p_t^3 \TL\eta_{\beta}\ak^{i\beta}\TP_i\lkk^2 v-\p_t^3\TL\lkk^2\eta_{\beta}\ak^{i\beta}\TP_i v|_{\dot{H}^{-0.5}}&\lesssim P(\|\p_t^3\eta\|_2, \|\eta\|_4, \|v\|_3).
\end{align}
Combining \eqref{psittt20}-\eqref{X3} and the bound \eqref{etattt2} for $\p_t^3\eta$, we get 
\[
\|\p_t^3\psi\|_2\lesssim P(\|\eta\|_4,\|v\|_4, \|\p_tv\|_3, \|\p_t^2 v\|_2, \|\p_t^3v\|_{1}),
\]which is exactly \eqref{psittt2}. Hence, \eqref{etatttt1} directly follows from \eqref{psittt2} and $\p_t^4\eta=\p_t^3(v+\psi)$.

Similar estimates hold for $\p_t^4\psi$. First we have

\begin{equation}\label{psir4}
\begin{cases}
\Delta \p_t^4\psi=0  &~~~\text{in }\Omega, \\
\p_t^4\psi=\TL^{-1}\mathbb{P}\p_t^4\left(\TL\eta_{\beta}\ak^{i\beta}\TP_i\lkk^2 v-\TL\lkk^2\eta_{\beta}\ak^{i\beta}\TP_i v\right) &~~~\text{on }\Gamma.
\end{cases}
\end{equation}
Using Lemma \ref{harmonictrace} for harmonic functions, we know
\begin{align*}
\|\p_t^4\psi\|_1\lesssim&|\p_t^4\psi|_{0.5}=\left|\TL^{-1}\mathbb{P}_{\neq 0}\p_t^4\left(\TL\eta_{\beta}\ak^{i\beta}\TP_i\lkk^2 v-\TL\lkk^2\eta_{\beta}\ak^{i\beta}\TP_i v\right)\right|_{0.5} \\
\lesssim&|\mathbb{P}_{\neq 0}\p_t^4\left(\TL\eta_{\beta}\ak^{i\beta}\TP_i\lkk^2v-\TL\lkk^2\eta_{\beta}\ak^{i\beta}\TP_i v\right)|_{\dot{H}^{-1.5}}\\
\lesssim&|\mathbb{P}_{\neq 0}\p_t^4\left(\TL\eta_{\beta}\ak^{i\beta}\TP_i\lkk^2 v-\TL\lkk^2\eta_{\beta}\ak^{i\beta}\TP_i v\right)|_{\dot{H}^{-0.5}}.
\end{align*}

The most difficult terms appear when $\p_t^4$ falls on $\TL\eta$ or $\TP v$. We only show how to control $\TL\ek_{\beta}\ak^{i\beta}\TP_i\p_t^4 v$ and the rest follows in the same way. Given any test function $\phi\in \dot{H}^{0.5}(\T^2)$ with $|\phi|_{\dot{H}^{0.5}}\leq 1$, we consider
\begin{align*} 
\left|\langle \TL\lkk^2\eta_{\beta}\ak^{i\beta}\TP_i\p_t^4 v, \phi\rangle\right|=&\left|\langle \TP_i\p_t^4 v,\TL\lkk^2\eta_{\beta}\ak^{i\beta}\phi\rangle\right|=\left|\langle \TP^{0.5}\p_t^4 v, \TP^{0.5}(\TL\lkk^2\eta_{\beta}\ak^{i\beta}\phi)\rangle\right| \\
\lesssim&\left|\p_t^4 v\right|_{\dot{H}^{0.5}} \left|\TL\lkk^2\eta_{\beta}\ak^{i\beta}\phi\right|_{\dot{H}^{0.5}} \\
\lesssim&  \left(\|\p_t^4 v\|_1\|\eta\|_4^2\right)\left|\phi\right|_{\dot{H}^{0.5}},
\end{align*}

Taking supremum over all  $\phi\in \dot{H}^{0.5}(\T^2)$ with $|\phi|_{\dot{H}^{0.5}}\leq 1$, we obtain 
\[
\|\TL\lkk^2\eta_{\beta}\ak^{i\beta}\TP_i\p_t^4 v\|_{\dot{H}^{-0.5}}\lesssim  \|\p_t^4 v\|_1\|\eta\|_4^2,
\] and thus gives the bound for $\|\p_t^4\psi\|_1$. The estimates of $\p_t^5\eta,\p_t^5 J$ follows directly from $\p_t^5\eta=\p_t^4 v+\p_t^4\psi$ and $J=\det[\p\eta].$
\end{proof}

\subsection{Estimates of the magnetic field}\label{nonb}

For ideal MHD,  the magnetic field can be written as $b=J^{-1}(b_0\cdot\p)\eta$, which should be controlled together with the same derivatives of $v=\p_t\eta$, and higher order terms are expected to vanish due to subtle cancellation.  But for resistive MHD, the magnetic diffusion, together with the boundary condition $b=\mathbf{0}$, allows us to control the higher order term $\lapak b$ directly, and $(b\cdot\pak)b$ (Lorentz force) with the help of Christodoulou-Lindblad type elliptic estimates Lemma \ref{GLL}.

\subsubsection{Control of $\p_t^k b$ when $k\leq 2$}\label{nonb1}

First we estimate $\|\p_t^{4-k}b\|_k$. Notice that, when $k\geq 1$, we have $$\p^{k-1}\p_{\alpha}\p_t^{4-k} b=\p^{k-1}(\ak^{\mu}_{\alpha}\p_{\mu}\p_t^{4-k}b)+\p^{k-1}((\delta^{\mu}_{\alpha}-\ak^{\mu}_{\alpha})\p_\mu\p_t^{4-k}b),$$ which gives
$$\|\p_t^{4-k} b\|_k^2\lesssim\|\pak \p_t^{4-k} b\|_{k-1}^2+\|\text{Id}-\ak\|_{k-1}^2\|\p_t^{4-k}b\|_{k}^2\leq \|\pak \p_t^{4-k} b\|_{k-1}^2+\epsilon^2\|\p_t^{4-k}b\|_{k}^2.$$  Here the $\epsilon$-term can be absorbed into the LHS.  When $k=1,2$, $\|\text{Id}-\ak\|_{k-1}$ should be replaced by $\|\text{Id}-\ak\|_{L^{\infty}}$. Therefore, we have that for $1\leq k\leq 4$, $\|\p_t^{4-k}b\|_k\lesssim \|\pak(\p_t^{4-k} b)\|_{k-1}$, which motivates us to use Christodoulou-Lindblad elliptic estimates Lemma \ref{GLL}.

 Applying Lemma \ref{GLL} to $b$, we have
\begin{equation}\label{be}
\|b\|_4\approx\|\pak b\|_3\lesssim P(\|\ek\|_3)(\|\lapak b\|_2+\|\TP\ek\|_3\|b\|_3)
\end{equation}

Invoking the heat equation $\lapak b=\p_t b-(b\cdot\pak)v+b\diva v$, we have
\begin{equation}\label{b04}
\begin{aligned}
\|b\|_4\lesssim&P(\|\ek\|_3)\left(\|\p_t b\|_2+\|(b\cdot\pak)v\|_2+\|b\diva v\|_2+\|\TP\ek\|_3\|b\|_3\right)\\
\lesssim& P(\|\ek\|_3)P\left(\|\p_t b\|_2,\|b\|_2,\|u\|_3\right) +P(\|\ek\|_3)\|\TP\ek\|_3\|b\|_3\\
\end{aligned}
\end{equation}

The first term $P(\|\ek\|_3)P\left(\|\p_t b\|_2,\|b\|_2,\|u\|_3\right)$ can be directly controlled by $\PP_0+\int_0^T P(E_{\kk}(t))\dt$. For the second term, we notice that $\|\TP\ek\|_3\lesssim\|\TP\ek\|_0+\|\p^3\TP\ek\|_0$, and $\TP_i\eta_{\alpha}|_{t=0}=\delta_{i\alpha}$, $\p^3\TP\eta|_{t=0}=0$, so $$\|\TP\ek\|_3\lesssim 1+\int_0^T \|\p_t\TP\ek\|_3\dt.$$ Plugging this into the second term $\left(\|\TP\eta\|_0+\|\TP\p\ek\|_2\right)P(\|\ek\|_3)\|b\|_3$, we know 
\begin{align*}
&\left(\|\TP\eta\|_0+\|\TP\p\ek\|_2\right)P(\|\ek\|_3)\|b\|_3\\
\lesssim&P(\|\ek\|_3)\|b\|_3\int_0^T\|\p_t\TP\ek\|_3\dt +P(\|\ek\|_3)\left(\|b_0\|_3+\int_0^T\|\p_t b\|_3\dt\right)\\
\lesssim&\PP_0+P(E_\kk(t))\int_0^T P(E_\kk(t))\dt.
\end{align*}
Therefore, \eqref{b04} becomes
\begin{equation}\label{b4}
\|b\|_4\lesssim \PP_0+P(E_\kk(t))\int_0^TP(E_\kk(t))\dt
\end{equation}

Since $\p_t$ is tangential on $\Gamma$, $\p_t b$ also vanishes on the boundary. Applying elliptic estimates as in \eqref{be}, we get
\begin{align}
\label{bte} \|\p_t b\|_3&\approx\|\pak \p_t b\|_2\lesssim P(\|\ek\|_2)\left(\|\lapak \p_t b\|_1+\|\TP\ek\|_2\|\p_t b\|_2\right)\\
\label{btte} \|\p_t^2 b\|_2&\approx\|\pak \p_t^2 b\|_1\lesssim P(\|\ek\|_2)\left(\|\lapak \p_t^2 b\|_0+\|\TP\ek\|_2\|\p_t^2 b\|_1\right).
\end{align}

Taking time derivatives in the heat equation of $b$, we have $$\lapak \p_t^k b=\p_t^{k+1}b-\p_t^{k}\left((b\cdot\pak)v-b\diva v\right),$$ of which the RHS is of one less derivative than LHS. Therefore, we are able to control $\|\p_t b\|_3,\|\p_t^2b\|_2$ in the same way as \eqref{b4}:
\begin{equation}\label{bt3btt2}
\|\p_t b\|_3+\|\p_t^2 b\|_2\lesssim \PP_0+P(E_\kk(t))\int_0^TP(E_\kk(t))\dt
\end{equation}

\subsubsection{Control of $\p_t^3 b$: Heat equation}\label{nonb2}

Note that $\|\p_t^3 b\|_1\approx\|\pak\p_t^3 b\|_0$ is a part of the energy of 3-rd order time-differentiated heat equation $$\p_t^4 b-\lapak \p_t^3 b=\p_t^3\left((b\cdot\pak) v)-b\diva v\right)+[\p_t^3,\lapak]b,$$ of which the RHS only contain terms with $\leq 4 $ derivatives.

Taking $L^2$ inner product with $\Jk\p_t^4 b$, integrating in $y\in\Omega$ and $t\in[0,T]$, and then integrating by parts, one has
\begin{align*}
LHS=&\int_0^T\io \Jk\left|\p_t^4 b\right|^2\dy\dt-\int_0^T\io\p_t^4 b\cdot\Jk\lapak\p_t^3 b\dy\dt\\
=&\int_0^T\io \Jk\left|\p_t^4 b\right|^2\dy\dt+\int_0^T\io\Jk\pak\p_t^4 b\cdot\pak\p_t^3b\dy\dt\\
=&\int_0^T\io \Jk\left|\p_t^4 b\right|^2\dy\dt+\frac{1}{2}\io\Jk\left|\pak\p_t^3 b\right|^2\dy\bigg|^T_0\\
&-\frac{1}{2}\int_0^T\io\p_t \Jk\left|\pak\p_t^3 b\right|^2\dy\dt+\int_0^T\io\Jk\left[\pak,\p_t\right]\p_t^3 b\cdot\pak\p_t^3 b\dy\dt\\
RHS=&\int_0^T\io \p_t^4 b\cdot\left(\p_t^3\left((b\cdot\pak) v)-b\diva v\right)+[\p_t^3,\lapak]b\right)\dy\dt
\end{align*}

Therefore, one has 
\begin{equation}\label{bttt1}
\begin{aligned}
&\int_0^T\io \Jk\left|\p_t^4 b\right|^2\dy\dt+\frac{1}{2}\io\Jk\left|\pak\p_t^3 b(T)\right|^2\dy\\
=&\frac{1}{2}\io\Jk\left|\pak\p_t^3 b(0)\right|^2\dy+\frac{1}{2}\int_0^T\io\p_t \Jk\left|\pak\p_t^3 b\right|^2\dy\dt-\int_0^T\io\Jk\left[\pak,\p_t\right]\p_t^3 b\cdot\pak\p_t^3 b\dy\dt\\
&+\int_0^T\io \p_t^4 b\cdot\left(\p_t^3\left((b\cdot\pak) v)-b\diva v\right)+[\p_t^3,\lapak]b\right)\dy\dt\\
\lesssim&\PP_0+\int_0^T P(E_{\kk}(t))\dt,
\end{aligned}
\end{equation}which gives the $H^1$ control of $\p_t^3 b$.

\subsubsection{Control of $\p_t^4 b$: Higher order estimates needed}\label{nonb3}

There are two ways to control $\|\p_t^4 b\|_0$. One way is to use Poincar\'e's inequality 
\begin{equation}\label{b401}
\|\p_t^4 b\|_0\lesssim \|\p_t^4 b\|_1\approx\|\pak\p_t^4 b\|_0
\end{equation} due to $\p_t^4 b|_{\Gamma}=0$
 Another way is direct computation
\begin{equation}\label{b402}
\begin{aligned}
\frac{1}{2}\|\p_t^4 b\|_0^2=&\frac{1}{2}\|\p_t^4 b(0)\|_0^2+\int_0^T\p_t^4 b\cdot\p_t^5 b\dt\\
\lesssim&\PP_0+\left\|\p_t^5 b\right\|_{L_t^2L_x^2([0,T]\times\Omega)}\left\|\p_t^4 b\right\|_{L_t^2L_x^2([0,T]\times\Omega)}\\
\lesssim&\PP_0+\epsilon\int_0^T\io\left|\p_t^5b\right|^2\dy\dt+\frac{1}{4\epsilon}\int_0^T \io\left|\p_t^4 b\right|^2\dy\dt\\
\lesssim&\epsilon\int_0^T\io\left|\p_t^5b\right|^2\dy\dt+\PP_0+\int_0^TP(E_\kk(t))\dt.
\end{aligned}
\end{equation}
From \eqref{b401} and \eqref{b402}, we find that either $\|\pak \p_t^4 b\|_0^2$ or $\|\p_t^5 b\|_{L_t^2L_x^2}$ is required to control $\|\p_t^4 b\|_0^2$. On the other hand, we notice that these two terms exactly come from the energy functional of 4-th time-differentiated heat equation of $b$:
\[
\p_t^5 b-\lapak\p_t^4b=\left[\p_t^4,\lapak\right]b+\p_t^4\left((b\cdot\pak)v-b\diva v\right).
\] The energy estimate cannot be controlled in the same way as in Section \ref{nonb2} because the RHS of this heat equation contains 5-th order derivatives. Instead, we will seek for a common control of $b$ and $p$ via the heat equation and wave equation. This part will be postponed to Section \ref{noncommon}.

\subsubsection{Estimates of Lorentz force}\label{nonbb}

Later on we will see both the estimates of $u$ and common control of higher order heat and wave equations require the control of 5-th derivatives of magnetic field, all of which are actually 4-th space-time derivatives of Lorentz force $(b\cdot\pak)b$. Notice that $b=0$ on the boundary implies $(b\cdot\pak)b$ also vanishes on $\Gamma$. Therefore, we can apply the elliptic estimate Lemma \ref{GLL} to $(b\cdot\pak)b$.

We start with $\|(b\cdot\pak)b\|_4$. Similarly as in \eqref{be}, we have
\begin{equation}\label{bbe}
\|(b\cdot\pak)b\|_{4}\approx\|\pak((b\cdot\pak)b)\|_3\lesssim P(\|\ek\|_3)\left(\|\lapak ((b\cdot\pak)b)\|_2+\|\TP\ek\|_3\|(b\cdot\pak)b\|_3\right)
\end{equation}

The second term $P(\|\ek\|_3)\|\TP\ek\|_3\|(b\cdot\pak)b\|_3$ can again be controlled by $\PP_0+P(E_\kk(T))\int_0^TP(E_\kk(t))\dt$ by writting $\|\TP\ek\|_3\lesssim1+\int_0^T\|\p_t\TP\ek\|_3$ as in \eqref{b4}. For the first term, we invoke the heat equation of $b$ to get
\begin{align*}
\lapak\left((b\cdot\pak)b\right)=&(b\cdot\pak)\left(\lapak b\right)+[\lapak,b\cdot\pak] b\\
=&(b\cdot\pak)\left(\p_t b-(b\cdot\pak)v+b\diva v\right)+[\lapak,b\cdot\pak] b,
\end{align*} of which the RHS only contains terms with $\leq 2$ derivatives. So we have
\[
\|\lapak\left((b\cdot\pak)b\right)\|_2\lesssim P(E_\kk(T)),
\]and thus
\begin{equation}\label{bb4}
\|(b\cdot\pak)b\|_{4}\lesssim P(E_\kk(T))+\PP_0+P(E_\kk(T))\int_0^TP(E_\kk(t))\dt.
\end{equation}

When $k=1,2$, $\|\p_t^k((b\cdot\pak)b)\|_{4-k}$ can be controlled in the same way as \eqref{bte}, \eqref{btte} and \eqref{bt3btt2}:
\begin{equation}\label{bbtbbtt}
\begin{aligned}
&\|\p_t((b\cdot\pak)b)\|_{3}+\|\p_t^2((b\cdot\pak)b)\|_{2}\\
\lesssim& P(\|\ek\|_2)\left(\|\lapak\p_t((b\cdot\pak)b)\|_1+\|\lapak\p_t^2((b\cdot\pak)b)\|_0\right)\\
&+P(\|\ek\|_2)\|\TP\ek\|_2\left(\|\p_t((b\cdot\pak)b)\|_{2}+\|\p_t^2((b\cdot\pak)b)\|_{1}\right)\\
\lesssim&P(\|\ek\|_2)\left(\|\p_t(b\cdot\pak)(\lapak b)\|_1+\|[\lapak,\p_t(b\cdot\pak)]b\|_1+\|\p_t^2(b\cdot\pak)\lapak b\|_0+\|[\lapak,\p_t^2(b\cdot\pak)]b\|_0\right)\\
&+\PP_0+P(E_{\kk}(T))\int_0^T P(E_{\kk}(t))\dt \\
\lesssim& P(E_\kk(T))+\PP_0+P(E_{\kk}(T))\int_0^T P(E_{\kk}(t))\dt.
\end{aligned}
\end{equation}

When $k=3$, we have $\p(\p_t^3(b\cdot\pak)b)=(b\cdot\pak)\p\p_t^3 b+[\p\p_t^3,b\cdot\pak]b$, where the commutator only contains the terms of $\leq 4$-th order derivative, so 
\begin{equation}\label{bbttt1}
\begin{aligned}
\|\p_t^3(b\cdot\pak)b\|_1\lesssim&\|(b\cdot\pak)\p\p_t^3 b\|_0+\|[\p\p_t^3,b\cdot\pak]b\|_0\\
\lesssim&\|b\|_2\|\pak\p_t^3 b\|_1+P(E_\kk(T))
\end{aligned}
\end{equation} Then by elliptic estimates Lemma \ref{GLL} and the heat equation, 
\begin{align*}
\|\pak\p_t^3 b\|_1\lesssim& P(\|\ek\|_2)(\|\lapak \p_t^3 b\|_0+\|\TP\ek\|_2\|\p_t^3 b\|_1)\\
\lesssim& P(\|\ek\|_2)\left(\|\p_t^4 b\|_0+\|\p_t^3((b\cdot\pak)v-b\diva v)\|_0+\|[\lapak,\p_t^3]b\|_0\right)+ P(E_\kk(T))\\
\lesssim&P(E_\kk(T))
\end{align*}

Similarly, for $k=4$, we have $\p_t^4((b\cdot\pak)b)=(b\cdot\pak)\p_t^4 b+[\p_t^4,b\cdot\pak]b$, where the commutator only contains the terms of $\leq 4$-th order derivative, so 
\begin{equation}\label{bbtttt1}
\begin{aligned}
\|\p_t^4(b\cdot\pak)b\|_0\lesssim&\|(b\cdot\pak)\p_t^4 b\|_0+\|[\p_t^4,b\cdot\pak]b\|_0\\
\lesssim&\|b\|_2\|\pak\p_t^4 b\|_0+P(E_\kk(T)),
\end{aligned}
\end{equation}where the term $\|\pak \p_t^4 b\|_0^2$ is exactly part of the energy functional $H_\kk(T)$ of 4-th time-differentiated heat equation. 

Summing up \eqref{bb4}, \eqref{bbtbbtt}, \eqref{bbttt1} and \eqref{bbtttt1}, we get the estimates of Lorentz force
\begin{equation}\label{bb1}
\sum_{k=0}^4\left\|\p_t^{4-k}((b\cdot\pak)b)\right\|_k^2\lesssim \left\|b\right\|_2^2\left\|\pak \p_t^4 b\right\|_0^2+ P(E_{\kk}(T))+\PP_0+P(E_{\kk}(T))\int_0^TP(E_{\kk}(T))\dt
\end{equation}

Therefore, we find that the estimates of Lorentz force are again reduced to the control of higher order heat equation.

\subsection{Estimates of the velocity and the pressure}\label{nonvp}

In this part we control the space-time Sobolev norm of $v$ and $q$. We first apply the Hodge-type div-curl decomposition (Lemma \ref{hodge}) to $v$ (and its time derivatives). The curl part can be directly controlled by the counterpart of Lorentz force. The boundary term can be reduced to interior tangential estimates by using Sobolev trace Lemma. The divergence part together with the estimates of $q$ can be reduced to the control of full time derivatives, which is also part of tangential estimates. One should keep in mind that, we no longer seek for subtle cancellation to eliminate higher order terms as what was done for ideal MHD, no matter in curl or tangential estimates. Instead, those higher order terms (with 5 derivatives) can be controlled either by Lorentz force, or by the combination of heat equation and wave equation, i.e., $H_\kk(T)$ and $W_\kk(T)$. 

First we recall Lemma \ref{hodge}: $\forall s\geq 1$ and any sufficiently regular vector field $X$, we have
\[
\|X\|_s\lesssim\|X\|_0+\|\dive X\|_{s-1}+\|\curl X\|_{s-1}+|X\cdot N|_{s-1/2}.
\]

Let $X=v,\p_t v,\p_t^2 v,\p_t^3 v$ and $s=4,3,2,1$ respectively. We have
\begin{equation}\label{vhodge}
\begin{aligned}
\|v\|_4&\lesssim\|v\|_0+\|\dive v\|_3+\|\curl v\|_3+|v\cdot N|_{3.5} \\
\|\p_t v\|_3&\lesssim\|\p_t v\|_0+\|\dive \p_t v\|_2+\|\curl \p_t v\|_2+|\p_t v\cdot N|_{2.5} \\
\|\p_t^2 v\|_2&\lesssim\|\p_t^2 v\|_0+\|\dive \p_t^2 v\|_1+\|\curl\p_t^2 v\|_1+|\p_t^2 v\cdot N|_{1.5} \\
\|\p_t^3 v\|_1&\lesssim\|\p_t^3 v\|_0+\|\dive \p_t^3 v\|_0+\|\curl\p_t^3 v\|_0+|\p_t^3 v\cdot N|_{0.5}.
\end{aligned}
\end{equation}
First, the $L^2$-norms are of lower order. The $L^2$-norm of $v$ has been controlled in the energy dissapation. While for $\|\p_tv\|_0,\|\p_t^2v\|_0$ and $\|\p_t^3 v\|_0$, we commute $\p_t$ through $\rho_0\Jk^{-1}\p_t v=(b\cdot\pak)b-\pak Q$ and obtain
\begin{equation}\label{vL2}
\|\p_t v(T)\|_0+\|\p_t^2 v(T)\|_0+\|\p_t^3 v(T)\|_0\lesssim\PP_0+\int_0^T P(E_\kk(t))\dt
\end{equation}

\subsubsection{Boundary estimates: Reduced to tangential estimates}\label{bdryv}

The boundary part of div-curl decomposition can be reduced to the interior tangential estimates by invoking the normal trace Lemma \ref{normaltrace}
\begin{equation}\label{vbdry}
|\TP v^3|_{2.5}\lesssim\|\TP^4 v\|_0+\|\dive v\|_3.
\end{equation}

Similarly we have for $1\leq k\leq 3$
\begin{equation}\label{vtbdry}
|\p_t^k v^3|_{3.5-k}\lesssim\|\TP^{4-k}\p_t^k v\|_0+\|\dive \p_t^k v\|_{3-k}
\end{equation}

\subsubsection{Curl control: Reduced to Lorentz force}\label{curlv}

By the a priori assumption \eqref{small1}, we can estimate the Lagrangian vorticity via Eulerian vorticity plus a small error, for $1\leq k\leq 4$
\begin{equation}\label{curlvv}
\|\curl \p_t^{4-k} v\|_{k-1}^2\lesssim\|\curla \p_t^{4-k} v\|_{k-1}^2+\epsilon^2\|\p_t^{4-k}v\|_k^2
\end{equation}

Taking $\curla$ in $\rho_0\Jk^{-1}\p_t v=(b\cdot\pak)b-\pak Q$, we have 
\begin{equation}\label{curleq}
\rho_0\Jk^{-1}\p_t\curla v=\curla\left((b\cdot\pak)b\right)+\left[\rho_0\Jk^{-1}\p_t,\curla\right]v,
\end{equation} where the commutator only contains first order derivative of $v,\rho,\p_t\eta$.

Taking derivative $\p_t^{4-k}\p^{k-1}$ in \eqref{curleq}, we get the differentiated equation of $\curla v$:
\begin{equation}\label{curleq2}
\begin{aligned}
\rho_0\Jk^{-1}\p_t&(\p^{k-1}\curla \p_t^{4-k} v)=\p_t^{4-k}\p^{k-1}\curl((b\cdot\pak)b)\\
&\underbrace{+\p_t^{4-k}\p^{k-1}\left(\left[\rho_0\Jk^{-1}\p_t,\curla\right]v\right)+[\p_t^{4-k}\p^{k-1},\rho_0\Jk^{-1}\p_t]\curla v+\rho_0\Jk^{-1}\p_t\p^{k-1}([\curla,\p_t^{4-k}]v)}_{F_k}
\end{aligned}
\end{equation}
then taking $L^2$-inner product with $\p^{k-1}\curla\p_t^{4-k} v$, we have
\begin{equation}
\begin{aligned}
&\frac{1}{2}\io\rho_0\Jk^{-1}\left|\p^{k-1}\curla\p_t^{4-k} v(T)\right|^2\dy-\frac{1}{2}\io\rho_0\Jk^{-1}\left|\p^{k-1}\curla\p_t^{4-k} v(0)\right|^2\dy\\
=&\frac{1}{2}\int_0^T\io\p_t\left(\rho_0\Jk^{-1}\right)\left|\p^{k-1}\curla\p_t^{4-k} v\right|^2\dy\dt\\
&+\int_0^T\io\rho_0\Jk^{-1}\p^{k-1}\curla\p_t^{4-k} v\cdot\p^{k-1}\curla\p_t^{4-k}((b\cdot\pak)b)\dy\dt\\
&+\int_0^T\io\rho_0\Jk^{-1}\p^{k-1}\curla\p_t^{4-k} v\cdot F_k\dy\dt\\
\lesssim&\int_0^T\left\|\left(\rho_0\Jk^{-1}\right)\right\|_{L^{\infty}}^2\left\|\p_t^{4-k}v\right\|_k^2\dt+\int_0^T\left\|\p_t^{4-k}v\right\|_k\left\|\p_t^{4-k}\left((b\cdot\pak)b\right)\right\|_k\dt+\int_0^T\left\|\p_t^{4-k}v\right\|_k\left\|F_k\right\|_{L^2}\dt\\
\lesssim&\epsilon T\sup_{0\leq t\leq T}\left\|\p_t^{4-k}\left((b\cdot\pak)b\right)\right\|_k^2+\int_0^T P(E_\kk(t))\dt.
\end{aligned}
\end{equation}Here we used the fact that all terms in $F_k$ are of $\leq 4$ derivatives, and thus can be controlled by $P(E_\kk(t))$.

\subsubsection{Divergence Control: Reduction to full time derivatives}\label{vpreduce}

Before going into detailed proof, we briefly describe the procedure of such reduction. The second and third equations of \eqref{nonlinearkk} give us the following relations if we omit the coefficients $\rho_0\Jk^{-1}$ and $\frac{\Jk R'(q)}{\rho_0}$:
\[
\p\p_t^{k} q\approx \pak \p_t^k q\approx-\p_t^{k+1}v+\p_t^k\left((b\cdot\pak)b-\frac{1}{2}|b|^2\right),
\]
and
\[
\p\p_t^{k}v\xrightarrow{\dive}\p_t^k\diva v+\text{curl}+\text{boundary}\sim\p_t^{k+1}q+\text{curl}+\text{boundary}.
\]
Since the terms containing magnetic field $b$ can be reduced to lower order with the help of magnetic diffusion, the procedure above allows us to control $\dive v$ by $\p_t q$, and control $\p q$ by $\p_t v$. In other words, we are able to trade one spatial derivative by one time derivative, and finally reduce the control to the full time derivative estimates. See the following 
\begin{equation}\label{reducevq}
\begin{aligned}
\p^4 q\xrightarrow{\eqref{nonlinearkk}}\p^3\p_t v\xrightarrow{\dive}\p^2\p_t^2q\xrightarrow{\eqref{nonlinearkk}}\p\p_t^3v &\xrightarrow{\dive} \p_t^4 q \\
\p^3\p_t q\xrightarrow{\eqref{nonlinearkk}}\p^2\p_t^2 v\xrightarrow{\dive}\p\p_t^3 q&\xrightarrow{\eqref{nonlinearkk}}\p_t^4 v.
\end{aligned}
\end{equation}

\bigskip

\noindent\textbf{Step 1: Reduce $q$ to $\p_t v$}

First we investigate $\|\p_t^3 q\|_1$. We take $\p_t^3$ in the second equation in \eqref{nonlinearkk} to get
\[
\p\p_t^3 q=\p_t^3(\pak q)+\nabla_{I-\ak}\p_t^3q=-\p_t^3\left(\rho_0\Jk^{-1}\p_t v\right)+\p_t^3\left((b\cdot\pak)b-\frac{1}{2}\pak|b|^2\right)+\nabla_{I-\ak}\p_t^3q,
\]where we have
 Therefore, $\p_t^3 q$ is estimated as
\begin{equation}\label{qttt1}
\begin{aligned}
\left\|\p_t^3 q\right\|_1\lesssim&\epsilon\left\|\p_t^3 q\right\|_1+\left\|\p_t^3(\rho_0\Jk^{-1}\p_t v)\right\|_0+\left\|\p_t^3\left((b\cdot\pak)b-\frac{1}{2}\pak|b|^2\right)\right\|_0 \\
\lesssim&\epsilon\left\|\p_t^3 q\right\|_1+\left\|\rho_0\Jk^{-1}\right\|_{L^{\infty}}\left\|\p_t^4 v\right\|_0+\PP_0+\int_0^T P(E_\kk(t))\dt+L.O.T.,
\end{aligned}
\end{equation}where $\epsilon>0$ can be chosen suitably small in order for being absorbed by LHS. The $\PP_0+\int_0^T P(E_\kk(t))\dt$ comes from the magnetic field according to Section \ref{nonb}.

Similarly as in the derivation of \eqref{qttt1}, we get the following estimates
\begin{align}
\label{qtt1} \|\p_t^2 q\|_{2}&\lesssim\left\|\rho_0\Jk^{-1}\right\|_{L^{\infty}}\|\p_t^3 v\|_1+\PP_0+\int_0^T P(E_\kk(t))\dt+L.O.T. \\
\label{qt1} \|\p_t q\|_{3}&\lesssim\left\|\rho_0\Jk^{-1}\right\|_{L^{\infty}}\|\p_t^2 v\|_2+\PP_0+\int_0^T P(E_\kk(t))\dt+L.O.T. \\
\label{q1} \|q\|_{4}&\lesssim\left\|\rho_0\Jk^{-1}\right\|_{L^{\infty}}\|\p_t v\|_3+\PP_0+\int_0^T P(E_\kk(t))\dt+L.O.T. 
\end{align}

\bigskip

\noindent\textbf{Step 2: Divergence estimates of $v$}

The Eulerian divergence is $\diva X=\dive X+(\ak^{\mu\alpha}-\delta^{\mu\alpha})\p_\mu X_{\alpha}$, which together with \eqref{small1} implies 
\begin{equation}\label{divgap}
\begin{aligned}
\forall s>2.5&:~\|\dive X\|_{s-1}\lesssim\|\diva X\|_{s-1}+\|I-\ak\|_{s-1}\|X\|_s\lesssim\|\diva X\|_{s-1}+\epsilon\|X\|_s\\
\forall 1\leq s\leq 2.5&:~\|\dive X\|_{s-1}\lesssim\|\diva X\|_{s-1}+\|I-\ak\|_{L^{\infty}}\|X\|_s\lesssim\|\diva X\|_{s-1}+\epsilon\|X\|_s.
\end{aligned}
\end{equation} The $\epsilon$-terms can be absorbed by $\|X\|_s$ on LHS by choosing $\epsilon>0$ sufficiently small. So it suffices to estimate the Eulerian divergence which satisfies $\diva v=-\frac{R'(q)\Jk}{\rho_0}\p_t q$. Taking time derivatives in this equation, we get
\[
\diva \p_t^k v=-\p_t^{k}\left(\frac{R'(q)\Jk}{\rho_0}\p_t q\right)-[\p_t^k,\ak^{\mu\alpha}]\p_\mu v_{\alpha}\approx\frac{R'(q)\Jk}{\rho_0}\p_t^{k+1} q-[\p_t^k,\ak^{\mu\alpha}]\p_\mu v_{\alpha},~~k=0,1,2,3.
\]
Therefore, we have 
\begin{equation}\label{divav}
\begin{aligned}
\|\diva v\|_3&\lesssim\|R'(q)\Jk\|_{L^{\infty}}\|\p_t q\|_3+L.O.T.\\
\|\diva \p_t v\|_2&\lesssim \|R'(q)\Jk\|_{L^{\infty}}\|\p_t^2 q\|_2+L.O.T.\\
\|\diva \p_t^2 v\|_1&\lesssim \|R'(q)\Jk\|_{L^{\infty}}\|\p_t^3 q\|_1+L.O.T.\\
\|\diva \p_t^3 v\|_0&\lesssim \|R'(q)\Jk\|_{L^{\infty}}\|\p_t^4 q\|_0+L.O.T.\\
\end{aligned}
\end{equation}
 Combining \eqref{divgap} and \eqref{divav}, by choosing $\epsilon>0$ in \eqref{divgap} to be suitably small, we know the divergence estimates are all be reduced to one more time derivative of $q$:
\begin{align}
\label{divv}\|\dive v\|_3&\lesssim\epsilon\|v\|_4+\|R'(q)\Jk\|_{L^{\infty}}\|\p_t q\|_3+L.O.T.\\
\label{divvt}\|\dive \p_t v\|_2&\lesssim \epsilon\|v_t\|_3+\|R'(q)\Jk\|_{L^{\infty}}\|\p_t^2 q\|_2+L.O.T.\\
\label{divvtt}\|\dive \p_t^2 v\|_1&\lesssim\epsilon\|\p_t^2v\|_2+ \|R'(q)\Jk\|_{L^{\infty}}\|\p_t^3 q\|_1+L.O.T.\\
\label{divvttt}\|\dive \p_t^3 v\|_0&\lesssim\epsilon\|\p_t^3 v\|_1+ \|R'(q)\Jk\|_{L^{\infty}}\|\p_t^4 q\|_0+L.O.T.
\end{align}

Combining \eqref{qttt1}-\eqref{q1}. \eqref{divv}-\eqref{divvttt} with the previous analysis of curl and boundary estimates, the control of $\|\p_t^{4-k}q\|_k$ and $\|\p_t^{4-k}v\|_0$ are reduced to $\|\p_t^4 v\|_0$ and $\|\p_t^4 q\|_0$ together with the tangential estimates of $v$.

\subsubsection{Tangential space-time derivative estimates}\label{tgtime}

Denote $\dd=\TP$ or $\p_t$. First we consider the case $\dd^4=\p_t^4,\p_t^3\TP,\p_t^2\TP^2,\p_t\TP^3$, i.e., there are at least one time derivative in the four tangential derivatives. 

Direct computation gives
\begin{equation}\label{tg1}
\begin{aligned}
&\frac{1}{2}\io \rho_0\Jk^{-1}\left|\dd^4 v\right|^2\dy\bigg|^T_0\\
=&\int_0^T\io\dd^4(\rho_0\Jk^{-1}\p_tv)\cdot\dd^4 v\dy\dt\\
&+\underbrace{\frac{1}{2}\int_0^T\p_t\left(\rho_0\Jk^{-1}\right)\left|\dd^4 v\right|^2\dy\dt+\int_0^T\io\left[\dd^4,\rho_0\Jk^{-1}\right]\p_tv\cdot\dd^4 v\dy\dt}_{L_1}\\
=&-\int_0^T\io\dd^4(\pak Q)\cdot\dd^4 v\dy\dt+\underbrace{\int_0^T\io\dd^4\left((b\cdot\pak)b\right)\dd^4 v\dy\dt}_{L_2}+L_1,
\end{aligned}
\end{equation}where $L_1$ can be directtly bounded by $\int_0^T P(E_\kk(t))\dt$, and $L_2$ can be controlled by the Lorentz force $$L_2\lesssim\int_0^T\left\|\dd^4 v\right\|_0\left\|\dd^4\left((b\cdot\pak)b\right)\right\|_0\dt\lesssim\epsilon\int_0^T\left\|\dd^4\left((b\cdot\pak)b\right)\right\|_0^2\dt+\int_0^T\left\|\dd^4 v\right\|_0^2\dt.$$

For the first term, we first commute $\dd^4$ with $\pak$, then integrate $\pak$ by parts to get
\begin{equation}\label{tg2}
\begin{aligned}
&-\int_0^T\io\dd^4(\pak Q)\cdot\dd^4 v\dy\dt\\
=&-\int_0^T\io\pak\dd^4 Q\cdot\dd^4 v\dy\dt+\underbrace{\int_0^T\io \left[\dd^4,\ak^{\mu\alpha}\right]\p_\mu Q\cdot\dd^4 v_{\alpha}\dy\dt}_{L_3}\\
=&-\int_0^T\ig\ak^{3\alpha} \underbrace{\dd^4Q}_{=0}\dd^4 v_{\alpha}\dS\dt+\underbrace{\int_0^T\io\dd^4 Q\dd^4(\diva v)\dy\dt}_{K_1}\\
&+\underbrace{\int_0^T\io\dd^4 Q\cdot\left[\dd^4,\ak^{\mu\alpha}\right]\p_{\mu}v_{\alpha}\dy\dt}_{L_4}+\underbrace{\int_0^T\io\p_\mu\ak^{\mu\alpha}\dd^4 Q\dd^4v_{\alpha}}_{L_5}.
\end{aligned}
\end{equation}

Notice that, $\dd^4=\dd^3\p_t$ now contains at least one time derivative, and $\ak\sim\p\ek\cdot\p\ek$, so by the estimates correction term $\psi$, we know the $L^2$-norm of $\dd^4\ak\approx\dd^3\p\p_t\ek\cdot\p\ek+L.O.T.$ can be controlled by $P(E_\kk(t))$, and thus $L_3,L_4$ can be controlled by $\int_0^T P(E_\kk(t))\dt$. The term $L_5$ is also directly bounded by $\int_0^T P(E_{\kk}(t))\dt$.

Next we plug $\diva v=-\frac{\Jk R'(q)}{\rho_0}\p_t q$ and $Q=q+\frac{1}{2}|b|^2$ into $K_1$ to get
\begin{equation}\label{K11}
\begin{aligned}
K_1=&-\int_0^T\io\dd^4 q\dd^4\left(\frac{\Jk R'(q)}{\rho_0}\p_t q\right)\dy\dt-\int_0^T\io\dd^4 \left(\frac{1}{2}|b|^2\right)\dd^4\left(\frac{\Jk R'(q)}{\rho_0}\p_t q\right)\dy\dt\\
=&-\frac{1}{2}\io\frac{\Jk R'(q)}{\rho_0}\left|\dd^4 q(t)\right|^2\dy\bigg|^T_0\underbrace{-\int_0^T\io\dd^4 Q\cdot\left[\dd^4,\frac{\Jk R'(q)}{\rho_0}\right]\p_t q\dy\dt}_{L_6}\\
&\underbrace{-\int_0^T\io\frac{\Jk R'(q)}{\rho_0}\dd^4\left(\frac{1}{2}|b|^2\right)\dd^4\p_t q\dy\dt}_{K_2}.
\end{aligned}
\end{equation}

From the computation above, we find that the energy term $\|\p_t^4 q\|_0^2$ automatically appears if $\dd^4=\p_t^4$. The commutator term $L_6$ can be directly bounded by $\int_0^T P(E_{\kk}(t))\dt$. The term $K_2$ satisfies
\begin{equation}\label{K21}
K_2\lesssim\int_0^T\left\|\dd^4\left(\frac{1}{2}|b|^2\right)\right\|_0\cdot\left\|\dd^4\p_t q\right\|_0\dt\lesssim \epsilon\int_0^T\left\|\dd^4\p_t q\right\|_0^2\dt+\int_0^T P(E_\kk(t))\dt,
\end{equation}therefore we need the energy functional of 4-th time differentiated wave equation of $q$ and elliptic estimates Lemma \ref{GLL} to bound $K_2$. This will also be postponed to Section \ref{noncommon}.

\subsubsection{Tangential spatial derivative estimates: Alinhac good unknowns}\label{tgspace}

When $\dd^4=\TP^4$, the above analysis no longer works due to $[\TP^4,\ak^{\mu\alpha}]\p_{\mu}f$ being out of control. According to the discussion in Section 1.4, we introduce the Alinhac good unknowns. In specific, we replace $\TP^4$ by $\tpl$ due to the special structure of correction term $\psi$. Then for any function $f$ and its corresponding Alinhac good unknown $$\mathbf{f}:=\tpl f-\tpl\ek\cdot\pak f,$$ the following equality holds

\begin{align*}
\tpl(\pak^{\alpha}f)&=\pak^{\alpha}(\tpl f)+(\tpl\ak^{\mu\alpha})\p_\mu f+[\tpl,\ak^{\mu\alpha},\p_{\mu} f] \\
&=\pak^{\alpha}(\tpl f)-\TP\TL(\ak^{\mu\gamma}\TP\p_{\beta}\ek_{\gamma}\ak^{\beta\alpha})\p_\mu f+[\tpl,\ak^{\mu\alpha},\p_{\mu} f] \\
&=\pak^{\alpha}(\tpl f)-\ak^{\beta\alpha}\p_{\beta}\tpl\ek_{\gamma}\ak^{\mu\gamma}\p_\mu f-([\TP\TL,\ak^{\mu\gamma}\ak^{\beta\alpha}]\TP\p_{\beta}\ek_{\gamma})\p_\mu f+[\tpl,\ak^{\mu\alpha},\p_{\mu} f] \\
&=\underbrace{\pak^{\alpha}(\tpl f-\tpl \eta_{\gamma}\ak^{\mu\gamma}\p_\mu f)}_{=\pak^{\alpha}\mathbf{f}}+\underbrace{\tpl \eta_{\gamma}\pak^{\alpha}(\pak^{\gamma} f)-([\TP\TL,\ak^{\mu\gamma}\ak^{\beta\alpha}]\TP\p_{\beta}\ek_{\gamma})\p_\mu f+[\tpl,\ak^{\mu\alpha},\p_{\mu} f] }_{=:C^{\alpha}(f)},
\end{align*} where $[\tpl,g,h]:=\tpl(gh)-\tpl(g)h-g\tpl(h)$.
A direct computation yields that
\begin{align*}
\|\tpl \eta_{\gamma}\pak^{\alpha}(\pak^{\gamma} f)\|_0&\lesssim\|\ek\|_4\|\pak^{\alpha}(\pak^{\gamma} f)\|_{L^{\infty}}\;\\
\|([\TP\TL,\ak^{\mu\gamma}\ak^{\beta\alpha}]\TP\p_{\beta}\ek_{\gamma})\p_\mu f\|_0&\lesssim\|[\TP\TL,\ak^{\mu\gamma}\ak^{\beta\alpha}]\TP\p_{\beta}\ek_{\gamma}\|_0\|f\|_{W^{1,\infty}}\lesssim P(\|\ek\|_4)\|f\|_3 \\
\|[\tpl,\ak^{\mu\alpha},\p_{\mu} f]\|_0&\lesssim P(\|\ek\|_4)\|f\|_4.
\end{align*}
Therefore, Alinhac good unknown enjoys the following important properties:
\begin{equation}\label{alinhaca}
\tpl(\pak^{\alpha}f)=\pak^{\alpha}\mathbf{f}+C^{\alpha}(f)
\end{equation}
with 
\begin{equation}\label{alinhacc}
\|C^{\alpha}(f)\|\lesssim P(\|\ek\|_4)\|f\|_4.
\end{equation}

\bigskip

For \eqref{nonlinearkk}, we define $\VV=\tpl v-\tpl\ek\cdot\pak v$ and $\QQ=\tpl Q-\tpl\ek\cdot\pak Q$ to be the Alinhac good unknowns for $v$ and $Q=q+\frac{1}{2}|B|^2$. Taking $\tpl$ in the second equation of \eqref{nonlinearkk}, we get
\begin{equation}\label{goodeq}
\rho_0\Jk^{-1}\p_t\VV+\pak\QQ=\underbrace{\tpl\left((b\cdot\pak)b\right)+[\rho_0\Jk^{-1},\tpl]\p_t v_{\alpha}+\rho_0\Jk^{-1}\p_t(\tpl\ek\cdot\pak v)+C(Q)}_{\FF},
\end{equation}subject to 
\begin{equation}\label{goodQ}
\QQ=-\tpl\ek_{\beta}\ak^{3\beta}\p_3 Q~~~on~\Gamma.
\end{equation} and 
\begin{equation}\label{gooddiv}
\pak\cdot\VV=\tpl(\diva v)-C^{\alpha}(v_{\alpha})~~~in~\Omega.
\end{equation}

Taking $L^2$ inner product with $\VV$ and time integral, we have
\begin{equation}\label{tg00}
\frac{1}{2}\io \rho_0\Jk^{-1}\left|\VV(T)\right|^2\dy=\frac{1}{2}\io \rho_0\left|\VV(0)\right|^2\dy-\int_0^T\io\pak\QQ\cdot\VV\dy\dt+\int_0^T\io \frac{1}{2} \rho_0\p_t\Jk^{-1}\left|\VV\right|^2+\FF\cdot\VV\dy\dt.
\end{equation}
By \eqref{alinhaca}-\eqref{alinhacc} and direct computation, we know the last term on RHS can be directly controlled:
\begin{equation}\label{goodF}
\int_0^T\io \frac{1}{2}\rho_0\p_t \Jk^{-1}\left|\VV\right|^2+\FF\cdot\VV\dy\dt\lesssim\int_0^T P(E_\kk(t))\dt+\epsilon\int_0^T\left\|\tpl\left((b\cdot\pak)b\right)\right\|_0^2\dt.
\end{equation}

We integrate $\pak$ by parts to get
\begin{equation}\label{IJ1}
\begin{aligned}
&-\int_0^T\io\pak\QQ\cdot\VV\dy\dt=-\int_0^T\io\ak^{\mu\alpha}\QQ\cdot\VV_{\alpha}\dy\dt\\
=&-\int_0^T\ig \QQ(\ak^{3\alpha}\VV_{\alpha})\dS\dt+\int_0^T\io\QQ(\pak\cdot\VV)\dy\dt\underbrace{+\int_0^T\io(\p_\mu\ak^{\mu\alpha})\QQ\VV_{\alpha}\dy\dt}_{J_1}\\
=&\int_0^T\ig\p_3Q\tpl\ek_{\beta}\ak^{3\beta}\ak^{3\alpha}\VV_{\alpha}\dS\dt+\int_0^T\io\QQ\tpl(\diva v)\dy\dt-\int_0^T\io\QQ C^{\alpha}(v_{\alpha})\dy\dt+J_1\\
=:&I_0+I_1+J_2+J_1.
\end{aligned}
\end{equation}

The term $J_1, J_2$ can be directly controlled by $\int_0^T P(E_\kk(t))\dt$ and we omit the details. Let us first investigate $I_1=\int_0^T\io\QQ\tpl(\diva v)\dy\dt$. Invoking $\diva v=-\frac{\Jk R'(q)}{\rho_0}\p_t q$, we have

\begin{equation}\label{IJ2}
\begin{aligned}
&I_1=\int_0^T\io\QQ\tpl(\diva v)\dy\dt=\int_0^T\io\left(\tpl q+\tpl\left(\frac12 |b|^2\right)-\tpl\ek\cdot\pak Q\right)\tpl\left(-\frac{\Jk R'(q)}{\rho_0}\p_t q\right)\\
=&-\int_0^T\io \tpl q\cdot\tpl\left(-\frac{\Jk R'(q)}{\rho_0}\p_t q\right)\dy\dt+\int_0^T\io\tpl\ek\cdot\pak Q\tpl\left(-\frac{\Jk R'(q)}{\rho_0}\p_t q\right)\dy\dt\\
&-\int_0^T\io \tpl \left(\frac12 |b|^2\right)\cdot\tpl\left(-\frac{\Jk R'(q)}{\rho_0}\p_t q\right)\dy\dt\\
=:&I_{11}+I_{12}+I_{13}.
\end{aligned}
\end{equation}

The term $I_{11}$ and $I_{13}$ can be similarly computed as in $K_1$ and $K_{2}$ \eqref{K21}:
\begin{align}
\label{I111} I_{11}&\lesssim-\frac{1}{2}\io\frac{\Jk R'(q)}{\rho_0}\left|\tpl q\right|^2\dy\bigg|^T_0+\int_0^T P(E_\kk(t))\dt,\\
\label{I131} I_{13}&\lesssim \epsilon\int_0^T\left\|\tpl\p_t q\right\|_0^2\dt+\int_0^T P(E_\kk(t))\dt.
\end{align}

One can see that $I_{11}$ has been controlled, while $I_{13}$ requires the control of 5-th order wave equation of $q$ to absorb that $\epsilon$-term. This will again be postponed in Section \ref{noncommon}. For $I_{12}$, we just need to integrate $\p_t$ by parts
\begin{equation}\label{I121}
\begin{aligned}
I_{12}=&-\int_0^T\io\tpl\ek\cdot\pak Q\tpl\left(\frac{\Jk R'(q)}{\rho_0}\p_t q\right)\dy\dt\\
\approx&\io\frac{\Jk R'(q)}{\rho_0}\tpl\ek\cdot\pak Q\tpl q\dy\bigg|^{t=T}_{t=0}+\int_0^T\io \frac{\Jk R'(q)}{\rho_0}\tpl q\p_t(\tpl\ek\cdot\pak Q)\dy\dt \\
\lesssim& \epsilon\io\frac{\Jk R'(q)}{\rho_0}\left|\tpl q(T)\right|^2\dy+\frac{1}{8\epsilon}\left(\|\ek(T)\|_4^4+\|\pak Q(T)\|_{L^{\infty}}^4\right)+\PP_0+\int_0^T P(E_\kk(t))\dt.
\end{aligned}
\end{equation} Here in the last step we use $\epsilon$-Young's inequality to deal with the first term in the second line. The second term can be directly controlled by using the estimates of $\|\p_t\ek\|_4$ and we skip the details.

It remains to control the boundary integral $I_0$. Plugging $\VV_{\alpha}=\tpl v_{\alpha}-\tpl\eta\cdot\pak v_{\alpha}$ into $I_0$, we get
\begin{equation}\label{I001}
I_0=\int_0^T\ig \p_3 Q\ak^{3\alpha}\ak^{3\beta}\tpl\ek_{\beta}(\tpl\p_t\eta_{\alpha}-\tpl\psi-\tpl\ek\cdot\pak v_{\alpha})\dS\dt.
\end{equation}

The first term in \eqref{I001} produces the Taylor sign term contributing to the boundary term in $E_\kk(t)$ after commuting a $\lkk$:

\begin{equation}\label{I01}
\begin{aligned}
&\int_0^T\ig \p_3 Q\ak^{3\alpha}\ak^{3\beta}\tpl \ek_{\beta}\tpl\p_t\eta_{\alpha}\dS\dt \\
=&\int_0^T\ig\p_3 Q\ak^{3\alpha}\ak^{3\beta}\tpl \lkk\eta_{\beta}\tpl\p_t\lkk\eta_{\alpha}\dS\dt\\
&+\int_0^T\ig \left(\tpl\lkk\eta_{\beta}\right)\left(\left[\lkk,\p_3 Q\ak^{3\alpha}\ak^{3\beta}\right]\tpl\p_t\eta_{\alpha}\right)\dS\dt \\
=&\frac{1}{2}\ig \p_3 Q\left|\ak^{3\alpha}\tpl\lkk\eta_{\alpha}\right|^2\dS\bigg|^T_0-\frac{1}{2}\int_0^T\ig\p_t\p_3 Q \left|\ak^{3\alpha}\tpl\lkk\eta_{\alpha}\right|^2\dS \\
&\underbrace{-\int_0^T\ig \p_3 Q \ak^{3\beta}\tpl\lkk\eta_{\beta}\p_t\ak^{3\alpha}\tpl\lkk\eta_{\alpha}\dS\dt}_{I_{01}}\\
&+\underbrace{\int_0^T\ig \left(\tpl\lkk\eta_{\beta}\right)\left(\left[\lkk,\p_3 Q\ak^{3\alpha}\ak^{3\beta}\right]\tpl\p_t\eta_{\alpha}\right)\dS\dt}_{L_7}.
\end{aligned}
\end{equation}
$L_7$ can be directly controlled after integrating $\TP^{1/2}$ by parts
\begin{equation}\label{L7}
\begin{aligned}
L_7=&\int_0^T\ig \left(\TP^{3/2}\TL\lkk\eta_{\beta}\right)\TP^{1/2}\left(\left[\lkk,\p_3 Q\ak^{3\alpha}\ak^{3\beta}\right]\TP(\TP\TL\p_t\eta_{\alpha})\right)\dS \\
\lesssim&\int_0^T\left\|\eta\right\|_4\left|\p_3 Q\ak^{3\alpha}\ak^{3\beta}\right|_{W^{1,\infty}}\left|\TP\TL\p_t\eta_{\alpha}\right|_{1/2} \\
\lesssim&\int_0^T\|\eta\|_4\|Q\|_4\|\p\ak\|_{L^{\infty}} \|\p_t\eta\|_4\lesssim \int_0^TP(E_\kk(t))\dt.
\end{aligned}
\end{equation}

In $I_{01}$, we have $\p_t\ak^{3\alpha}=-\ak^{3\gamma}\p_{\mu}\p_t\ek_{\gamma}\ak^{\mu\alpha}$. Note that $\p_t\eta=v+\psi$. The $\psi$ term can be directly bounded by using the mollifier property, which the contribution of $v$ cannot be bounded directly. Luckily, later on we will see that term can be cancelled together with another higher order term in \eqref{I001} with the help of $\psi$. In specific we have
\begin{equation}\label{I002}
\begin{aligned}
B_1=&\underbrace{\int_0^T\ig \p_3 Q \ak^{3\beta}\tpl\lkk\eta_{\beta}\ak^{3\gamma}\p_3\p_t\ek_{\gamma}\ak^{3\alpha}\tpl\lkk\eta_{\alpha}\dS}_{L_8}\\
&+\underbrace{\int_0^T\ig \p_3 Q \ak^{3\beta}\tpl\lkk\eta_{\beta}\ak^{3\gamma}\TP_i\p_t\lkk^2\psi_{\gamma}\ak^{i\alpha}\tpl\lkk\eta_{\alpha}\dS}_{L_9} \\
&+\underbrace{\int_0^T\ig \p_3 Q \ak^{3\beta}\tpl\lkk\eta_{\beta}\ak^{3\gamma}\TP_i\lkk^2v_{\gamma}\ak^{i\alpha}\tpl\lkk\eta_{\alpha}\dS\dt}_{I_{02}}.
\end{aligned}
\end{equation}

$L_8$ can be directly bounded by the Taylor sign part
\begin{equation}\label{L8}
L_8\lesssim\int_0^T\left|\ak^{3\beta}\tpl\lkk\eta_{\beta}\right|_0^2\cdot\left|\p_3 Q\ak^{3\gamma}\p_e\p_t\ek_{\gamma}\right|_{L^{\infty}}\lesssim \int_0^TP(E_\kk(t))\dt
\end{equation}

$L_9$ can be bounded by using $|\tpl\lkk\eta|\lesssim\kk^{-1/2}|\eta|_{7/2}$ by sacrificing a factor $\kk^{-1/2}$. 
\[
L_9\lesssim\int_0^T \frac{1}{\sqrt{\kk}}\left|\eta\right|_{7/2}\left|\p_3 Q\ak^{3\gamma}\ak^{i\alpha}\right|_{L^{\infty}}\left|\ak^{3\beta}\tpl\lkk\eta_{\beta}\right|_0\left|\TP\tilde{\psi}\right|_{L^{\infty}}\dt
\]
This can be compensated by estimating $|\TP\psi|_{L^{\infty}}$ and $W^{1,4}(\T^2)\hookrightarrow L^{\infty}(\T^2)$. Since $\psi$ removes the zero-frequency part (so the lowest frequency is $\pm 1$ because the frequency on $\T^2$ is discrete), we know $|\TL\psi|_{L^4}$ is comparable to $|\TP\psi|_{W^{1,4}}$. Therefore,
\begin{align*}
\left|\TP\psi\right|_{L^{\infty}}\lesssim&\left|\TL\psi\right|_{L^4}=\left|\mathbb{P}_{\neq 0}\left(\TL\eta_{\beta}\ak^{i\beta}\TP_i\lkk^2 v-\TL\lkk^2\eta_{\beta}\ak^{i\beta}\TP_i v\right)\right|_{L^4}\\
\lesssim&\left|\TL\eta_{\beta}\ak^{i\beta}\TP_i\lkk^2 v-\TL\lkk^2\eta_{\beta}\ak^{i\beta}\TP_i v\right|_{L^4}\\
=&\left|\TL(\eta_{\beta}-\lkk^2\eta_{\beta})\ak^{i\beta}\TP_i\lkk^2 v-\TL\lkk^2\eta_{\beta}\ak^{i\beta}\TP_i(v-\lkk^2 v)\right|_{L^4}\\
\lesssim&\left|\TL\eta_{\beta}-\TL\ek_{\beta}\right|_{L^{\infty}}\left|\ak\right|_{L^{\infty}}\left|\TP\tilde{v}\right|_{0.5}+\left|\TL\ek\right|_{1/2}\left|\ak\right|_{L^{\infty}}\left|\TP(v-\lkk v)\right|_{L^{\infty}}\\
\lesssim&\sqrt{\kk} P(E_{\kk}(t)).
\end{align*}
Therefore we know $L_9$ can be bounded uniformly in $\kk$
\begin{equation}\label{L9}
L_9\lesssim\int_0^T P(E_{\kk}(t))\dt
\end{equation}

The estimate of $I_{02}$ will be postponed after computing the third term in \eqref{I001}, for which we repeat the steps above to get
\begin{equation}\label{I003}
\begin{aligned}
&-\int_0^T\ig \p_3 Q\ak^{3\alpha}\ak^{3\beta}\tpl \ek_{\beta}\tpl\ek\cdot\pak v_{\alpha}~dS \\
=&-\int_0^T\ig \p_3 Q\ak^{3\alpha}\ak^{3\beta}\tpl \ek_{\beta}\tpl\ek_{\gamma}\ak^{3\gamma} \p_3v_{\alpha}~dS\dt\\
& \underbrace{-\int_0^T\ig \p_3 Q\ak^{3\alpha}\ak^{3\beta}\tpl \ek_{\beta}\tpl\ek_{\gamma}\ak^{i\gamma} \TP_iv_{\alpha}~dS\dt}_{I_{03}} \\
=&\int_0^T\ig\left(-\p_3 Q\ak^{3\alpha}\p_3v_{\alpha}\right)\left(\ak^{3\beta}\tpl \ek_{\beta}\right)\left(\ak^{3\gamma}\tpl\ek_{\gamma}\right)\dS\dt+I_{03} \\
\end{aligned}
\end{equation}

The first term can be bounded by Taylor sign after commuting one $\lkk$:
\[
\left|\ak^{3\beta}\tpl\lkk\eta_{\beta}\right|_0\lesssim\left|\lkk\left(\ak^{3\beta}\tpl\lkk\eta_{\beta}\right)\right|_0+\left|\left[\lkk,\ak^{3\beta}\right]\tpl\lkk\eta_{\beta}\right|\lesssim P(E_{\kk}(t)).
\]

Therefore, it remains to control 
\begin{align}
\label{I040} I_{04}&:=-\int_0^T\ig\p_3 Q\ak^{3\alpha}\ak^{3\beta}\tpl\eta_{\beta}\tpl\psi.\\
\label{I020} I_{02}&:=\int_0^T\ig \p_3 Q \ak^{3\beta}\tpl\lkk\eta_{\beta}\ak^{3\gamma}\TP_i\lkk^2v_{\gamma}\ak^{i\alpha}\tpl\lkk\eta_{\alpha}\dS\dt\\
\label{I030} I_{03}&:=-\int_0^T\ig \p_3 Q\ak^{3\alpha}\ak^{3\beta}\tpl \ek_{\beta}\tpl\ek_{\gamma}\ak^{i\gamma} \TP_iv_{\alpha}~dS\dt.
\end{align}

Plugging the expression of $\TL\psi$ into \eqref{I040}, we get
\begin{align}
\label{I041} I_{04}=&-\int_0^T\ig \p_3 Q\ak^{3\alpha}\ak^{3\beta}\tpl \ek_{\beta}\TP^2\left(\TL\eta_{\gamma}\ak^{ir}\TP_i\lkk^2 v_{\alpha}\right)\dS\dt\\
\label{I042} &+\int_0^T\ig \p_3 Q\ak^{3\alpha}\ak^{3\beta}\tpl \ek_{\beta}\TP^2\ek_{\gamma}\ak^{i\gamma} \TP_iv_{\alpha}\dS\dt\\
\label{I043} &+\int_0^T\ig \p_3 Q\ak^{3\alpha}\ak^{3\beta}\tpl \ek_{\beta}\left(\left[\TP^2,\ak^{i\gamma}\TP_i v_{\alpha}\right]\TL\ek_{\gamma}\right)\dS\dt\\
\label{I044} &+\int_0^T\ig \p_3 Q\ak^{3\alpha}\ak^{3\beta}\tpl \ek_{\beta}\TP^2\mathbb{P}_{=0}\left(\TL\eta_{\beta}\ak^{i\beta}\TP_i\lkk^2 v-\TL\lkk^2\eta_{\beta}\ak^{i\beta}\TP_i v\right)\dS\dt.
\end{align}

Clearly, \eqref{I042} exactly cancels with \eqref{I030}, \eqref{I043} can be directly bounded by $\int_0^T P(E_{\kk}(t))\dt$, and \eqref{I044} can be controlled by using Bernstein's inequality $|\mathbb{P}_{\neq 0} f|_2\approx |f|_0$.
\begin{equation}\label{I045}
\begin{aligned}
(3.96)\lesssim&\int_0^T\left|\p_3 Q\ak^{3\alpha}\right|_{L^{\infty}}\left|\ak^{3\beta}\tpl\lkk\eta_{\beta}\right|_0\left|\TL\eta_{\beta}\ak^{i\beta}\TP_i\tilde{v}-\TL\ek_{\beta}\ak^{i\beta}\TP_i v\right|_0\dt\\
\lesssim&\int_0^T P(E_{\kk}(t))\dt.
\end{aligned}
\end{equation} 

In \eqref{I041}, we move one $\lkk$ on $\eta_{\beta}$ to $\eta_{\alpha}$ to cancel $I_{02}$:
\begin{align}
\label{I0411} (3.93)=&-\int_0^T\ig \p_3 Q \ak^{3\beta}\tpl\lkk\eta_{\beta}\left(\ak^{3\alpha}\TP_i\lkk^2v_{\alpha}\right)\left(\ak^{i\gamma}\tpl\lkk\eta_{\gamma}\right)\dS\dt \\
\label{I0412} &-\int_0^T\ig \p_3 Q \ak^{3\beta}\tpl\lkk\eta_{\beta}\left(\left[\lkk,\ak^{3\alpha}\ak^{3\beta}\ak^{ir}\TP_i\lkk^2 v_{\alpha}\right]\tpl\eta_{\gamma}\right)\dS \\
\label{I0413} &-\int_0^T\ig \p_3 Q\ak^{3\alpha}\ak^{3\beta}\tpl \ek_{\beta} \left(\left[\TP^2,\ak^{i\gamma}\TP_i\lkk^2v_{\alpha}\right]\TL\eta_{\gamma}\right)\dS\\
\label{I0414}=&-I_{02}+(3.99)+(3.100).
\end{align}

Summarising \eqref{I001}-\eqref{I003}, \eqref{I041}-I{045} and I{0414}, we are able to control the boundary integral $I_0$ by invoking Taylor sign condition \eqref{taylor1}: $\p_3Q\leq -\frac{c_0}{2}$
\begin{equation}\label{I0}
I_0\lesssim-\frac{c_0}{4}\left|\ak^{3\alpha}\tpl\lkk\eta_{\alpha}\right|_0^2\bigg|^T_0+\int_0^T P(E_\kk(t))\dt
\end{equation}
Combining \eqref{I0} with previous estimates \eqref{tg00}-\eqref{I121}, we finish the estimates of full tangential derivatives by
\begin{equation}\label{tg0}
\begin{aligned}
&\frac{1}{2}\io\rho\left|\TP^4 v\right|_0^2\dy+\frac{1}{2}\io\frac{\Jk R'(q)}{\rho_0}\left|\TP^4 q\right|_0^2\dy+\frac{c_0}{4}\left|\ak^{3\alpha}\tpl\lkk\eta_{\alpha}\right|_0^2\\
\lesssim&\PP_0+\int_0^T P(E_\kk(t))\dt+\epsilon\int_0^T\left\|\tpl\left((b\cdot\pak)b\right)\right\|_0^2+\left\|\tpl\p_t q\right\|_0^2\dt
\end{aligned}
\end{equation}

\subsection{Common control of the higher order heat and wave equations}\label{noncommon}

\subsubsection{Summarizing the previous energy estimates}

Before going to the next step, let us summarize what energy estimates we have gotten. First, from div-curl restimates(\eqref{vhodge}, \eqref{vbdry}, \eqref{vtbdry}, \eqref{curleq2}, \eqref{qttt1}-\eqref{q1}, \eqref{divv}-\eqref{divvttt}) and tangential estimates (\eqref{tg1}-\eqref{K21} and \eqref{tg0}) in Section \ref{nonvp}, we got
\begin{equation}\label{sumvp}
\begin{aligned}
&\sum_{k=0}^4\left\|\p_t^{4-k} v\right\|_k^2+\left\|\p_t^{4-k} q\right\|_{k}^2+\left|\ak^{3\alpha}\tpl\lkk\eta_{\alpha}\right|_0^2\\
\lesssim&\epsilon \left(\sum_{k=0}^3\left\|\p_t^{4-k} v\right\|_k^2+\left\|\p_t^{4-k} q\right\|_{k}^2\right)+\PP_0+ P(E_\kk(T))\int_0^T P(E_\kk(t))\dt\\
&+\epsilon\sum_{k=0}^4\int_0^T\left\|\TP^{k}\p_t^{4-k}\left((b\cdot\pak)b\right)\right\|_0^2+\left\|\TP^{k}\p_t^{4-k}\p_t q\right\|_0^2\dt
\end{aligned}
\end{equation}

The magnetic field $b$ has the following estimates by combining \eqref{b4}, \eqref{bt3btt2}, \eqref{bttt1} and \eqref{b402}:
\begin{equation}\label{sumb}
\sum_{k=0}^4\left\|\p_t^{4-k} b(T)\right\|_k^2\lesssim\PP_0+P(E_\kk(T))\int_0^TP(E_{\kk}(t))\dt+\epsilon H_\kk(T).
\end{equation}

Summing up \eqref{sumvp} and \eqref{sumb}, we get the estimates of $E_\kk(T)$ as 
\begin{equation}\label{sumEkk}
\begin{aligned}
E_\kk(T)\lesssim&\PP_0+P(E_\kk(T))\int_0^TP(E_{\kk}(t))\dt\\
&+\epsilon\left(H_\kk(T)+\sum_{k=0}^4\int_0^T\left\|\TP^{k}\p_t^{4-k}\left((b\cdot\pak)b\right)\right\|_0^2+\left\|\TP^{k}\p_t^{4-k}\p_t q\right\|_0^2\dt\right)
\end{aligned}
\end{equation}

\eqref{sumEkk} shows that we need $H_\kk, W_\kk$ together with Lorentz force to absorb the $\epsilon$-term in \eqref{sumEkk}. From \eqref{bb1}, we know Lorentz force can be controlled by $E_\kk(T)$ plus a term in $H_\kk(T)$
\begin{equation}\label{sumbb}
\sum_{k=0}^4\left\|\p_t^{4-k}((b\cdot\pak)b)\right\|_k^2\lesssim \left\|b\right\|_2^2\left\|\pak \p_t^4 b\right\|_0^2+ P(E_{\kk}(T))+\PP_0+P(E_{\kk}(T))\int_0^TP(E_{\kk}(T))\dt.
\end{equation} 

Also notice that $\p_t q=0$ on $\Gamma$, which allows us to reduce the space-time control of $\p_t q$ to the full time derivative case by using Christodoulou-Lindblad elliptic estimates Lemma \ref{GLL} (See Section \ref{qe}). Therefore, all the estimates of the total energy $\EE_\kk$ in \eqref{EEkk} are reduced to seek for a common control of $W_\kk(T)$ and $H_{\kk}(T)$, the energy functionals of 4-th time-differentiated heat and wave equations, by $\epsilon\EE_{\kk}(T)+\PP_0+P(\EE_\kk(T))\int_0^T P(\EE_\kk(t))\dt$.

\subsubsection{Elliptic estimates of $\p_tq$}\label{qe}

Let us recall the heat equation of $b$ and wave equation of $q$

\begin{equation}\label{heatb0}
\p_t b-\lapak b=(b\cdot\pak)v-b\diva v,
\end{equation}

\begin{equation}\label{waveq0}
\begin{aligned}
&\frac{\Jk R'(q)}{\rho_0}\p_t^2 q-\lapak q\\
=&\frac{1}{2}\lapak|b|^2+\frac{R'(q)}{R}\left(R\p_t v\cdot\pak q\right)+R\p_t\ak^{\mu\alpha}\p_{\mu}v_{\alpha}-\pak^{\alpha} b\cdot\pak b_{\alpha}+\left(\Jk\frac{R'(q)}{\rho_0}-\frac{R\Jk R''(q)}{\rho_0}\right)(\p_t q)^2\\
=&b\cdot\lapak b+\frac{R'(q)}{R}\left(R\p_t v\cdot\pak q\right)+R\p_t\ak^{\mu\alpha}\p_{\mu}v_{\alpha}-\pak^{\alpha} b\cdot\pak b_{\alpha}+|\lapak b|^2+\left(\Jk\frac{R'(q)}{\rho_0}-\frac{R\Jk R''(q)}{\rho_0}\right)(\p_t q)^2\\
=&b\cdot\lapak b+R\p_t\ak^{\mu\alpha}\p_{\mu}v_{\alpha}-\pak^{\alpha} b\cdot\pak b_{\alpha}+|\pak b|^2\\
&+\frac{R'(q)}{R}\left((\pak Q-(b\cdot\pak)b)\cdot\pak q\right)+\left(\Jk\frac{R'(q)}{\rho_0}-\frac{R\Jk R''(q)}{\rho_0}\right)(\p_t q)^2\\
=:&b\cdot\lapak b+w_0.
\end{aligned}
\end{equation}Here we note that all the terms in $w_0$ only contain one derivative!

In \eqref{sumEkk}, there are 4-th order space-time tangential derivatives of $\p_t q$. It seems that we can directly consider the energy functional of $\dd^4$-differentiated wave equation of $q$ \eqref{waveq0}. However, that also requires the control of commutator $[\dd^4,\diva]\pak q$, which is out of control when $\dd^4=\TP^4$. Therefore, we have to use Lemma \ref{GLL} to reduce spatial derivatives to time derivatives.

We start with full spatial derivatives. Since $\|\p_t q\|_4\approx\|\pak \p_t q\|_3$, we have
\begin{equation}\label{qte1}
\|\p_t q\|_4\lesssim P\left(\|\ek\|_3\right)\left(\|\p_t\lapak q\|_2+\left\|[\lapak,\p_t]q\right\|_2+\|\TP\ek\|_3\|\p_t q\|_3\right)\lesssim P(\|\ek\|_3)\|\p_t\lapak q\|_2+ P(E_\kk(T))
\end{equation}

Invoking the $\p_t$-differentiated wave equation, we find that
\[
\p_t\lapak q=\frac{\Jk R'(q)}{\rho_0}\p_t^3 q-b\cdot\p_t\lapak b-\p_t w_0.
\]
Then using the heat equation \eqref{heatb0} to reduce $\lapak b$ to lower order terms, we get
\[
\p_t\lapak q=\frac{\Jk R'(q)}{\rho_0}\p_t^3 q-b\cdot\p_t\left(\p_t b-(b\cdot\pak)v+b\diva v\right)-\p_t w_0.
\]
Plugging this back to \eqref{qte1}, we trade two spatial derivatives by two time derivatives
\begin{equation}\label{qte}
\|\p_t q\|_4\lesssim P(\|\ek\|_3)\|\p_t^3 q\|_2+ P(E_\kk(T)).
\end{equation}

Repeating the same thing for $\|\p_t^2 q\|_3,\|\p_t^3 q\|_2$, we are able to get the following reduction
\begin{align}
\label{qtte}\|\p_t^2 q\|_3\lesssim& P(\|\ek\|_2)\|\p_t^4 q\|_1+ P(E_\kk(T))\approx P(\|\ek\|_2)\|\pak\p_t^4 q\|_0+ P(E_\kk(T)),\\
\label{qttte}\|\p_t^3 q\|_2\lesssim& P(\|\ek\|_2)\|\p_t^5 q\|_0+ P(E_\kk(T)).
\end{align}

From \eqref{qte}-\eqref{qttte}, we are able to reduce the energy estimates of $\p_t q$ to $\|\pak\p_t^4 q\|_0$ and $\|\p_t^5 q\|_0$, which motivates us to consider the 4-th time-differentiated wave equation \eqref{waveq0} together with 4-th time differentiated heat equation \eqref{heatb0}.

\subsubsection{4-th time differentiated heat and wave equation}\label{qb4}

Taking $\p_t^4$ in \eqref{heatb0} and \eqref{waveq0}, we get
\begin{equation}\label{heatb4}
\begin{aligned}
\p_t^5b-\lapak\p_t^4 b=&\p_t^4\left((b\cdot\pak)v-b\diva v\right)+[\p_t^4,\lapak] b\\
=&(b\cdot\pak)\p_t^4v+b\frac{\Jk R'(q)}{\rho_0}\p_t^5 q+\left[\p_t^4,\lapak\right] b+\left[\p_t^4,b\cdot\pak\right]v+\left[\p_t^4,b\frac{\Jk R'(q)}{\rho_0}\right]\p_t q\\
=&:h_5
\end{aligned}
\end{equation}

In $h_5$, there are 5 derivatives of $v$. We can invoke the second equation of \eqref{nonlinearkk} to reduce to $q$ and $B$, e.g., 
\[
\|\p_t^5 v\|_0\lesssim\|\p_t^4((b\cdot\pak)b)\|_0+\|\p_t^4\pak Q\|_0+\cdots,
\]in which the leading order terms above are $\pak\p_t^4 b$ and $\pak\p_t^4 q$, exactly the same as part of energy functional $W_\kk$ and $H_\kk$.

Taking $L_t^2L_x^2$-inner product with $\p_t^5 b$ and integrating by parts, we get
\begin{equation}
\begin{aligned}
RHS=&\int_0^T\io h_5\cdot\p_t^5 b\dy\dt\\
LHS=&\int_0^T\io\left|\p_t^5 b\right|^2\dt-\int_0^T\io\p_t^5 b\cdot\lapak\p_t^4 b\dy\dt\\
=&\int_0^T\io\left|\p_t^5 b\right|^2\dt+\int_0^T\io\p_t\left(\pak\p_t^4b\right)\cdot\left(\pak\p_t^4 b\right)\dy\dt+\int_0^T\io\p_\mu\ak^{\mu\alpha}\left(\p_t^5b\right)\cdot\left(\pak\p_t^4 b\right)\dy\dt\\
&+\int_0^T\io\left(\left[\pak,\p_t\right]\p_t^4 b\right)\cdot\left(\pak\p_t^4 b\right)\dy\dt-\int_0^T\ig\ak^{3\alpha}\p_t^5 b\cdot\left(\pak\p_t^4 b\right)_{\alpha}\dS\dt
\end{aligned}
\end{equation}
 
Since $b=\mathbf{0}$ on the boundary, we know the boundary integral vanishes. The first and second integrals give the energy functional $H_\kk(T)-H_\kk(0)$. Therefore, we have
\begin{equation}\label{Hkk4}
\begin{aligned}
&H_\kk(T)-H_\kk(0)=\int_0^T\io\left|\p_t^5 b\right|^2\dy\dt+\io\left|\pak\p_t^4b\right|^2\dy\bigg|^T_0\\
=&\int_0^T\io h_5\cdot\p_t^5 b\dy\dt-\int_0^T\io\p_\mu\ak^{\mu\alpha}\left(\p_t^5b\right)\cdot\left(\pak\p_t^4 b\right)\dy\dt\\
&-\int_0^T\io\left(\left[\pak,\p_t\right]\p_t^4 b\right)\cdot\left(\pak\p_t^4 b\right)\dy\dt\\
\lesssim&\epsilon\int_0^T\io\left|\p_t^5 b\right|^2\dy\dt+\int_0^T\|h_5\|_0^2\dt+\int_0^T\left\|\pak\p_t^4b\right\|_0\|\ek\|_{4}\dt\\
&+\int_0^T\|\p_t\ak\|_{L^{\infty}}\left\|\p_t^4 b\right\|_0\left\|\pak\p_t^4 b\right\|_0\dt\\
\lesssim&\epsilon\int_0^T\io\left|\p_t^5 b(t)\right|^2\dy\dt+\int_0^T P(E_\kk(t))+\left(H_\kk(t)+W_{\kk}(t)\right)\dt\\
\lesssim&\epsilon H_\kk(T)+\int_0^T P(E_\kk(t))+\left(H_\kk(t)+W_{\kk}(t)\right)\dt.
\end{aligned}
\end{equation}Here $W_\kk$ appears in the last term because $\p_t^5 v$ contains $\pak\p_t^4 q$ which is part of $W_\kk(t)$.

Next we $\p_t^4$ differentiate \eqref{waveq0} to get
\begin{align*}
\frac{\Jk R'(q)}{\rho_0}\p_t^6q-\lapak\p_t^4q=b\cdot\lapak\p_t^4 b+\p_t^4 w_0+\left[b\cdot\lapak,\p_t^4\right]+\left[\p_t^4,\lapak\right]q+\left[\frac{\Jk R'(q)}{\rho_0},\p_t^4\right]\p_t^2 q.
\end{align*}
Then plug the heat equation \eqref{heatb4} $\lapak b=\p_t^5 b-h_5$ to get
\begin{equation}\label{waveq4}
\begin{aligned}
\frac{\Jk R'(q)}{\rho_0}\p_t^6q-\lapak\p_t^4q=&b\cdot\left(\p_t^5 b-h_5\right)+\p_t^4 w_0+\left[b\cdot\lapak,\p_t^4\right]+\left[\p_t^4,\lapak\right]q+\left[\frac{\Jk R'(q)}{\rho_0},\p_t^4\right]\p_t^2 q\\
=:&w_5
\end{aligned}
\end{equation}

Taking $L_t^2L_x^2$ inner product with $\p_t^5 q$, we have
\begin{equation}
\begin{aligned}
RHS=&\int_0^T\io w_5\cdot\p_t^5 q\dy\dt\\
LHS=&\int_0^T\io\frac{\Jk R'(q)}{\rho_0}\p_t^6q\p_t^5 q\dt-\int_0^T\io\p_t^5 q\cdot\lapak\p_t^4 q\dy\dt\\
=&\frac12\io\frac{\Jk R'(q)}{\rho_0}\left|\p_t^5 q\right|^2\dy\bigg|^T_0+\int_0^T\io\p_t\left(\pak\p_t^4q\right)\cdot\left(\pak\p_t^4 q\right)\dy\dt\\
&+\int_0^T\io\p_\mu\ak^{\mu\alpha}\left(\p_t^5q\right)\cdot\left(\pak\p_t^4 q\right)\dy\dt+\int_0^T\io\left(\left[\pak,\p_t\right]\p_t^4 q\right)\cdot\left(\pak\p_t^4 q\right)\dy\dt\\
&-\int_0^T\ig\ak^{3\alpha}\underbrace{\p_t^5 q}_{0}\cdot\left(\pak\p_t^4 q\right)_{\alpha}\dS\dt-\int_0^T\io\frac12\p_t\left(\frac{\Jk R'(q)}{\rho_0}\right)\left|\p_t^5 q\right|^2\dy\dt,
\end{aligned}
\end{equation} and thus we have
\begin{equation}
\begin{aligned}
&W_\kk(T)-W_\kk(0)=\frac{1}{2}\io\frac{\Jk R'(q)}{\rho_0}\left|\p_t^5 q\right|^2\dy\bigg|^T_0+\frac{1}{2}\io\left|\pak\p_t^4 q\right|^2\dy\bigg|^T_0\\
=&\int_0^T\io w_5\cdot\p_t^5 q\dy\dt+\int_0^T\io\frac12\p_t\left(\frac{\Jk R'(q)}{\rho_0}\right)\left|\p_t^5 q\right|^2\dy\dt\\
&-\int_0^T\io\p_\mu\ak^{\mu\alpha}\left(\p_t^5q\right)\cdot\left(\pak\p_t^4 q\right)\dy\dt-\int_0^T\io\left(\left[\pak,\p_t\right]\p_t^4 q\right)\cdot\left(\pak\p_t^4 q\right)\dy\dt.
\end{aligned}
\end{equation}

The term $\|w_5\|_0^2$ can be controlled by $H_\kk(T)+W_\kk(T)+P(E_\kk(T))$, because all the terms in $w_5$ are of $\leq 5$ derivatives, and can be controlled by  either heat energy or wave energy. The precise detailed estimate of $w_5$ is referred to (7.12)-(7.19) in the author's previous work \cite{ZhangCRMHD1}. Therefore, we have
\begin{equation}\label{Wkk4}
W_\kk(T)-W_\kk(0)\lesssim\epsilon \left(W_\kk(T)+H_\kk(T)\right)+\int_0^T H_\kk(t)+W_\kk(t)+P(E_\kk(t))\dt
\end{equation}

Summing up \eqref{Hkk4} and \eqref{Wkk4}, we establish the common control of $H_\kk$ and $W_\kk$
\begin{equation}\label{sumHW}
\left(H_\kk(T)+W_\kk(T)\right)-\left(H_\kk(0)+W_\kk(0)\right)\lesssim\epsilon \left(W_\kk(T)+H_\kk(T)\right)+\int_0^T H_\kk(t)+W_\kk(t)+P(E_\kk(t))\dt
\end{equation}

\subsubsection{Closing the energy estimates}

Combining \eqref{sumEkk}, \eqref{sumbb} and \eqref{sumHW}, we get the inequality
\begin{equation}
\begin{aligned}
\EE_\kk(T)-\EE_\kk(0)=&\left(E_\kk+H_\kk+W_\kk+\sum_{k=0}^4\left\|\p_t^{4-k}((b\cdot\pak)b)\right\|_k^2\right)\bigg|^T_0\\
\lesssim&\epsilon\left(H_\kk(T)+W_\kk(T)\right)+P(E_\kk(T))\int_0^T P(\EE_\kk(t))\dt.
\end{aligned}
\end{equation}
By choosing $\epsilon>0$ sufficiently small, the $\epsilon$-term can be absorbed by LHS, and thus we get
\begin{equation}\label{EEkkt}
\EE_\kk(T)-\EE_\kk(0)\lesssim P(E_\kk(T))\int_0^T P(\EE_\kk(t))\dt.
\end{equation}
Finally, by the Gronwall-type inequality in Tao \cite{tao2006nonlinear}, we know there exists some $T>0$ only depending on $\|v_0\|_4,\|b_0\|_5,\|q_0\|_4,\|\rho_0\|_4$, such that
\begin{equation}\label{nonEE}
\sup_{0\leq t\leq T}\EE_\kk(t)\leq 2\EE_\kk(0).
\end{equation} This finalizes the proof of Proposition \ref{nonlinear1}, i.e., uniform-in-$\kappa$ a priori estimate for the nonlinear approximation system \eqref{nonlinearkk}.

\section{Existence of solutions to the linearized and nonlinear approximation system}\label{linear}

In this section we are going to prove the local existence of the nonlinear $\kk$-approximation system \eqref{nonlinearkk}. The method is standard Picard type iteration. We start with the trivial solution $(\eta^{(0)},v^{(0)},b^{(0)},q^{(0)})=(\eta^{(1)},v^{(1)},b^{(1)},q^{(1)})=(\text{Id},0,0,0)$. Suppose we have already constructed $\{(\eta^{(k)},v^{(k)},b^{(k)},q^{(k)})\}_{0\leq k\leq n}$ for some given $n\in\N^*$. Inductively we define $(\eta^{(n+1)},v^{(n+1)},b^{(n+1)},q^{(n+1)})$ by linearzing \eqref{nonlinearkk} near $a^{(n)}:=[\p\eta^{(n)}]^{-1}$.

\begin{equation}\label{linearn}
\begin{cases}
\p_t \eta^{(n+1)}=v^{(n+1)}+\psi^{(n)} &~~~\text{in } \Omega, \\
\frac{\rho_0}{\Jk^{(n)}}\p_t v^{(n+1)}=(b^{(n)}\cdot\nabla_{\ak^{(n)}}) b^{(n+1)}-\nabla_{\ak^{(n)}} Q^{(n+1)},~~Q^{(n+1)}=q^{(n+1)}+\frac{1}{2}|b^{(n+1)}|^2 &~~~\text{in } \Omega, \\
\frac{\Jk^{(n)} R'(q^{(n)})}{\rho_0}\p_tq^{(n+1)}+\dive_{\ak^{(n)}} v^{(n+1)}=0 &~~~\text{in } \Omega, \\
\p_t b^{(n+1)}-\Delta_{\ak^{(n)}} b^{(n+1)}=(b^{(n)}\cdot\nabla_{\ak^{(n)}}) v^{(n+1)}-b^{(n)}\dive_{\ak^{(n)}} v^{(n+1)} , &~~~\text{in } \Omega, \\
q^{(n+1)}=0,~b^{(n+1)}=\mathbf{0} &~~~\text{on } \Gamma, \\
(\eta^{(n+1)},v^{(n+1)}, b^{(n+1)},q^{(n+1)})|_{\{t=0\}}=(\text{Id},v_0, b_0,q_0).
\end{cases}
\end{equation}

Here $\ak^{(n)}:=(\p\ek^{(n)})^{-1}$ and the correction term $\psi^{(n)}$ is determined by \eqref{psi} with $\eta=\eta^{(n)}, v=v^{(n)}, \ak=\ak^{(n)}$ in that equation. What we need to verify are
\begin{enumerate}
\item System \eqref{linearn} has a (unique) solution $(\eta^{(n+1)},v^{(n+1)}, b^{(n+1)},q^{(n+1)})$ (in a suitable function space).

\item The solution of \eqref{linearn} satifies an energy estimate uniformly in $n$.

\item The approximate solutions $\{(\eta^{(n)},v^{(n)}, b^{(n)},q^{(n)})\}_{n=0}^{\infty}$ converge strongly.
\end{enumerate}

We denote $(\eta^{(n)},v^{(n)},b^{(n)},q^{(n)})$ by $(\er,\vr,\br,\qr)$, and $(\eta^{(n+1)},v^{(n+1)},b^{(n+1)},q^{(n+1)})$ by $(\eta,v,b,q)$ for the simplicity of notations. Then the linearized system \eqref{linearn} becomes 

\begin{equation}\label{linearr}
\begin{cases}
\p_t \eta=v+\psir &~~~\text{in } \Omega, \\
\rho_0\Jr^{-1}\p_t v=(\br\cdot\park) b-\park Q,~~Q=q+\frac{1}{2}|b|^2 &~~~\text{in } \Omega, \\
\frac{\Jrk R'(\qr)}{\rho_0}\p_tq+\divr v=0 &~~~\text{in } \Omega, \\
\p_t b-\lapark b=(\br\cdot\park) v-\br\divr v, &~~~\text{in } \Omega, \\
q=0,~b=\mathbf{0} &~~~\text{on } \Gamma, \\
(\eta,v,b,q)|_{\{t=0\}}=(\text{Id},v_0, b_0,q_0).
\end{cases}
\end{equation}

\subsection{A priori estimates of the linearized approximation system}

We first prove the a priori estimate of the linearized system \eqref{linearn} (or equivalently \eqref{linearr}) because such a priori bound helps us to choose a suitable function space when proving the existence of the linearized system by fixed-point argument.

Define the energy functional for $(\eta^{(n+1)},v^{(n+1)},b^{(n+1)},q^{(n+1)})$ by
\begin{equation}\label{EEn}
\EE^{(n+1)}(T):=E^{(n+1)}(T)+H^{(n+1)}(T)+W^{(n+1)}(T)+\sum_{k=0}^4\left\|\p_t^{4-k}\left( (b^{(n)}\cdot\nabla_{\ak^{(n)}})b^{(n+1)}\right)\right\|_k^2,
\end{equation}where
\begin{align}
\label{En} E^{(n+1)}(T):=&\left\|\eta^{(n+1)}\right\|_4^2+\sum_{k=0}^4\left\|\p_t^{4-k}v^{(n+1)}\right\|_k^2+\sum_{k=0}^4\left\|\p_t^{4-k}b^{(n+1)}\right\|_k^2+\sum_{k=0}^4\left\|\p_t^{4-k}q^{(n+1)}\right\|_k^2\\
\label{Hn} H^{(n+1)}(T):=&\int_0^T\io\left|\p_t^5 b^{(n+1)}\right|^2\dy\dt+\left\|\p_t^4 b^{(n+1)}\right\|_1^2\\
\label{Wn} W^{(n+1)}(T):=&\sum_{k=0}^4\left\|\nabla_{\ak^{(n)}}\p_t^{4-k}q^{(n+1)}\right\|_k^2+\left\|\p_t^{5}q^{(n+1)}\right\|_0^2.
\end{align}

The conclusion is
\begin{prop}\label{EEnn}
Suppose $(\eta^{(n+1)},v^{(n+1)},b^{(n+1)},q^{(n+1)})$ satisfies \eqref{linearn}, then there exists $T_\kk>0$ sufficiently small, independent of $n$. such that 
\begin{equation}\label{linear1}
\sup_{0\leq t\leq T_\kk} \EE^{(n+1)}(t)\leq \PP_0.
\end{equation}
\end{prop}

\begin{rmk}
Compared with $\EE_\kk$ in \eqref{EEkk}, we find that there are extra terms in $W^{(n+1)}(T)$. We note that these extra terms are not needed in the uniform-in-$n$ a priori estimates bacause the elliptic estimates of $\p_t q$ helps us reduce $\|\p_t^{4-k}q\|_{k+1}$ to the $L^2$-norm of $\p_t^5 q$ and $\park\p_t^4 q$, and $\|\park q\|_4$ is not needed. \textbf{However, these terms are needed when we verify the fixed-point argument in the construction of the solution to the linearized system \eqref{linearr}:} The $H^4$-norm of $v$ has to be controlled by $$v(T)=v_0+\int_0^T\left\|\p_t v(t)\right\|_{4}\dt,$$ and thus the $H^4$-norm of $\park Q$ is definitely needed.
\end{rmk}

\subsubsection{Estimates of the frozen coefficients}

We prove Proposition \ref{EEnn} by induction on $n$. When $n=-1,0$, it auotmatically holds for the trivial solution. Suppose the energy bound \eqref{linear1} holds for all $\EE^{(k)}$ with $1\leq k\leq n$. Then we have the following estimates for $\ar,\er,\Jr$.

\begin{lem}\label{coeff}
Let $T\in(0,T_\kk)$. Then there exists some $\epsilon\in(0,1)$ sufficiently small and constant $C>1$ such that
\begin{align}
\label{psir} \psir&\in L_t^{\infty}([0,T];H^4(\Omega)),~~\p_t^l\psir\in  L_t^{\infty}([0,T];H^{5-l}(\Omega)),~~\forall 1\leq l\leq 4; \\
\label{Ida} &\|\Jr-1\|_3+\|\Jrk-1\|_3+\|\text{Id}-\ark\|_3+\|\text{Id}-\ar\|_3\leq\epsilon ;\\
\label{etar} \p_t\er&\in L^{\infty}([0,T];H^4(\Omega)),~~\p_t^{l+1}\er\in L^{\infty}([0,T];H^{5-l}(\Omega)),~~\forall 1\leq l\leq 4; \\
\label{Jr} \Jr,\p_t\Jr&\in L_t^{\infty}([0,T];H^3(\Omega)),~~\p_t^{1+l}\Jr\in L_t^{\infty}([0,T];H^{4-l}(\Omega)),~~\forall 1\leq l\leq 4; \\
\label{weightr} 1/C\leq&\frac{\Jrk R'(\qr)}{\rho_0},\rho_0 \Jrk^{-1}\leq C,~ \p_t^{l}\left(\frac{\Jrk R'(\qr)}{\rho_0},\rho_0 \Jrk^{-1}\right)\in L^{\infty}([0,T];H^{5-l}(\Omega)),~~\forall 1\leq l\leq 5.
\end{align}
\end{lem}

\begin{proof}
\eqref{psir} follows in the same way as Lemma \ref{etapsi}. $\Jr=\det[\p\er]$ and $\ar=[\p\er]^{-1}$ prove \eqref{etar} and \eqref{Jr} because the elements are multilinear functions of $\p\er$. The smallness of $\Jk-1$ and Id$-\ar$ follows from $\Jr=\det[\p\er]$ and
\[
\text{Id}-\ar=-\int_0^T\p_t \ar=\int_0^T\ar:(\p(\vr+\psi^{(n-1)})):\ar\dt
\] and choosing $\epsilon$ (depending on $T_\kk$) sufficiently small. \eqref{weightr} can be similarly proven.
\end{proof}

\subsubsection{Control of $\EE^{(n+1)}$}\label{energyl}

The control of $\EE^{(n+1)}$ follows nearly in the same way as the nonlinear functional $\EE_\kk(T)$ except the extra term $\|\park q\|_4$ and boundary integral in the tangential estimates.

\bigskip

\noindent\textbf{Step 1: Estimates of magnetic field and Lorentz force}

Since $b=\mathbf{0}$ on the boundary and $\divr b=0$ in $\Omega$, we are able to directly mimic the proof in Section \ref{nonb} to get analogues of \eqref{b4}, \eqref{bt3btt2}, \eqref{bttt1} and \eqref{b402}:
\begin{equation}\label{b4l}
\sum_{k=0}^4\left\|\p_t^{4-k} b(T)\right\|_k^2\lesssim\PP_0+P(E^{(n+1)}(T))\int_0^TP(E^{(n+1)}(t))\dt+\epsilon H^{(n+1)}(T)
\end{equation} and an analogue of \eqref{bb4}
\begin{equation}\label{bb4l}
\sum_{k=0}^4\left\|\p_t^{4-k}((\br\cdot\park)b)\right\|_k^2\lesssim \left\|b\right\|_2^2\left\|\park \p_t^4 b\right\|_0^2+ P(E^{(n+1)}(T))+\PP_0+P(E^{(n+1)}(T))\int_0^TP(E^{(n+1)}(t))\dt.
\end{equation}

\bigskip

\noindent\textbf{Step 2: Div-Curl estimates of $v$}

By \eqref{Ida}, we know the div-curl estimates follow in the same way as Section \ref{bdryv}-\ref{vpreduce}. For $1\leq k\leq 4$, we have

\begin{align}
\label{curlvl} \sum_{k=1}^4\frac12\io\rho_0\Jrk^{-1}\left|\curlr\p_t^{4-k} v(t)\right|^2\dy\bigg|^T_0\lesssim&\epsilon T\sup_{0\leq t\leq T}\left\|\p_t^4((\br\cdot\park)b)\right\|_k^2+\int_0^T P(E^{(n+1)}(t))\dt.\\
\label{divvl} \sum_{k=1}^4\left\|\divr \p_t^{4-k} v\right\|_{k-1}\lesssim&\epsilon\sum_{k=1}^4\left\|\p_t^{4-k}v\right\|_{k}+\sum_{k=1}^4\left\|\p_t^{5-k} q\right\|_{k-1}+L.O.T. \\
\label{bdryvl} \sum_{k=1}^4\left|\p_t^{4-k}v^3\right|_{k-1/2}\lesssim&\left\|\TP^{k}\p_t^{4-k} v\right\|_0+\|\dive \p_t^{4-k}v\|_{k-1}.\\
\label{vpreducel} \sum_{k=1}^4\left\|\p_t^{4-k}q\right\|_{k}\lesssim&\sum_{k=1}^4\left\|\p_t^{5-k}v\right\|_{k-1}+\PP_0+\int_0^TP(E^{(n+1)}(t))\dt+L.O.T.
\end{align}

\bigskip

\noindent\textbf{Step 3: Space-Time tangential estimates}

Let $\dd=\TP$ or $\p_t$. When $\dd^4$ contains at least one time derivative, we are able to directly ocmmute $\ark$ with $\dd^4$ because $\p_t\er$ has the same regularity as $\er$, see Lemma \ref{coeff}. Since the boundary condition of \eqref{linearr} is the same as \eqref{nonlinearkk}, we are able to mimic the proof of the nonlinear functional. The result is

\begin{equation}\label{tgvpl}
\begin{aligned}
&\sum_{k=0}^3\left\|\TP^k\p_t^{4-k} v\right\|_k^2+\left\|\TP^k\p_t^{4-k} q\right\|_{k}^2\\
\lesssim&\epsilon \left(\sum_{k=0}^3\left\|\p_t^{4-k} v\right\|_k^2+\left\|\p_t^{4-k} q\right\|_{k}^2\right)+\PP_0+ P(E^{(n+1)}(T))\int_0^T P(E^{(n+1)}(t))\dt\\
&+\epsilon\sum_{k=0}^4\int_0^T\left\|\TP^{k}\p_t^{4-k}\left((b\cdot\pak)b\right)\right\|_0^2+\left\|\TP^{k}\p_t^{4-k}\p_t q\right\|_0^2\dt
\end{aligned}
\end{equation}

\noindent\textbf{Step 4: Tangential spatial derivative estimates}

This part contains a non-trivial boundary integral. In the nonlinear estimates, that boundary term together with Taylor sign condition gives the boundary part of nonlinear functional $E_\kk(T)$. However, here we no longer need Taylor sign condition. Instead, we can sacrifise $1/\kk$ to directly control the boundary integral by using the mollifier property, because the derivative loss is only tangential. 

Similarly as in Section \ref{tgspace}, we rewrite the equation in terms of Alinhac good unknonws. Define the Alinhac good unknowns of $v,Q$ in \eqref{linearr} by 
\[
\VVr:=\tpl v-\tpl\erk\cdot\park v,~~\QQr:=\tpl Q-\tpl\erk\cdot\park Q.
\] Then we take $\tpl$ in the second equation of \eqref{linearr}
\begin{equation}\label{goodeql}
\rho_0\Jrk^{-1}\p_t\VVr+\park\QQr=\FFr:=\tpl((\br\cdot\park)b)+[\rho_0\Jrk^{-1},\tpl]\p_t v+\rho_0\Jrk^{-1}\p_t(\tpl\erk\park v)+\mathring{C}(Q),
\end{equation}subjected to
\begin{equation}\label{QQbdryl}
\QQr=-\tpl\erk_{\beta}\ark^{3\beta}\p_3 Q~~on~\Gamma,
\end{equation} and 
\begin{equation}
\park\cdot\VVr=\tpl(\divr v)-\mathring{C}^{\alpha}(v_\alpha)~~in~\Omega.
\end{equation}

Multiplying $\Jrk\VVr$ and take space-time integral, we have
\begin{equation}
\begin{aligned}
&\frac12\io\rho_0\left|\p_t\VVr(t)\right|^2\dy\bigg|^T_0\\
=&-\int_0^T\io\Jrk\park\QQr\cdot\VVr\dy\dt+\int_0^T\io\FFr\cdot\VVr\dy\dt\\
=&\int_0^T\Jrk\p_3 Q\tpl\erk_{\beta}\ark^{3\beta}\ark^{3\alpha}\VVr_{\alpha}\dS\dt+\int_0^T\Jrk\QQr\tpl\left(\divr v\right)\dy\dt-\int_0^T\io Q\mathring{C}(v)\dy\dt\\
=&:LI_0+LI_1+LJ_1.
\end{aligned}
\end{equation}

Mimicing the estimates \eqref{IJ2}-\eqref{I121}, we are able to control $LI_1$ as
\begin{equation}\label{LI1}
\begin{aligned}
LI_1\lesssim&-\frac{1}{2}\io\frac{\Jrk R'(\qr)}{\rho_0}\left|\tpl q\right|^2\dy\bigg|^T_0+\epsilon\int_0^T\left\|\tpl\p_t q\right\|_0^2\dt\\
&+\PP_0+\int_0^T P(E^{(n+1)}(t))\dt.
\end{aligned}
\end{equation}

For the boundary integral $LI_0$, we integral $\TP^{1/2}$ by parts to get
\begin{align*}
LI_0=&\int_0^T\Jrk\p_3 Q\tpl\erk_{\beta}\ark^{3\beta}\ark^{3\alpha}\VVr_{\alpha}\dS\dt\\
=&\int_0^T\TP^{1/2}\left(\Jrk\p_3 Q\tpl\erk_{\beta}\ark^{3\beta}\ark^{3\alpha}\right)\TP^{-1/2}\VVr_{\alpha}\dS\dt\\
\lesssim&\int_0^T\left(\left|\p_3 Q\right|_{L^{\infty}}\left|\Jrk\ark\right|_{L^{\infty}}^2\left|\tpl\erk\right|_{1/2}+\left|\p_3 Q\Jrk\ark^{3\beta}\ark^{3\alpha}\right|_{W^{\frac12 ,4}}\left|\tpl\erk_\beta\right|_{L^4}\right)\left|\VVr\right|_{-1/2}\dt.
\end{align*}

By the mollifier property $|\tpl\erk|_{1/2}\lesssim\kk^{-1}|\er|_{7/2}$ and $H^{1/2}(\T^2)\hookrightarrow L^4(\T^2)$, we are able to control $LI_0$ by 
\begin{equation}\label{LI0}
LI_0\lesssim\frac{1}{\kk} P\left(\| Q\|_3, \|v\|_4,\|\er\|_4\right).
\end{equation} This together with \eqref{LI0} gives the tangential spatial estimates

\begin{equation}\label{tgl}
\begin{aligned}
&\frac{1}{2}\io\rho_0\left|\TP^4 v\right|_0^2\dy+\frac{1}{2}\io\frac{\Jrk R'(\qr)}{\rho_0}\left|\TP^4 q\right|_0^2\dy\\
\lesssim&\PP_0+\int_0^T P(E^{(n+1)}(t))\dt+\epsilon\int_0^T\left\|\tpl\left((\br\cdot\park)b\right)\right\|_0^2+\left\|\tpl\p_t q\right\|_0^2\dt
\end{aligned}
\end{equation}

\bigskip

\noindent\textbf{Step 5: Elliptic estimates of $q$}

The control of $\|\p_t^{5-k} q\|_{k}$ is the same as Section \ref{qe} so we omit the proof. However, we still need to control $\|\park q\|_4$. By Lemma \ref{GLL}, we have
\begin{equation}\label{ql5}
\|\park q\|_4\lesssim P(\|\erk\|_4)(\|\lapark q\|_3+\|\TP\erk\|_{4}\|q\|_4)\lesssim P(\|\erk\|_4)\|\lapark q\|_3+\frac{1}{\kk} P(E^{(n+1)}(T)).
\end{equation}

Taking $\divr$ in the second equation of \eqref{linearr}, we get the wave equation of $q$
\begin{equation}\label{waveql0}
\begin{aligned}
&\frac{\Jrk R'(\qr)}{\rho_0}\p_t^2 q-\lapark q\\
=&b\cdot\lapark b+R\p_t\ark^{\mu\alpha}\p_{\mu}v_{\alpha}-\left[\divr,(\br\cdot\park)\right]b+|\park b|^2\\
&+\frac{\Jrk R'(\qr)}{\rho_0}\left((\park Q-(\br\cdot\park)b)\cdot\park q\right)+\left(\Jrk\frac{R'(\qr)}{\rho_0}-\Jrk R''(q){\rho_0}\right)(\p_t q)^2\\
=:&b\cdot\lapark b+w_{00}.
\end{aligned}
\end{equation}

So $\|\lapark q\|_3$ can be reduced to $\|b\cdot\lapark b\|_3+\|w_{00}\|_3$. Then $\|\lapark b\|_3$ can again be reduced to the terms with no more than 4 derivatives by the heat equation
\begin{equation}\label{heatbl0}
\p_t b=\lapark b=(\br\cdot\park)v-\br\divr v.
\end{equation} Therefore we are able to reduce $\|\park q\|_4$ to the finished estimates by sacrifising a $1/\kk$ with the help of mollifier.

Combining \eqref{ql5} with the analogue of \eqref{qte}-\eqref{qttte} (replacing $\ak$ by $\ark$), we get
\begin{equation}\label{qle}
\left\|\park q\right\|_4+\sum_{k=2}^4\left\|\p_t^{5-k} q\right\|_{k-1}\lesssim \left(1+\frac{1}{\kk}\right) \left(P(E^{(n+1)}(T))+\left\|\park\p_t^4 q\right\|_0+\left\|\p_t^5q\right\|_0\right)
\end{equation}

\bigskip

\noindent\textbf{Step 6: Common control of higher order heat and wave equation}

We differentiate $\p_t^4$ in \eqref{waveql0} and \eqref{heatbl0} to get

\begin{equation}\label{heatbl4}
\begin{aligned}
\p_t^5b-\lapark\p_t^4 b=&\p_t^4\left((\br\cdot\park)v-\br\divr v\right)+[\p_t^4,\lapark] b\\
=&(\br\cdot\park)\p_t^4v+b\frac{\Jrk R'(q)}{\rho_0}\p_t^5 q+\left[\p_t^4,\lapark\right] b+\left[\p_t^4,\br\cdot\park\right]v+\left[\p_t^4,b\frac{\Jrk R'(q)}{\rho_0}\right]\p_t q\\
=&:h_{55}
\end{aligned}
\end{equation}
 and
\begin{align*}
\frac{\Jrk R'(\qr)}{\rho_0}\p_t^6q-\lapark\p_t^4q=b\cdot\lapark\p_t^4 b+\p_t^4 w_0+\left[b\cdot\lapark,\p_t^4\right]+\left[\p_t^4,\lapark\right]q+\left[\frac{\Jrk R'(\qr)}{\rho_0},\p_t^4\right]\p_t^2 q.
\end{align*}
Then plug the heat equation \eqref{heatbl4} $\lapark b=\p_t^5 b-h_{55}$ to get
\begin{equation}\label{waveql4}
\begin{aligned}
\frac{\Jrk R'(q)}{\rho_0}\p_t^6q-\lapark\p_t^4q=&b\cdot\left(\p_t^5 b-h_{55}\right)+\p_t^4 w_0+\left[b\cdot\lapark,\p_t^4\right]+\left[\p_t^4,\lapark\right]q+\left[\frac{\Jrk R'(q)}{\rho_0},\p_t^4\right]\p_t^2 q\\
=:&w_{55}
\end{aligned}
\end{equation}

Similarly as in Section \ref{qb4}, we are able to get a common control of the energy functional of these 2 equations. Define
\[
\w{W}^{(n+1)}:=\left\|\p_t^5 q\right\|_0^2+\left\|\park\p_t^4 q\right\|_0^2,
\] then we have the analogue of \eqref{sumHW}
\begin{equation}\label{sumHWl}
\begin{aligned}
&\left(H^{(n+1)}(T)+\w{W}^{(n+1)}(T)\right)-\left(H^{(n+1)}(0)+\w{W}^{(n+1)}(0)\right)\\
\lesssim&\epsilon \left(H^{(n+1)}(T)+\w{W}^{(n+1)}(T)\right)+\int_0^T H^{(n+1)}(t)+\w{W}^{(n+1)}(t)+P(E^{(n+1)}(t))\dt
\end{aligned}
\end{equation}

\bigskip

\noindent\textbf{Step 7: Finalizing the a priori estimates}

Summing up \eqref{b4l}, \eqref{bb4l}, \eqref{curlvl}, \eqref{divvl}, \eqref{bdryvl}, \eqref{vpreducel}, \eqref{tgvpl}, \eqref{tgl}, \eqref{qle} and \eqref{sumHWl}, we get 
\[
\EE^{(n+1)(T)}-\EE^{(n+1)(0)}\lesssim_{1/\kk} \epsilon\EE^{(n+1)(T)}+P(E^{(n+1)}(T))+\int_0^T P(\EE^{(n+1)}(t))\dt.
\] By the Gronwall inequality, we are able to find some $T_\kk>0$ independent of $n$, such that 
\[
\sup_{0\leq t\leq T_\kk} \EE^{(n+1)}(t)\leq 2\EE^{(n+1)}(0)\lesssim \PP_0.
\] This finalizes the proof of Proposition \ref{EEnn}.

\subsection{Construction of the solution to the linearized approximation system}

This part provides a fixed-point argument of constructing the solution to the linearized system \eqref{linearr}
\[
\begin{cases}
\p_t \eta=v+\psir &~~~\text{in } \Omega, \\
\rho_0\Jrk^{-1}\p_t v=(\br\cdot\park) b-\park Q,~~Q=q+\frac{1}{2}|b|^2 &~~~\text{in } \Omega, \\
\frac{\Jrk R'(\qr)}{\rho_0}\p_tq+\divr v=0 &~~~\text{in } \Omega, \\
\p_t b-\lapark b=(\br\cdot\park) v-\br\divr v, &~~~\text{in } \Omega, \\
q=0,~b=\mathbf{0} &~~~\text{on } \Gamma, \\
(\eta,v,b,q)|_{\{t=0\}}=(\text{Id},v_0, b_0,q_0).
\end{cases}
\]

Define the norm $\|\cdot\|_{\XX^r}$ by
\[
\left\|f\right\|_{\XX^r}^2:=\sum_{m=0}^r\sum_{k+l=m}\left\|\p_t^k\p^lf\right\|_0^2
\] and a Banach space on $[0,T]\times\Omega$
\[
\XX(M,T):=\left\{\left(\xi,w,h,\pi\right)\bigg|\left(\xi,w,h,\pi\right)\bigg|_{t=0}=\left(\text{Id},v_0,b_0,q_0\right),\left\|\left(\xi,w,h,\pi\right)\right\|_{\XX}\leq M\right\}
\]where
\[
\left\|\left(\xi,w,h,\pi\right)\right\|_{\XX}^2:=\left\|\left(\xi,\p_t\xi,w,h,\park h,\pi,\p_t\pi,\park\pi\right)\right\|_{L_t^{\infty}\XX^4}^2+\|\p_t^5 h\|_{L_t^2L_x^2}^2
\]

Next we define the solution map
\begin{align*}
\Xi: \XX(M,T)&\to\XX(M,T)\\
\left(\xi,w,h,\pi\right)&\mapsto (\eta,v,b,q)
\end{align*} as follows:
\begin{enumerate}
\item Define $\eta$ by $\p_t\eta=w+\psir$ with $\eta(0)=$Id
\item Define $v$ by $\rho_0\Jrk^{-1}\p_t v:=(\br\cdot\park)h-\park(\pi+\frac{1}{2}|h|^2).$ with $v(0)=v_0$
\item Define $b,q$ by the coupled system of heat equation and wave equation
\end{enumerate}
\begin{equation}
\begin{cases}
\p_t b-\lapark b=(\br\cdot\park)v-\br\divr v\\
b|_{\Gamma}=\mathbf{0}\\
b(0)=b_0
\end{cases}
\end{equation}
and
\begin{equation}
\begin{cases}
R'(\qr)\p_t^2 q-\lapark q=&\lapark\left(\frac12|b|^2\right)-\left[\divr,(\br\cdot\park)\right]b\\
&+\rho_0\Jrk^{-1}\p_t\ark^{\mu\alpha}\p_\mu v_{\alpha}+\ark^{\mu\alpha}\p_\mu(\rho_0\Jrk^{-1})\p_t v_{\alpha}-\Jrk^{-1}\p_t\left(\Jrk R'(\qr)\right)\p_t q\\
q|_{\Gamma}=0,&\\
(q(0),\p_tq(0))=(q_0,q_1).&
\end{cases}
\end{equation}

We need to verify the following things to prove the existence and uniqueness of the system \eqref{linearr}.
\begin{enumerate}
\item The image of $\XX(M,T)$ under $\Xi$ still lies in $\XX(M,T)$.
\item $\Xi$ is a contraction on $\XX(M,T)$.
\end{enumerate}

We first prove $\Xi$ is a self-mapping of $\XX(M,T)$. The velocity is directly controlled by $$\rho_0\Jrk^{-1}\p_t v:=(\br\cdot\park)h-\park(\pi+\frac{1}{2}|h|^2).$$
\begin{equation}
\begin{aligned}
\|\p_t^{4-k} v(T)\|_k^2\lesssim&\|\p_t^{4-k} v(0)\|_0^2+\int_0^T\left\|\p_t^{4-k}\left((\br\cdot\park)h-\park(\pi+\frac{1}{2}|h|^2)\right)\right\|_k^2 \\
\lesssim&\|\p_t^{4-k} v(0)\|_0^2+\int_0^T\left\|\park h\right\|_{\XX^4}^2+\left\|\park \pi\right\|_{\XX^4}^2\dt
\end{aligned}
\end{equation}And thus the bound for $\|\p_t\eta\|_{\XX^4}$ and $\|\eta\|_{\XX^4}$ directly follows.

Next we control $\|b\|_{\XX^4}$ by elliptic estimates as in Section \ref{nonb}. For example
\[
\|b\|_4\approx\|\park b\|_3\lesssim P(\|\erk\|_3)\left(\|\lapak b\|_2+\|\TP\erk\|_3\|b\|_3\right)
\]. Then invoking $\lapark b=\p_t b-(\br\cdot\park)v+\br\divr v$ to get
\[
\|b\|_4\lesssim P(\|\erk\|_3)\left(\|(\br\cdot\park)v\|_2+\|\br\divr v\|_2+\|\TP\erk\|_3\|b\|_3\right)\lesssim P(\|\erk\|_3)\left((\|b\|_{2}+\|\br\|_2)\|v\|_3+\|\TP\erk\|_3\|b\|_3\right).
\] Combining the estimates of $v$ above, we are able to write
\[
\|b\|_4\lesssim P(\|\erk\|_3)\|\TP\erk\|_3\|b\|_3+\|v(0)\|_{\XX^3}+\int_0^T\left\|\park h\right\|_{\XX^3}+\left\|\park \pi\right\|_{\XX^3}\dt
\] Then one can repeat the same steps for $\|b\|_3$ to get
\begin{equation}
\|b\|_4\lesssim \PP_0+ P(\|\er\|_3)\int_0^T\left\|\park h\right\|_{\XX^3}+\left\|\park \pi\right\|_{\XX^3}\dt
\end{equation} 
Similar estimates hold for $\|\p_t^{4-k} b\|_{k}$ for $1\leq k\leq 4$, while $\|\p_t^4 b\|_0^2$ is again reduced to $\int_0^T\|\p_t^5 b\|_0^2\dt$ as before.

One can mimic the proof above to estimate the space-time derivative of $\park b$ or $\p_t b$. One exception is $\|\park b\|_4$, for which we have to use the mollifier property.
\[
\|\park b\|_4\lesssim P(\|\er\|_4)\left(\|\lapark b\|_3+\frac{1}{\kk}\|\er\|_3\|b\|_4\right).
\] Again, invoking the heat equation and the $\XX^4$ estimates of $v$, we get
\[
\|\park b\|_4\lesssim \PP_0+ P(\|\er\|_4)\int_0^T\left\|\park h\right\|_{\XX^4}+\left\|\park \pi\right\|_{\XX^4}\dt
\]
Similar estimates holds for the space-time derivatives except $\|\p_t^5b\|_{L_t^2L_x^2}$ and $\|\park\p_t^4 b\|_0$.
\begin{equation}
\sum_{k=1}^{4}\left\|\park\p_t^{4-k} b\right\|_k^2\lesssim\PP_0+ P(\|\er\|_4)\int_0^TP\left(\left\|\park h\right\|_{\XX^4},\left\|\park \pi\right\|_{\XX^4}\right)\dt.
\end{equation}

Analogously, we can apply the elliptic estimates and wave equation to $q$ in order to reduce the estimates to the full time derivatives. For example
\[
\|q\|_4\approx\|\park q\|_3\lesssim P(\|\erk\|_3)\left(\|\lapark q\|_2+\|\TP\erk\|_3\|q\|_3\right)
\]
Invoking the wave equation and heat equation $$\lapak q=\p_t^2 q-\lapak(1/2|b|^2)+\cdots=\p_t^2 q-\p_t b-(\br\cdot\park)v+\br\divr v+\cdots,$$ we are able to reduce $\|\lapark q\|_2$ to $\|\p_t^2 q\|_2$ plus the terms with $\leq 3$ derivatives. Repeat the steps above, we are able to reduce $\|q\|_{\XX^4}$ to $\|\p_t^4 q\|_0$ and $\|\p_t^3q\|_1$. Similarly,
\[
\|\park q\|_4\lesssim P(\|\erk\|_4)\left(\|\lapark q\|_3+\kk^{-1}\|\er\|_4\|q\|_4\right)
\] Therefore, the control of $\|\park q\|_{\XX^4}$ and $\|\p_t q\|_{\XX^4}$ are reduced to $\|\p_t^5q\|_{0}$ and $\|\park\p_t^4 q\|_{0}$.

The final step is to seek for a common control of 4-th order time-differetiated heat and wave equations. The proof is the same as in Section \ref{qb4} and step 6 in Section \ref{energyl}. The only thing we would like to remark here is that there are terms like $\p_t^5 v$ and $\p\p_t^4 v$ appearing in the time integral of the source term. In this case, we can invoke the equation of $v$ to eliminate one time derivative and reduce to the $\XX^4$ norm of $\pak\pi$ and $(\br\cdot\park)h$.  
\begin{equation}\label{qbqb}
\begin{aligned}
&\int_0^T\io\left|\p_t^5 b\right|^2\dy\dt+\left(\left\|\park \p_t^4 b\right\|_0^2+\left\|\p_t^5 q\right\|_0^2+\left\|\park \p_t^4 q\right\|_0^2\right)\bigg|^T_0\\
\lesssim&\epsilon \int_0^T\io\left|\p_t^5 b\right|^2\dy\dt+\int_0^T P\left(\left\|\p_t^5 b\right\|_{L_t^2L_x^2},\left\|\left(v,\park b,b,\p_t q,q,\park q\right)\right\|_{\XX^4}\right)\dt\\
&+\PP_0+\int_0^T P\left(\left\|\park h\right\|_{\XX^4},\left\|\park \pi\right\|_{\XX^4},\left\|\p_t\pi\right\|_{\XX^4}\right)\dt
\end{aligned}
\end{equation}By choosing $\epsilon>0$ sufficiently small, we can absorb the $\epsilon$-term to LHS. 

Summarizing these steps above, we find that, there exists some $T_\kk>0$ sufficiently small and $M$ chosen suitably large, such that
\begin{equation}
\left\|\left(\eta,\p_t\eta,b,\park b,q,\p_tq,\park q\right)\right\|_{\XX^4}<\infty.
\end{equation}

Next we prove $\Xi$ is a contraction. Pick any $\left(\xi_i,w_i,h_i,\pi_i\right)\mapsto\left(\eta_i,v_i,b_i,q_i\right)$ and define $[f]:=f_1-f_2$. Then by the linearity of the equations above, we know $\left([\eta],[v],[b],[q]\right)$ satisfies the same equation with $\left(\xi,w,h,\pi\right)$ replaced by $\left([\xi],[w],[h],[\pi]\right)$ and zero initial data. Thus  $\left([\eta],[v],[b],[q]\right)$ satisfies
\[
\left\|\left([\eta],[v],[b],[q]\right)\right\|_{\XX}\lesssim_{\kk^{-1}}\int_0^T P\left(\left\|\left([\xi],[w],[h],[\pi]\right)\right\|_{\XX}\right)\dt.
\] Choosing a suitably small $T_{\kk}>0$ such that
\[
\left\|\left([\eta],[v],[b],[q]\right)\right\|_{\XX}\leq \frac{1}{2} \left\|\left([\xi],[w],[h],[\pi]\right)\right\|_{\XX},
\] we know $\Xi$ is indeed a contraction. By Contraction Mapping Theorem, $\Xi$ has a unique fixed point $\left(\eta,v,b,q\right)$, and thus the local existence and uniqueness of the solution to the linearized equation \eqref{linearr} is established.
 
\subsection{Iteration to the nonlinear approximation system}

For each $n$, we have already established the local existence and uniqueness of solution $(\eta^{(n+1)},v^{(n+1)},b^{(n+1)},q^{(n+1)})$ to the $n$-th linearized approximation system
\[
\begin{cases}
\p_t \eta^{(n+1)}=v^{(n+1)}+\psi^{(n)} &~~~\text{in } \Omega, \\
\frac{\rho_0}{\Jk^{(n)}}\p_t v^{(n+1)}=(b^{(n)}\cdot\nabla_{a^{(n)}}) b^{(n+1)}-\nabla_{\ak^{(n)}} Q^{(n+1)},~~Q^{(n+1)}=q^{(n+1)}+\frac{1}{2}|b^{(n+1)}|^2 &~~~\text{in } \Omega, \\
\frac{\Jk^{(n)} R'(q^{(n)})}{\rho_0}\p_tq^{(n+1)}+\dive_{\ak^{(n)}} v^{(n+1)}=0 &~~~\text{in } \Omega, \\
\p_t b^{(n+1)}-\Delta_{\ak^{(n)}} b^{(n+1)}=(b^{(n)}\cdot\nabla_{a^{(n)}}) v^{(n+1)}-b^{(n)}\dive_{\ak^{(n)}} v^{(n+1)} , &~~~\text{in } \Omega, \\
q^{(n+1)}=0,~b^{(n+1)}=\mathbf{0} &~~~\text{on } \Gamma, \\
(\eta^{(n+1)},v^{(n+1)}, b^{(n+1)},q^{(n+1)})|_{\{t=0\}}=(\text{Id},v_0, b_0,q_0).
\end{cases}
\] This part shows the Picard-type iteration of the sequence $\{(\eta^{(n)},v^{(n)},b^{(n)},q^{(n)})\}_{n\in\N}$ which gives a subsequential limit $(\eta,v,b,q)$ converging in $H^3$-norm. Such limit $(\eta,v,b,q)$ exactly solves the nonlinear $\kk$-approximation problem \eqref{nonlinearkk}.

Define $[\eta]^{(n)}:=\eta^{(n+1)}-\eta^{(n)},~[v]^{(n)}:=v^{(n+1)}-v^{(n)},~[b]^{(n)}:=b^{(n+1)}-b^{(n)},~[q]^{(n)}:=q^{(n+1)}-q^{(n)},$ and $[a]^{(n)}:=a^{(n)}-a^{(n-1)},~[A]^{(n)}:=A^{(n)}-A^{(n-1)},~[\psi]^{(n)}:=\psi^{(n)}-\psi^{(n-1)}$. Then these quantities satisfy the following system consisting of:

The equation of momentum
\begin{equation}\label{diffvn}
\begin{aligned}
\rho_0\p_t [v]^{(n)}=&\left(b^{(n)}\cdot\nabla_{\Ak^{(n)}}\right)\left[b\right]^{(n)}-\nabla_{\Ak^{(n)}}[Q]^{(n)}\\
&+b^{(n)}\cdot\nabla_{[\Ak]^{(n)}}b^{(n)}+[b]^{(n-1)}\cdot\nabla_{\Ark^{(n-1)}}b^{(n)}-\nabla_{\Ak^{(n)}}Q^{(n)}.
\end{aligned}
\end{equation}

Continuity equation:
\begin{equation}\label{diffqn}
r^{(n)}\p_t [q]^{(n)}+\dive_{\ak^{(n)}}[v]^{(n)}=-\dive_{[\ak]^{(n)}}v^{(n)}+[r]^{(n)}\p_t q^{(n)},
\end{equation}here $r^{(n)}:=\Jrk^{(n)}R'(q^{(n)})/\rho_0$.

Equation of magnetic field:
\begin{equation}\label{diffbn}
\begin{aligned}
\p_t [b]^{(n)}-\Delta_{\ak^{(n)}}[b]^{(n)}=&(b^{(n)}\cdot\nabla_{\ak^{(n)}})[v]^{(n)}-b^{(n)}\dive_{\ak^{(n)}}[v]^{(n)}\\
&+b^{(n)}\cdot\nabla_{[\ak]^{(n)}}v^{(n)}-b^{(n)}\dive_{[\ak]^{(n)}}v^{(n)}\\
&+[b]^{(n-1)}\cdot\nabla_{\ark^{(n-1)}}v^{(n)}-[b]^{(n-1)}\dive_{\ark^{(n-1)}}v^{(n)}\\
&+\dive_{\ak^{(n)}}\left(\nabla_{[\ak]^{(n)}}b^{(n)}\right)+\dive_{[\ak]^{(n)}}\left(\nabla_{\ark^{(n-1)}}b^{(n)}\right).
\end{aligned}
\end{equation}


The initial data of $([\eta],[v],[b],[q])=(\mathbf{0},\mathbf{0},\mathbf{0},0)$. The boundary conditions are
\begin{equation}\label{diffbdry}
[b]^{(n)}=\mathbf{0},[q]^{(n)}=0.
\end{equation}

Define the energy functional
\begin{equation}\label{diffEEn}
[\EE]^{(n)}(T):=[E]^{(n)}(T)+[H]^{(n)}(T)+[W]^{(n)}(T)+\sum_{k=0}^3\left\|\p_t^{3-k}\nabla_{\ak^{(n)}}[b]^{(n)}\right\|_k^2,
\end{equation}where
\begin{align}
[E]^{(n)}(T):=&\sum_{k=0}^3\left(\left\|\p_t^{3-k}[v]^{(n)}\right\|_k^2+\left\|\p_t^{3-k}[b]^{(n)}\right\|_k^2+\left\|\p_t^{3-k}[q]^{(n)}\right\|_k^2\right),\\
[H]^{(n)}(T):=&\int_0^T\left\|\p_t^4 [b]^{n}\right\|_0^2\dt+\left\|\p_t^3\nabla_{\ak^{(n)}}[b]^{(n)}\right\|_0^2,\\
[W]^{(n)}(T):=&\left\|\p_t^4[q]^{(n)}\right\|_0^2+\left\|\p_t^3\nabla_{\ak^{(n)}}[q]^{(n)}\right\|_0^2.
\end{align}

The conclusion is 
\begin{prop}\label{diffconv}
For $n$ sufficiently large and $T_{\kk}>0$ suitably small, we have that $\forall T\in[0,T_{\kk}]$
\[
[E]^{(n)}(T)\leq\frac{1}{4}\left([E]^{(n-1)}(T)+[E]^{(n-2)}(T)\right).
\]
\end{prop}
\begin{flushright}
$\square$
\end{flushright}

By Proposition \ref{diffconv}, we know $[E]^{(n)}\leq\frac{1}{2^n} P_\kk(\PP_0)$, and thus yields the limit for each fixed $\kk>0$: $$\left(\eta^{(n)},v^{(n)},b^{(n)},q^{(n)}\right)\xrightarrow{\text{converge strongly}}(\eta(\kk),v(\kk),b(\kk),q(\kk))~~~as~n\to\infty.$$ Such limit exactly solves the nonlinear approximation system \ref{nonlinearkk}.

\begin{cor}\label{lwpkk}
The limit $(\eta(\kk),v(\kk),b(\kk),q(\kk))$ gotten in Proposition \ref{diffconv} is the unique strong solution to the nonlinear approximation system \eqref{nonlinearkk} and satisfies the energy estimates in $[0,T_\kk]$
\[
\sup_{0\leq T\leq T_\kk}\w{\EE}_\kk(T)\leq 2\left(\|v_0\|_4^2+\|b_0\|_5^2+\|q_0\|_4^2\right),
\]where
\begin{equation}\label{EEh}
\w{\EE}_\kk(T):=\w{E}_\kk(T)+\w{H}_\kk(T)+\w{W}_\kk(T)+\sum_{k=0}^4\left\|\p_t^{4-k}\left((b(\kk)\cdot\pak)b(\kk)\right)\right\|_k^2,
\end{equation} and
\begin{align}
\label{Eh} \w{E}_\kk(T):=&\left\|\eta\right\|_4^2+\sum_{k=0}^4\left\|\p_t^{4-k}v(\kk)\right\|_k^2+\sum_{k=0}^4\left\|\p_t^{4-k}b(\kk)\right\|_k^2+\sum_{k=0}^4\left\|\p_t^{4-k}q(\kk)\right\|_k^2\\
\label{Hh} \w{H}_\kk(T):=&\int_0^T\io\left|\p_t^5 b(\kk)\right|^2\dy\dt+\left\|\p_t^4 b(\kk)\right\|_1^2\\
\label{Wh} \w{W}_\kk(T):=&\sum_{k=0}^4\left\|\pak\p_t^{4-k}q(\kk)\right\|_k^2+\left\|\p_t^{5}q(\kk)\right\|_0^2.
\end{align}
\end{cor}
\begin{flushright}
$\square$
\end{flushright}

The proof process is nearly the same as in the a priori estimates part, so we do not write all details here but still state the main steps. 

\subsubsection*{Step 1: Correction Terms}

First we estimate the coefficients and correction terms.

 $[\psi]^{(n)}$ satisfies $-\Delta[\psi]^{(n)}=0$ with the boundary condition
\begin{align*}
[\psi]^{(n)}=\TL^{-1}\mathbb{P}_{\neq 0}\bigg(&\TL[\eta]^{(n-1)}_{\beta}\ak^{(n)i\beta}\TP_i\lkk^2v^{(n)}+\TP\eta^{(n-1)}_{\beta}[\ak]^{(n)i\beta}\TP_i\lkk^2v^{(n)}+\TP\eta^{(n-1)}_{\beta}\ak^{(n-1)i\beta}\TP_i\lkk^2[v]^{(n-1)} \\
&-\TL\lkk^2[\eta]^{(n-1)}_{\beta}\ak^{(n) i\beta}\TP_i v^{(n)}-\TL\lkk^2\eta^{(n-1)}_{\beta}[\ak]^{(n) i\beta}\TP_i v^{(n)}-\TL\lkk^2\eta^{(n-1)}_{\beta}\ak^{(n-1) i\beta}\TP_i [v]^{(n-1)}\bigg).
\end{align*}
By the standard elliptic estimates, we have the control for $[\psi]^{(n)}$
\begin{equation}\label{psig3}
\|[\psi]^{(n)}\|_3^2\lesssim |[\psi]^{(n)}|_{2.5}\lesssim \PP_0\left(\|[\eta]^{(n-1)}\|_3^2+\|[v]^{(n-1)}\|_2^2+\|[\ak]^{(n)}\|_1^2\right).
\end{equation}

On the other hand, we have
\begin{align*}
[a]^{(n)\mu\nu}(T)&=\int_0^T\p_t(a^{(n)\mu\nu}-a^{(n-1)\mu\nu})\dt\\
&=-\int_0^T [a]^{(n)\mu\gamma}\p_\beta\p_t\eta^{(n)}_{\gamma}a^{(n)\beta\nu}+a^{(n-1)\mu\gamma}\p_\beta\p_t[\eta]^{(n-1)}_{\gamma}a^{(n)\beta\nu}+a^{(n-1)\mu\gamma}\p_\beta\p_t\eta^{(n-1)}_{\gamma}[a]^{(n)\beta\nu},
\end{align*} which gives
\begin{equation}
\|[a]^{(n)}(T)\|_{2}\lesssim\PP_0\int_0^T\|[a]^{(n)}(t)\|_2^2(\|[v]^{(n-1)}\|_3+\|[\psi]^{(n-1)}\|_3))\dt.
\end{equation}

Therefore we get
\begin{equation}\label{ag2}
\sup_{[0,T]}\|[a]^{(n)}\|_2^2\lesssim \PP_0T^2\left(\|[a]^{(n)},[a]^{(n-1)}\|_{L_t^{\infty}H^2}+\|[v]^{(n-1)},[v]^{(n-2)},[\eta]^{(n-2)}\|_{L_t^{\infty}H^3}^2\right),
\end{equation} and the bound for $[\eta]$ via $\p_t [\eta]^{(n)}=[v]^{(n)}+[\psi]^{(n)}$:
\begin{equation}\label{etag3}
\sup_{[0,T]}\|[\eta]^{(n)}\|_3^2\lesssim\PP_0T^2\left(\|[a]^{(n)}\|_{L_t^{\infty}H^2}^2+\|[v]^{(n)},[v]^{(n-1)},[\eta]^{(n-1)}\|_{L_t^{\infty}H^3}^2\right)
\end{equation}

Similar as in Lemma \ref{etapsi}, one can get estimates for the time derivatives of $[\eta]$ and $[\psi]$
\begin{align}
\|[\p_t\psi]^{(n)}\|_3^2\lesssim&\PP_0\left(\|[a]^{(n)}\|_{2}^2+\|[\p_t v]^{(n-1)}\|_{2}^2+\|[v]^{(n-1)},[\eta]^{(n-1)}\|_{3}^2\right)\\
\|[\p_t^2\psi]^{(n)}\|_2^2\lesssim&\PP_0\left(\|[a]^{(n)}\|_{2}^2+\|[\p_t^2 v]^{(n-1)}\|_{1}^2+\|[\p_t v]^{(n-1)}\|_{2}^2+\|[v]^{(n-1)},[\eta]^{(n-1)}\|_{3}^2\right)\\
\|[\p_t^3\psi]^{(n)}\|_1^2\lesssim&\PP_0\left(\|[a]^{(n)}\|_{2}^2+\|[\p_t^2 v]^{(n-1)}\|_{1}^2+\|[\p_t^3 v]^{(n-1)}\|_{0}^2+\|[\p_t v]^{(n-1)}\|_{2}^2+\|[v]^{(n-1)},[\eta]^{(n-1)}\|_{3}^2\right)\\
\|[\p_t\eta]^{(n)}\|_3^2\lesssim&\PP_0T^2\left(\|[a]^{(n)},[\p_tv]^{(n)},[\p_tv]^{(n-1)}\|_{L_t^{\infty}H^2}^2+\|[v]^{(n)},[v]^{(n-1)},[\eta]^{(n-1)}\|_{L_t^{\infty}H^3}^2\right)\\
\|[\p_t^2\eta]^{(n)}\|_2^2\lesssim&\PP_0T^2\left(\|[\p_t^2 v]^{(n),(n-1)}\|_{L_t^{\infty}H^1}^2+\|[a]^{(n)},[\p_tv]^{(n),(n-1)}\|_{L_t^{\infty}H^2}+\|[v]^{(n),(n-1)},[\eta]^{(n-1)}\|_{L_t^{\infty}H^3}^2\right)\\
\|[\p_t^3\eta]^{(n)}\|_1^2\lesssim&\PP_0\left(\|[\p_t^2 v]^{(n),(n-1)}\|_{L_t^{\infty}H^1}^2+\|[a]^{(n)},[\p_tv]^{(n),(n-1)}\|_{L_t^{\infty}H^2}^2+\|[v]^{(n),(n-1)},[\eta]^{(n-1)}\|_{L_t^{\infty}H^3}^2\right).\\
\|[\p_t^4\eta]^{(n)}\|_0^2\lesssim&\PP_0\bigg(\|[\p_t^3 v]^{(n,n-1)}\|_{L_t^{\infty}L_x^2}^2+\|[\p_t^2 v]^{(n,n-1)}\|_{L_t^{\infty}H^1}^2+\|[a]^{(n)},[\p_tv]^{(n,n-1)}\|_{L_t^{\infty}H^2}^2\\&+\|[v]^{(n,n-1)},[\eta]^{(n-1)}\|_{L_t^{\infty}H^3}^2\bigg).
\end{align}

\subsubsection*{Step 2: Magnetic field and Lorentz force}

 The first step is still the elliptic estimates of $[b]^{(n)}$. We show an example of $\|[b]^{(n)}\|_3$:
\[
\|[b]^{(n)}\|_3\lesssim P(\|\erk^{(n)}\|_2)\left(\|\Delta_{\ak^{(n)}}[b]^{(n)}\|_1+P(\|\TP\erk\|_2)\|[b]^{(n)}\|_2\right).
\] One can still use the heat equation \eqref{diffbn} to eliminate the Laplacian terms, but now we have two more higher order terms when ``[$\cdot$]" falls on $\divr$ or $\ark$. Such terms can be controlled directly by $\EE^{(n-1,n,n+1)}$ and thus by $\PP_0$. In specific, such terms are $$\left\|\dive_{\ak^{(n)}}\left(\nabla_{[\ak]^{(n)}}b^{(n)}\right)\right\|_1+\left\|\dive_{[\ak]^{(n)}}\left(\nabla_{\ark^{(n-1)}}b^{(n)}\right)\right\|_1.$$ The leading order part in these two terms can be written as $[\ak]^{(n)}$ times the top order derivatives (4-th order) of $b^{(n)}$ which has been controlled uniformly in $n$ in Proposition \ref{EEnn}. For example,
\[
\left\|\dive_{\ak^{(n)}}\left(\nabla_{[\ak]^{(n)}}b^{(n)}\right)\right\|_1\lesssim \left\|[\ak]^{(n)}\right\|_2\left\| b^{(n)}\right\|_4\times\cdots\lesssim\PP_0\left\|[\ak]^{(n)}\right\|_2
\] Therefore, the control of $[b]^{(n)}$ can be controlled in the same manner as before. Similar estimates hold for $\p_t[b]^{(n)}$. The control of $\|\p_t^2[b]\|_1$ and $\|\p_t^3[b]\|_0$ is reduced to the estimates of heat equation \eqref{diffbn}. The proof is the same as Section \ref{nonb} so we omit it here. 

The Lorentz force is controlled in a silimar way. For example, 
\[
\left\|\nabla_{\ak^{(n)}}[b]^{(n)}\right\|_3\lesssim P\left(\|\erk^{(n)}\|_3\right)\left(\left\|\Delta_{\ak^{(n)}}[b]^{(n)}\right\|_2+\left\|\TP\erk^{(n)}\right\|_3\left\|[b]^{(n)}\right\|_3\right)
\] We again use the heat equation \eqref{diffbn} to eliminate the Laplacian term, and the extra terms can be controlled in the same way as above. (Note that $\|\park b\|_4$ is controlled in Proposition \ref{EEnn}). Therefore,
\[
\left\|\nabla_{\ak^{(n)}}[b]^{(n)}\right\|_3\lesssim \kk^{-1}\PP_0\left\|[\ak]^{(n)}\right\|_2.
\] Similar estimates hold for the time derivatives of Lorentz force.

\subsubsection*{Step 3: Div-Curl estimates}

The control of $[v]^{(n)}$ and $[q]^{(n)}$ also follows the same way as Section \ref{nonvp}. The equation of $\curl_{\ak^{(n)}} [v]^{(n)}$ is 
\begin{equation}
\begin{aligned}
\rho_0\p_t\curl_{\Ak^{(n)}}[v]^{(n)}=&\curl_{\Ak^{(n)}}\left((b^{(n)}\cdot\nabla_{\Ak^{(n)}})[b]^{(n)}\right)+\left[\rho_0\p_t,\curl_{\Ak^{(n)}}\right][v]^{(n)}\\
&+\curl_{\Ak^{(n)}}\left([b]^{(n-1)}\cdot\nabla_{\Ark^{(n-1)}} b^{(n)}+b^{(n-1)}\cdot\nabla_{[\Ak]^{(n)}}b^{(n)}-\nabla_{[\Ak]^{(n)}}Q^{(n)}\right)
\end{aligned}
\end{equation}
The first two terms in the second line is controlled in the same way as before(just consider $\curl_{\Ak^{(n)}}$ as the covariant derivative $\nabla_{\Ak^{(n)}}$. Also $$\left\|\curl_{\Ak^{(n)}}\left(\nabla_{[\Ak]^{(n)}}Q^{(n)}\right)\right\|_2\lesssim \|[a]^{(n)}\|_{2}\|Q^{(n)}\|_4\PP_0\lesssim\|[a]^{(n)}\|_{2} \PP_0.$$ Therefore,
\[
\|\curl [v]^{(n)}\|_2^2\lesssim\epsilon\|[v]^{(n)}\|_3^2+P_{\kk}(\PP_0)T^2\sup_{[0,T]}[\EE]^{(n),(n-1)}(t) .
\] And similarly
\[
\|\curl [\p_t v]^{(n)}\|_1^2\lesssim\epsilon\|[\p_tv]^{(n)}\|_2^2+P_{\kk}(\PP_0)T^2\sup_{[0,T]}[\EE]^{(n),(n-1)}(t), 
\]
\[
\|\curl [\p_t^2 v]^{(n)}\|_0^2\lesssim\epsilon\|[\p_t^2v]^{(n)}\|_1^2+P_{\kk}(\PP_0)T^2\sup_{[0,T]}[\EE]^{(n),(n-1)}(t) .
\]

Invoking the divergence equation \eqref{diffqn}, we are still able to reduce that control to $\p_t^3 v$ and $\p_t^3 q$.

\subsubsection*{Step 4: Space-time tangential estimates}

Let $\dd^3=\TP^2\p_t,\TP\p_t^2,\p_t^3$. Using the same method as in Section \ref{tgtime}, we can derive the estimates
\begin{equation}\label{tgtimen}
\sum_{k=1}^3\left\|\TP^{3-k}\p_t^k [v]^{(n)}\right\|_{0}^2+\left\|\TP^{3-k}\p_t^k [q]^{(n)}\right\|_0^2\lesssim\int_0^T P([\EE]^{(n),(n-1),(n-2)}(t))\dt.
\end{equation}

\subsubsection*{Step 5: Spatial tangential estimates}

We adopt the same method as in Section \ref{tgspace}. For each $n$, we define the Alinhac good unknowns by 
\begin{equation}
\VV^{(n+1)}=\TP^3v^{(n+1)}-\TP^3\ek^{(n)}\cdot\nabla_{\ak^{(n)}}v^{(n+1)},~~\QQ^{(n+1)}=\TP^3Q^{(n+1)}-\TP^3\ek^{(n)}\cdot\nabla_{\ak^{(n)}}Q^{(n+1)}.
\end{equation} Their difference is denoted by 
\begin{align*}
[\VV]^{(n)}:=\VV^{(n+1)}-\VV^{(n)}, [\QQ]^{(n)}:=\QQ^{(n+1)}-\QQ^{(n)}.
\end{align*}

Similarly as in Section \ref{tgspace}, we can derive the analogue of \eqref{goodeq} as
\begin{equation}\label{goodnn}
\rho_0\p_t[\VV]^{(n)}+\nabla_{\ak^{(n)}}[\QQ]^{(n)}=-\nabla_{[\ak]^{(n)}}\QQ^{(n)}+\FF^{(n)},
\end{equation}subject to the boundary condition
\begin{equation}\label{bdrynn}
[\QQ]^{(n)}|_{\Gamma}=-\left(\TP^3\ek^{(n)}_{\beta}\ak^{(n)3\beta}\p_3[Q]^{(n)}+\TP^3[\ek]^{(n-1)}_{\beta}\ak^{(n)3\beta}\p_3Q^{(n)}+\TP^3\ek^{(n-1)}_{\beta}[\ak]^{(n)3\beta}\p_3Q^{(n)}\right),
\end{equation}and 
\begin{equation}\label{divnn}
\nabla_{\ak^{(n)}}\cdot[\VV]^{(n)}=-\nabla_{[\ak]^{(n)}}\cdot\VV^{(n)}+\GG^{(n)},
\end{equation}where
\begin{align*}
\FF^{(n)\alpha}=&[\rho_0,\TP^3]\p_t[v]^{(n)\alpha}+\TP^3\left((b^{(n)}\cdot\nabla_{\ak^{(n)}})[b]^{(n)}+(b^{(n)}\cdot\nabla_{[\ak]^{(n)}})b^{(n)}+([b]^{(n-1)}\cdot\nabla_{\ak^{(n-1)}})b^{(n)}\right)\\
&+\rho_0\p_t\left(\TP^3[\ek]^{(n-1)}_{\beta}\ak^{(n)\mu\beta}\p_{\mu}v^{(n+1)}_{\alpha}+\TP^3\ek^{(n-1)}_{\beta}[\ak]^{(n)\mu\beta}\p_{\mu}v^{(n+1)}_{\alpha}+\TP^3\ek^{(n-1)}_{\beta}\ak^{(n)\mu\beta}\p_{\mu}[v]^{(n)}_{\alpha}\right)\\
&+[\ak]^{(n)\mu\beta}\p_{\mu}(\ak^{(n)\gamma\alpha}\p_{\gamma}Q^{(n+1)})\TP^3\ek^{(n)}_{\beta}+\ak^{(n-1)\mu\beta}\p_{\mu}([\ak]^{(n)\gamma\alpha}\p_{\gamma}Q^{(n+1)})\TP^3\ek^{(n)}_{\beta} \\
&+\ak^{(n-1)\mu\beta}\p_{\mu}(\ak^{(n-1)\gamma\alpha}\p_{\gamma}[Q]^{(n)})\TP^3\ek^{(n)}_{\beta}+\ak^{(n-1)\mu\beta}\p_{\mu}([\ak]^{(n)\gamma\alpha}\p_{\gamma}h^{(n)})\TP^3[\ek]^{(n-1)}_{\beta}\\
&-\left[\TP^2,[\ak]^{(n)\mu\beta}\ak^{(n)\gamma\alpha}\TP\right]\p_{\gamma}\ek^{(n)}_{\beta}\p_{\mu}Q^{(n+1)}-\left[\TP^2,\ak^{(n-1)\mu\beta}[\ak]^{(n)\gamma\alpha}\TP\right]\p_{\gamma}\ek^{(n)}_{\beta}\p_{\mu}Q^{(n+1)}\\
&-\left[\TP^2,\ak^{(n-1)\mu\beta}\ak^{(n-1)\gamma\alpha}\TP\right]\p_{\gamma}[\ek]^{(n-1)}_{\beta}\p_{\mu}Q^{(n+1)}-\left[\TP^2,\ak^{(n-1)\mu\beta}\ak^{(n-1)\gamma\alpha}\TP\right]\p_{\gamma}\ek^{(n-1)}_{\beta}\p_{\mu}[Q]^{(n)}\\
&-\left[\TP^3,[\ak]^{(n)\mu\alpha},\p_{\mu}Q^{(n+1)}\right]-\left[\TP^3,\ak^{(n-1)\mu\alpha},\p_{\mu}[Q]^{(n)}\right]
\end{align*}
and 
\begin{align*}
\GG^{(n)}&=\TP^3(\dive_{\ak^{(n)}}[v]^{(n)}-\dive_{[\ak]^{(n)}}v^{(n)}) \\
&~~~~-\left[\TP^2,[\ak]^{(n)\mu\beta}\ak^{(n)\gamma\alpha}\TP\right]\p_{\gamma}\ek^{(n)}_{\beta}\p_{\mu}v^{(n+1)}_{\alpha}-\left[\TP^2,\ak^{(n-1)\mu\beta}[\ak]^{(n)\gamma\alpha}\TP\right]\p_{\gamma}\ek^{(n)}_{\beta}\p_{\mu}v^{(n+1)}_{\alpha} \\
&~~~~-\left[\TP^2,\ak^{(n-1)\mu\beta}\ak^{(n-1)\gamma\alpha}\TP\right]\p_{\gamma}[\ek]^{(n-1)}_{\beta}\p_{\mu}v^{(n+1)}_{\alpha} -\left[\TP^2,\ak^{(n-1)\mu\beta}\ak^{(n)\gamma\alpha}\TP\right]\p_{\gamma}\ek^{(n-1)}_{\beta}\p_{\mu}[v]^{(n)}_{\alpha} \\
&~~~~-\left[\TP^3,[\ak]^{(n)\mu\alpha},\p_{\mu}v^{(n+1)}_{\alpha}\right]-\left[\TP^3,\ak^{(n-1)\mu\alpha},\p_{\mu}[v]^{(n)}_{\alpha}\right] \\
&~~~~+[\ak]^{(n)\mu\beta}\p_{\mu}(\ak^{(n)\gamma\alpha}\p_{\gamma}v^{(n+1)}_{\alpha})\TP^3\ek^{(n)}_{\beta}+\ak^{(n-1)\mu\beta}\p_{\mu}([\ak]^{(n)\gamma\alpha}\p_{\gamma}v^{(n+1)}_{\alpha})\TP^3\ek^{(n)}_{\beta} \\
&~~~~+\ak^{(n-1)\mu\beta}\p_{\mu}(\ak^{(n-1)\gamma\alpha}\p_{\gamma}[v]^{(n)}_{\alpha})\TP^3\ek^{(n)}_{\beta}+\ak^{(n-1)\mu\beta}\p_{\mu}([\ak]^{(n)\gamma\alpha}\p_{\gamma}v^{(n)}_{\alpha})\TP^3[\ek]^{(n-1)}_{\beta}.
\end{align*}

Multiplying $[\VV]^{(n)}$ in \eqref{goodnn} and integrate by parts, we get
\begin{align*}
\frac{1}{2}\frac{d}{dt}\left\|\VV^{(n)}\right\|_0^2&=\io [\QQ]^{(n)}\left(\nabla_{\ak^{(n)}}\cdot[\VV]^{(n)}-\p_{\mu}\ak^{\mu\alpha}[\VV]^{(n)}_{\alpha}\right)\dy+\io (\FF^{(n)}-\nabla_{[\ak]^{(n)}}\QQ^{(n)})\cdot[\VV]^{(n)}\dy\\
&~~~~-\ig [\QQ]^{(n)}\ak^{(n)3\alpha}[\VV]^{(n)}_{\alpha}\dS.
\end{align*}

Similarly as in Section \ref{tgspace}, we are able to control the first three terms by using $[Q]=[q]+\frac{1}{2}[|b|^2]$
\[
-\frac{1}{2}\frac{d}{dt}\left\|\TP^4[q]^{(n)}\right\|_0^2+\PP_0 P([\EE]^{(n),(n-1)}(t)).
\]

For the boundary term, we integrate $\TP^{1/2}$ by parts as in \eqref{LI0} to get
\begin{align*}
&~~~~-\ig [\QQ]^{(n)}\ak^{(n)3\alpha}[\VV]^{(n)}_{\alpha}\dS \\
&=\ig \p_3 [Q]^{(n)}\ak^{(n)3\alpha}[\VV]^{(n)}_{\alpha}\left(\TP^3\ek^{(n)}_{\beta}\ak^{(n)3\beta}+\TP^3[\ek]^{(n-1)}_{\beta}\ak^{(n)3\beta}+\TP^3\ek^{(n-1)}_{\beta}[\ak]^{(n)3\beta}\right) \\
&\lesssim  |[\VV]^{(n)}|_{\dot{H}^{-0.5}}\left(\frac{1}{\kk}\PP_0\left|[\eta]^{(n-1)}\right|_{2.5}+\left\|[\ak]\right\|_2\right).
\end{align*}This finalizes the tangential estimates.

\subsubsection*{Step 6: Elliptic estimates of $[\p_tq]^{(n)}$}

Since $[q]^{(n)}$ vanishes on the boundary, we can still use Lemma \ref{GLL}, the elliptic estimates to reduce the spatial derivative to time derivative by replacing the Laplacian term with $\p_t^2$ plus source terms. We only list the wave equation of $[q]^{(n)}$ and omit the computation.
\begin{equation}\label{diffqnn}
\begin{aligned}
-\Jrk^{(n)}R'(q^{(n)})\p_t^2[q]^{(n)}-\Delta_{\ak^{(n)}}[q]^{(n)}=&\frac12\Delta_{\ak^{(n)}}[|b|^2]^{(n)}\\
-&\dive_{\ak^{(n)}}\left((b^{(n)}\cdot\nabla_{\ak^{(n)}})[b]^{(n)}\right)+\p_t\left(\Jrk^{(n)}R'(q^{(n)})\right)\p_t[q]^{(n)}\\
+&\dive_{\ak^{(n)}}\left(\nabla_{\Ak^{(n)}}Q^{(n)}-(b^{(n)}\cdot\nabla_{[\Ak]^{(n)}})b^{(n)}-([b]^{(n-1)}\cdot\nabla_{\Ak^{(n-1)}})b^{(n)}\right)\\
-&\left(\dive_{[\ak]^{(n)}}v^{(n)}+[\Jrk R'(q)]^{(n)}\p_t q^{(n)}\right).
\end{aligned}
\end{equation}

\subsubsection*{Step 7: Common control of heat and wave equations}

Differentiate $\p_t^3$ in \eqref{diffbn} and \eqref{diffqnn}, we are able to get similar estimates of $[W]^{(n+1)}$ and $H^{(n+1)}$ as in Section \ref{noncommon}. We omit the proof here.

Finally, we conclude that
\[
[\EE]^{(n+1)}\lesssim_\kk\PP_0 T^2\left( [\EE]^{(n)}+[\EE]^{(n-1)}\right),
\] where we pick $T_\kk$ suitably small such that the coefficient $\leq 1/4$. This ends the proof of Proposition \ref{diffconv} and Corollary \ref{lwpkk}.

\section{Local well-posedness of the original system}

As stated in Corollary \ref{lwpkk}, the local well-posedness of the nonlinear approximation system \eqref{nonlinearkk} is established in an $\kk$-dependent time interval $[0,T_\kk]$. Combining the uniform-in-$\kk$ nonlinear a priori estimates Proposition \ref{nonlinear1}, we know that there exists a $\kk$-\textbf{independent} time $T_1>0$, such that the local existence of the solution $(\eta,v,b,q)$ to the original equation \eqref{CRMHDL} holds in $[0,T_1]$ by letting $\kk\to 0$. It remains to prove the uniqueness of the solution. Let us recall the original equation first
\[
\begin{cases}
\p_t \eta=v &~~~\text{in } \Omega, \\
\rho_0J^{-1}\p_t v=(b\cdot\pa) b-\pa Q,~~Q=q+\frac{1}{2}|b|^2 &~~~\text{in } \Omega, \\
\frac{J R'(q)}{\rho_0}\p_tq+\dive_a v=0 &~~~\text{in } \Omega, \\
\p_t b-\Delta_a b=(b\cdot\pa) v-b\dive_{a}v, &~~~\text{in } \Omega, \\
\dive_a b=0 &~~~\text{in } \Omega, \\
q=0,~b=\mathbf{0},~-\p_3 Q|_{t=0}\geq c_0>0 &~~~\text{on } \Gamma, \\
(\eta,v, b,q)|_{\{t=0\}}=(\text{Id},v_0, b_0,q_0).
\end{cases}
\]

Suppose $(\eta^i,v^i.b^i,q^i),~i=1,2$ solves \eqref{CRMHDL} with the same initial data $(\text{Id},v_0, b_0,q_0)$. Then we consider the system of $([\eta],[v],[b],[q])$ by setting $[f]:=f^1-f^2$. Then we have

The flow map:
\[
\p_t [\eta]=[v].
\]

The momentum equation:
\[
\rho_0(J^1)^{-1}\p_t [v]=(b^1\cdot\nabla_{a^1}) [b]-\nabla_{a^1}[Q] -\rho_0[J^{-1}]\p_t v^2+(b^1\cdot\nabla_{[a]})b^2+[b]\cdot\nabla_{a^2}b^2-\nabla_{[a]}Q^2.
\]

The continuity equation:
\[
\frac{J^1 R'(q^1)}{\rho_0}\p_t[q]+\dive_{a^1} [v]=\left[\frac{J R'(q)}{\rho_0}\right]\p_t q^2-\dive_{[a]}v^2 ~~~\text{in } \Omega.
\]

The equation of magnetic field:

\begin{align*}
\p_t [b]-\Delta_{a^1} [b]=&(b^1\cdot\nabla_{a^1}) [v]-b\dive_{a}[v]\\
&+\dive_{a^1}\left(\nabla_{[a]}b^2\right)+\dive_{[a]}\left(\nabla_{a^2}b^2\right)\\
&+(b^1\cdot\nabla_{[a]}) v^2+([b]\cdot\nabla_{a^2})v^2\\
&-b^1\dive_{[a]}v^2-[b]\dive_{a^2}v^2,
\end{align*}
and
\[
\dive_{a^1}[b]=-\dive_{[a]}b^2.
\]

The boundary conditions:
\[
[q]=0,~[b]=\mathbf{0},~-\p_3Q_1\text{ and }-\p_3Q_2|_{t=0}\geq c_0>0,
\] 
and zero initial data.

Define the energy functional
\[
[\EE](T):=[E](T)+[H](T)+[W](T)+\left\|\p_t^{2-k}\left(\left(b^1\cdot\nabla_{a^1}\right)[b]\right)\right\|_{k}^2,
\]
where
\begin{align*}
[E](t)&:=\left\|[\eta]\right\|_2^2+\left|\ak^{3\alpha}\TP^2[\eta]_{\alpha}\right|_0^2+\sum_{k=0}^2 \left(\left\|\p_t^{2-k}[v]\right\|_{k}^2+\left\|\p_t^{2-k}[b]\right\|_{k}^2+\left\|\p_t^{2-k} [q]\right\|_k^2\right),\\
[H](T)&:=\int_0^T\io\left|\p_t^3 [b]\right|^2\dy\dt+\left\|\p_t^2 [b]\right\|_1^2,\\
[W](T)&:=\left\|\p_t^3 [q]\right\|_0^2+\left\|\p_t^2 [q]\right\|_1^2.
\end{align*}

The energy estimate of $[\EE]$ is almost the same as $\EE_\kk$ except that $[Q]$ no longer satisfies Taylor sign condition. So what we need to do is to investigate the boundary integral $$\ig [\QQ](a^{1})^{3\alpha} [\VV]_{\alpha}\dS,$$ where we define the Alinhac good unknowns 
\[
\VV^i=\TP^2v^i-\TP^2\eta^i\cdot\nabla_{a^i} v^i,~~\QQ^i=\TP^2Q^i-\TP^2\eta^i\cdot\nabla_{a^i} Q^i, 
\]and
\[
[\VV]:=\VV^1-\VV^2,~~[\QQ]:=\QQ^1-\QQ^2.
\]
 The boundary terms then becomes
\begin{align*}
\ig [\QQ](a^{1})^{3\alpha} [\VV]_{\alpha}&=-\ig\p_3[Q]\TP^2\eta^2_{\beta}(a^2)^{3\beta}(a^2)^{3\alpha}[\VV]_{\alpha}\dS-\ig \p_3 Q^1(\TP^2[\eta]_{\beta}(a^1)^{3\beta}+\TP^2\eta^2_{\beta}[a]^{3\beta})(a^1)^{3\alpha}[\VV]_{\alpha}\dS \\
&\lesssim -\frac{1}{2}\frac{d}{dt}\ig \p_3 Q^1|(a^{1})^{3\alpha}\TP^2[\eta]_{\alpha}|_0^2\dS\\
&~~~~-\ig\p_3 Q^1(a^1)^{3\gamma}\TP^2[\eta]_{\gamma}(\TP^2\eta^2_{\beta}[a]^{\mu\beta}\p_\mu v^1_{\alpha}-\TP^2\eta^2_{\beta}(a^2)^{\mu\beta}\p_\mu[v]_{\alpha})(a^1)^{3\alpha}\dS \\
&~~~~-\ig \p_3 Q^1(\TP^2[\eta]_{\beta}(a^1)^{3\beta}+\TP^2\eta^2_{\beta}[a]^{3\beta})(a^1)^{3\alpha}[\VV]_{\alpha}\dS\\
&\lesssim-\frac{c_0}{2}\frac{d}{dt}\ig |(a^{1})^{3\alpha}\TP^2[\eta]_{\alpha}|_0^2\dS+P(\text{initial data})P([\EE](t)).
\end{align*}
Here in the second step we use the precise formula of $[\VV]$, and in the third step we use Taylor sign condition for $Q^1$. Thus similarly we get
\[
\sup_{t\in[0,T_1]}[\EE](t)\leq \text{initial data}+\int_0^{T_0}P([\EE](t))\dt.
\] Since the initial data of the system of $([\eta],[v],[b],[q])$ is 0, we know $[\EE](t)=0$ for all $t\in[0,T_1]$.

Conclusively, the local well-posedness of the free-boundary compressible resistive MHD system \eqref{CMHD} is established in Lagrangian coordinates \eqref{CRMHDL} with Sobolev initial data.

\end{document}